\documentclass[reqno, 11pt]{amsart}
\usepackage[top=4cm, bottom=3cm, left=3.2cm, right=3.2cm]{geometry}
\usepackage{enumerate}
\usepackage{bbm}
\usepackage{amssymb,amsmath}
\usepackage{lineno}
\newtheorem{thm}{Theorem}
\newtheorem{cor}{Corollary}
\newtheorem{lem}{Lemma}
\newtheorem{prop}{Proposition}

\newtheorem{defn}{Definition}
\newtheorem{rem}{Remark}
\allowdisplaybreaks

\numberwithin{equation}{section} \numberwithin{lem}{section}
\numberwithin{thm}{section} \numberwithin{prop}{section}
\numberwithin{cor}{section} \numberwithin{rem}{section}
\title[Gagliardo-Nirenberg
inequality ]{On the best constant for Gagliardo-Nirenberg interpolation
inequalities}

\author{Jian-Guo Liu$^{\,2}$, Jinhuan Wang$^{\,1}$}
\thanks{The work of Jian-Guo Liu is partially supported by  KI-Net NSF RNMS grant No. 1107444, NSF
DMS  grant No. 1514826.}
\thanks{Jinhuan Wang is partially supported  by   and Program for Liaoning Excellent Talents in University (Grant No. LJQ2015041). }
\begin{document}
\maketitle
\begin{center}
{\footnotesize
1-School of Mathematics, Liaoning University, Shenyang, 110036, P. R. China \\
 email: wjh800415@163.com\\
 \smallskip
2-Department of Physics and Department of Mathematics, Duke University, \\Durham, NC 27708. USA.
email: jliu@phy.duke.edu
}
\end{center}
\maketitle
\date{}
\begin{abstract}
In this paper
we derive the best constant for the following Gagliardo-Nirenberg interpolation
inequality
 \begin{eqnarray*}
\|u\|_{L^{m+1}}\leq  C_{q,m,p} \|u\|^{1-\theta}_{L^{q+1}}\|\nabla u\|^{\theta}_{L^p},\quad \theta=\frac{pd(m-q)}{(m+1)[d(p-q-1)+p(q+1)]},
\end{eqnarray*}
where parameters $q,m,p$ respectively belong to the following two ranges:
\begin{enumerate}[(i)]
\item $p>d\geq 1$, $q\geq0$ and $m=\infty$. That shows $L^{\infty}$-type Gagliardo-Nirenberg interpolation
inequality.
\item $p>\max\{1,\frac{2d}{d+2}\}$, $0\leq q<\sigma-1$, and $q<m<\sigma$, where $\sigma$ is defined by
\begin{align*}
\sigma:=\left\{
\begin{array}{llll}\smallskip
\frac{(p-1)d+p }{d-p}  && \mbox{ if } p<d,\\
\infty  && \mbox{ if } p\geq d.
\end{array}
\right.
\end{align*}
That gives $L^{m}$-type Gagliardo-Nirenberg interpolation
inequality.
\end{enumerate}
The best constant $C_{q,m,p}$ is given by
\begin{eqnarray*}
C_{q,m,p}:=\theta^{-\frac{\theta}{p}}(1-\theta)^{\frac{\theta}{p}-\frac{1}{m+1}}M_c^{-\frac{\theta}{d}},\quad M_c:=\int_{\mathbb{R}^d}u_{c,m}^{q+1}\,dx,
\end{eqnarray*}
where $u_{c,m}$ is the unique radial non-increasing solution to a generalized Lane-Emden equation. The case of equality holds when $u=Au_{c,m}(\lambda(x-x_0))$ for any real numbers $A>0$, $\lambda >0$ and $x_{0}\in \mathbb{R}^d$.

In particular, for the case $m=+\infty$, the generalized Lane-Emden equation becomes a Thomas-Fermi type equation.
For $q=0,~m=\infty$ or $d=1$, $u_{c,m}$ are closed form solutions expressed in term of the incomplete Beta functions. Moreover, we show that $u_{c,m}\to u_{c,\infty}$ and $C_{q,m,p}\to C_{q,\infty,p}$ as $m\to +\infty$ for $d=1$.
\end{abstract}

\section{Introduction}

Research on functional inequalities is an important topic in the Functional Analysis. In some circumstances one is
interested in the exact value of the smallest admissible constant in some functional inequalities. Possible motivations for this can be described from the three respects: (i) it provides some
geometrical insights (a sharp version of some functional inequality is
equivalent to the Euclidean isoperimetric inequality \cite{cordero2004mass}); (ii) it is helpful for the computation of the ground-state energy in a physical model; (iii) it can be used to determined sharp conditions on initial data to distinguish between global existence
 and finite time blow-up for some partial differential equations with competition effects from some biological or physical systems, c.f. \cite{BCC,BDP,CC,CLW,CW,LW1,weinstein1983nonlinear,WBB}.

 In 1938, Sobolev \cite{Sobolev1938} proved that there is a constant $C_{d,p}>0$ such that for  $d\geq 3$, any function $u\in L^{\frac{2d}{d-2}}(\mathbb{R}^d)$ with $\nabla u\in L^2(\mathbb{R}^d)$, it holds that
\begin{eqnarray}\label{sobolevineq}
\|u\|_{L^{\frac{pd}{d-p}}}\leq C_{d,p}\|\nabla u\|_{L^p},\quad 1<p<d.
\end{eqnarray}
The best constant $C_{d,p}$ in (\ref{sobolevineq}) is established by Aubin and Talenti \cite{Aubin1976,Talenti1976}. As a variant of the Sobolev inequality, the following Gagliardo-Nirenberg (G-N) inequality have been proved in the literatures \cite{Gagliardo1958,Nirenberg1959} (See also \cite{maz2011sobolev})
 \begin{eqnarray}\label{sharpine}
\|u\|_{L^{m+1}}\leq  C_{q,m,p} \|u\|^{1-\theta}_{L^{q+1}}\|\nabla u\|^{\theta}_{L^p},\quad \theta=\frac{pd(m-q)}{(m+1)[dp+(p-d)(q+1)]},
\end{eqnarray}
 where $C_{q,m,p}>0$ is a constant, and the parameters $d,q,m$ and $p$ satisfy (\ref{range2}) or (\ref{qandm}) below.

 For some special parameters $d,q,m$ and $p$, the best constant of the G-N inequality has been studied widely in the literatures \cite{Nagy1941,nash1958continuity,weinstein1983nonlinear,carlen1993sharp,del2002best,del2003optimal,cordero2004mass,carlen2013stability}. These results are the landmark works, and are used at length to various branches of mathematics. However, the computation of sharp constants $C_{q,m,p}$ in the inequality (\ref{sharpine}) is still an open problem for general parameters.
The goal of this paper is to derive the best constant $C_{q,m,p}$ of the G-N inequality (\ref{sharpine}) for general indexes as described below.

We assume that the dimension $d\geq1$ and the parameters $q,m,p$ are given by the following two ranges (in these regions, it is easy to verify that $0<\theta< 1$ ):
\begin{enumerate}[(i)]
\item one range is
\begin{eqnarray}\label{range2}
p>d,\quad q\geq0 \mbox{ and } m=\infty.
\end{eqnarray}
We refer to this case as the $L^{\infty}$-type G-N
inequality.

\item the other range is
\begin{align}
 p>\max\{1,\frac{2d}{d+2}\},\quad 0\leq q<\sigma-1,\quad q<m<\sigma,\label{qandm}
\end{align}
where $\sigma$ is defined by
\begin{align}
\sigma:=\left\{
\begin{array}{llll}\smallskip
\frac{(p-1)d+p }{d-p}  && \mbox{ if } p<d,\\
\infty  && \mbox{ if } p\geq d.
\end{array}
\right.\label{pandsigma}
\end{align}
This case is referred to as the $L^{m}$-type G-N
inequality.

\end{enumerate}

We show that the best constant is given by the following form
\begin{eqnarray}\label{bestc}
C_{q,m,p}=\theta^{-\frac{\theta}{p}}(1-\theta)^{\frac{\theta}{p}-\frac{1}{m+1}}M_c^{-\frac{\theta}{d}},\quad M_c=\int_{\mathbb{R}^d}u_{c,m}^{q+1}\,dx.
\end{eqnarray}
 For $m=\infty$, $u_{c,\infty}$ is described by the following two cases:
\begin{enumerate}[(i)]
\item if $q<p-1$, $u_{c,\infty}$ is the unique decreasing solution of the following free boundary problem (FBP)
\begin{eqnarray}
&& (|u'|^{p-2}u')' + \frac{d-1}{r} |u'|^{p-2}u' = u^{q} \quad\mbox{ for }0<r<R,\label{iMinfFBVP1}\\
&& u(0)=1,\,\, u(R)=u'(R)=0,\label{iMinfFBVPbound}
 \end{eqnarray}
 and $u(r)=0$ for $r\geq R$.
 \item if $q\geq p-1$, $u_{c,\infty}$ is the unique positive decreasing solution to the following problem
\begin{eqnarray}
&& (|u'|^{p-2}u')' + \frac{d-1}{r} |u'|^{p-2}u' = u^{q} \quad\mbox{ for }0<r<\infty,\label{iMinfchachy1}\\
&& u(0)=1,\,\,\lim_{r\to \infty} u(r)=0.\label{iMinfchachybound}
 \end{eqnarray}
\end{enumerate}
For $m<\infty$, $u_{c,m}$ is described by the following two cases:
\begin{enumerate}[(i)]
\item if $q<p-1$, $u_{c,m}$ is the unique decreasing solution of the following free boundary problem
\begin{eqnarray}
&& (|u'|^{p-2}u')' + \frac{d-1}{r} |u'|^{p-2}u' + u^{m} = u^{q} \quad\mbox{ for }0<r<R,\label{FBVP1}\\
&& u'(0)=0,\,\, u(R)=u'(R)=0,\label{FBVPbound}
 \end{eqnarray}
 and $u(r)=0$ for $r\geq R$.
 \item if $q\geq p-1$, $u_{c,m}$ is the unique positive decreasing solution to the following problem
\begin{eqnarray}
&& (|u'|^{p-2}u')' + \frac{d-1}{r} |u'|^{p-2}u' + u^{m} = u^{q} \quad\mbox{ for }0<r<\infty,\label{chachy1}\\
&&  u'(0)=0,\,\,\lim_{r\to \infty} u(r)=0.\label{chachybound}
 \end{eqnarray}
 \end{enumerate}

Notice that when $q<p-1$, the solution to (\ref{iMinfFBVP1})-(\ref{iMinfFBVPbound}) is equivalent to the non-negative radial solution of
the Thomas-Fermi type equation (see Proposition \ref{prop})
\begin{eqnarray}
&&\Delta_p u+a\delta(x)=u^q,\quad \mbox{in } \mathcal{D}(B(0,R)),\label{iinftyfunction}\\
&&   a:=\|\nabla u\|^p_{L^p}+\|u\|^{q+1}_{L^{q+1}},\\
&&u(0)=1,~~u(x)=\frac{\partial u}{\partial \vec{n}}(x)=0, \mbox{ for }|x|=R,\label{iinftyfunctionboud}
\end{eqnarray}
where $\delta(x)$ is a delta function and $\vec{n}$ be the unit outward normal vector to $\partial B(0,R)$. When $q\geq p-1$,
 the solution to
 the problem (\ref{iMinfchachy1})-(\ref{iMinfchachybound}) is equivalent to the positive radial solution to
the Thomas-Fermi type equation
\begin{eqnarray}
&&\Delta_p u+a\delta(x)=u^q,\quad  \mbox{in } \mathcal{D}(\mathbb{R}^d),\label{inonwhole}\\
&&u(0)=1,~~\lim_{|x|\to \infty}u(x)=0.\label{inonwholeboud}
\end{eqnarray}

The indexes used in this paper are very complicated. We sometimes suppress some of these indexes when they are clear in the context. Now we give closed form solutions for some special cases.

For the case $p>d\geq 1$, $q=0$ and $m=\infty$, a closed form solution $u_{c,\infty}$ can be expressed in term of an incomplete Beta function (see Proposition \ref{prop2}). The  best constant $C_{0,\infty,p}$ is given by
$$C_{0,\infty,p}=\left(\frac{p-d}{pd}\right)^{\frac{d}{pd+p-d}}M_c^{-\frac{p}{pd+p-d}},~~ M_c=\int_{\mathbb{R}^d}u_{c,\infty}\,dx.$$

For $p>d=1$, $q\geq 0$ and $m=\infty$, the free boundary problem (\ref{iMinfFBVP1})-(\ref{iMinfFBVPbound}) and the problem (\ref{iMinfchachy1})-(\ref{iMinfchachybound}) have closed form solutions respectively (see Proposition \ref{prop42}) and use them to
 deduce the best constant
 \begin{eqnarray} \label{cqinftyp}
 C_{q,\infty,p}= \left(\frac{2p}{p+(p-1)(q+1)}\right)^{-\frac{p}{p+(p-1)(q+1)}}.
  \end{eqnarray}

When the parameters $q,m,p$ satisfy (\ref{qandm}), for the one-dimensional case, we obtain a closed form solution $u_{c,m}$, which can be expressed in term of an incomplete Beta function (see Theorem \ref{onedtheorem}).
 The constant $C_{q,m,p}$ for $d=1$ is given by
\begin{eqnarray}
 C_{q,m,p}=\left(\frac{2^p(p-1)^{1-p}\eta_1^{\frac{\eta_1}{m-q}}}{(m-q)^{2p-1}\eta_2^{\frac{\eta_2}{m-q}}} \mathcal{B}\left(\frac{\eta_2}{p(m-q)},\frac{2p-1}{p}\right)^p\right)^{-\ell},\label{beta}
\end{eqnarray}
where
\begin{eqnarray}
&&\ell=\frac{m-q}{(m+1)[(p-1)(q+1)+p]},\label{candl1}\\
&&\eta_1=(p-1)(m+1)+p,~~\eta_2=(p-1)(q+1)+p.
\end{eqnarray}
 These one-dimensional results were obtained by Sz. Nagy \cite{Nagy1941}. Here we use an alternative method to derive those best constants and express them by a different form (See Theorem \ref{onedtheorem}).
 It can be directly verified that the best constant given by (\ref{cqinftyp}) is a limit of the constant in (\ref{beta}) as $m\to \infty$ (See Theorem \ref{cor51}).

For a special case $m=1$, $q=0$, $p>1$ and $d\geq 1$, we derive
the best constant
\begin{eqnarray}\label{m1q0}
C_{0,1,p}=\theta^{-\frac{\theta}{p}}(1-\theta)^{\frac{\theta}{p}-\frac{1}{2}}R^{-\theta} \omega_d^{-\frac{\theta}{d}}, \quad \theta=\frac{pd}{2(pd+(p-d))},
\end{eqnarray}
where $\omega_d$ is the volume of the $d$-dimensional unit ball $B_1$ and $R$ is the first touch down point of $u_{c,m}$ (see Theorem \ref{th63}).

The results in our paper recover and extend all the results obtained in previous papers \cite{Nagy1941,nash1958continuity,weinstein1983nonlinear,carlen1993sharp,del2002best,del2003optimal,cordero2004mass,carlen2013stability}.
We will divide three cases below to describe these landmark results and show that their results are exactly same as our results for these special cases.

{\bf For the case $q<p-1$,} The results obtained in this paper agree with the classical results \cite{Nagy1941,nash1958continuity,carlen1993sharp,del2002best,del2003optimal} as we describe below:
 \begin{itemize}
 \item In 1941, Nagy \cite{Nagy1941} obtained the best constant $C_{q,m,p}$ for the one-dimensional case. In particular, $C_{0,3,2}=\left(\frac{4\pi^2}{9}\right)^{-1/4}$ for $p=2$, $q=0$, $m=3$ and $d=1$ in (\ref{beta}).
  \item In 1993, Carlen and Loss \cite{carlen1993sharp} derive the best constant
  $$
 C_{0,1,2}=\left(\frac{2}{d}\right)^{\frac{d}{2(d+2)}}\left(\frac{d+2}{d}\right)^{\frac{1}{2}}R^{-\frac{d}{d+2}}\omega_d^{-\frac{1}{d+2}},
  $$
 i.e., take $p=2$ in (\ref{m1q0}) (This case is known as Nash inequality \cite{nash1958continuity}). Here $R$ is the first touch zero point of $u_{c,m}$. In fact, it is directly verified that $\lambda=R^2$ and $u(r)=1-u_{c,m}(Rr)$ are respectively the first eigenvalue and eigenfunction of the Laplace operator with Neumann boundary in the ball $B_1$.
 Particularly, for $d=1$ we have $R=\pi$ by (\ref{Rvaluecom}). Hence we get $C_{0,1,2}=\left(\frac{16\pi^2}{27}\right)^{-1/6}$, which is consistent with a special case of Nagy inequality \cite{Nagy1941}. See also (\ref{beta}).
  \item For the case $d\geq 2$, $1<p<d$, $0<m<p-1$ and $q=\frac{pm-p+1}{p-1}$ in (\ref{bestc}) (for this case, it is checked $q<p-1$), the solution to (\ref{FBVP1})-(\ref{FBVPbound}) is given by the following Barenblatt profile
  $$
  u_{c,m}(r)=a\left(1-b r^{\frac{p}{p-1}}\right)^{-\frac{p-1}{m+1-p}}_+ \quad \mbox{ for some } a,~~ b>0,
  $$
  and the best constant is given by
  \begin{align}\label{barC1}
  \overline C_{q,m,p}  =&\left(\frac{p-m-1}{p}\right)^{\theta}\left(\frac{p(m+1)}{d(p-m-1)}\right)^{\frac{\theta}{p}}
  \left(\frac{p(m+1)}{\eta}\right)^{\frac{1-\theta}{q+1}-\frac{\theta}{d}}\nonumber\\
  &\cdot\left(\frac{p}{p-1}\right)^{\frac{\theta}{d}}\left(\frac{\Gamma(\frac{d}{2}+1)}{d\pi^{d/2}}\right)^{\frac{\theta}{d}}\left(\frac{\Gamma(\frac{pm}{p-m-1}+d\frac{p-1}{p}+1)}{\Gamma(1+\frac{pm}{p-m-1})\Gamma(d\frac{p-1}{p})}\right)^{\frac{\theta}{d}},
 \end{align}
  where
  \begin{eqnarray}\label{thetaanddelta}
  \theta=\frac{d(p-m-1)}{(m+1)(d(p-m-1)+pm)},\quad \eta=dp-(m+1)(d-p)>0.
  \end{eqnarray}
  These results were obtained in the celebrated works of Del Pino and Dolbeault \cite[Theorem 2]{del2002best} for $p=2$ and \cite[Theorem 3.1]{del2003optimal} for $1<p<d$.  In Lemma \ref{lmApp}, we will show that our best constant $C_{q,m,p}$ is exactly equal to the best constant $\overline C_{q,m,p}$ in (\ref{barC1}).
\end{itemize}

{\bf For the case $q=p-1$}, the best constant $C_{q,m,p}$ in (\ref{bestc}) is consistent with many classical results \cite{Nagy1941,weinstein1983nonlinear} as follows
 \begin{itemize}
 \item In 1941, Nagy \cite{Nagy1941} obtained the best constant $C_{q,m,p}=\left(\frac{\pi^2}{4}\right)^{-1/6}$ for $d=1$, $p=2$, $q=1$, and $m=5$ in (\ref{beta}).
  \item In 1983, Weinstein \cite{weinstein1983nonlinear} obtained the best constant $C_{q,m,p}$ and also expressed it by the solution $u_{c,m}$ of the elliptic equation (\ref{chachy1})-(\ref{chachybound}) for $d\geq 2$, $p=2$, $q=1$ and $0<m<\frac{d+2}{d-2}$ in (\ref{sharpine}).
\end{itemize}

{\bf For the case $q>p-1$}, the best constant $C_{q,m,p}$ is consistent with the following two classical results \cite{del2002best,del2003optimal,carlen2013stability}:
 \begin{itemize}
   \item  For the special case $d\geq 2$, $1<p<d$, $p-1<q<\frac{(p-1)d}{d-p}$, and $m=\frac{pq-p+1}{p-1}<\sigma$ in (\ref{sharpine}), Del Pino and Dolbeault, 2002 \cite[Theorem 1]{del2002best}, 2003 \cite[Theorem 1.2]{del2003optimal} proved that the best constant is given by
 \begin{align}\label{barC1po}
  \overline C_{q,m,p}  =&\left(\frac{q+1-p}{p}\right)^{\theta}\left(\frac{p(q+1)}{d(q+1-p)}\right)^{\frac{\theta}{p}}
   \left(\frac{\eta}{p(q+1)}\right)^{\frac{1}{m+1}}\left(\frac{p}{p-1}\right)^{\frac{\theta}{d}}\nonumber\\
  &\cdot\left(\frac{\Gamma\left(\frac{d}{2}+1\right)}{d\pi^{d/2}}\right)^{\frac{\theta}{d}}
  \left(\frac{\Gamma\left(\frac{(q+1)(p-1)}{q+1-p}\right)}{\Gamma\left(\frac{(q+1)(p-1)}{q+1-p}-\frac{d(p-1)}{p}\right)\Gamma(d\frac{p-1}{p})}\right)^{\frac{\theta}{d}},
 \end{align}
   where
   \begin{eqnarray}\label{thetadeltapo}
   \theta=\frac{(q+1-p)d}{q(dp-(d-p)(q+1))}, \quad \eta=dp-(d-p)(q+1)>0.
   \end{eqnarray}
  The equality of the inequality (\ref{sharpine}) is achieved in
  $$
  u_{c,m}(r)=a\left(1+b r^{\frac{p}{p-1}}\right)^{-\frac{p-1}{q+1-p}} \quad \mbox{for some }  a\in\mathbb{R},~~b>0.
  $$
 In Lemma \ref{lm5.4}, we will show that the best constant $C_{q,m,p}$ is exactly equal to the best constant $\overline C_{q,m,p}$ in (\ref{barC1po}).
   In particular, if $p=2$, $q=3$ and $m=5$, we can check when $a=\sqrt{3}, b=3$, $u_{c,m}(r)$ is the radial ground state solution of (\ref{chachy1})-(\ref{chachybound}). Moreover,
    a simple computation gives $C_{3,5,2}=\pi^{-1/6}$ in (\ref{sharpine}) for $d=2$, which agrees with the following celebrated formula from \cite[(1.8)]{carlen2013stability}
    $$
   \pi\int_{\mathbb{R}^2}f^6\,dx\leq \int_{\mathbb{R}^2}|\nabla f|^2\,dx\int_{\mathbb{R}^2}f^4\,dx.
   $$
 \item In 2004, Cordero-Erausquin, Nazaret and Villani \cite{cordero2004mass} showed that mass transportation methods provide an elementary and powerful approach for the study of certain functional inequalities with a geometric content, like sharp Sobolev or
Gagliardo-Nirenberg inequalities.
\end{itemize}

For simplicity, we will use the same function $u=u(x)$ and $u=u(r)$ to represent a radial solution with $u(x)=u(|x|)$ in this paper. It should be clear according the content of the text.

This paper is arranged as follows. In Section 2, we derive the Euler-Lagrange equations. For the case $q<p-1$, the Euler-Lagrange equation is a free boundary problem. In order to derive the free boundary problem, we construct an auxiliary functional and connect it with the contact angle by a Pohozaev type identity. For the case $q\geq p-1$, we provide the decay rate of positive solutions at the far field.

In Section 3, the direct method of calculus variation can not be used for the $L^{\infty}(\mathbb{R}^d)$ minimization problem (\ref{ibetamini}). Instead, we prove existence and uniqueness of the Euler-Lagrange equation whose solution is a minimizer. However, in this case the Euler-Lagrange equation is a Thomas-Fermi type equation, which contains a delta function. There is a singularity at the origin. We provide some delicate estimates and new techniques to overcome this difficulty.

In Section 4, we derive the best constant of $L^\infty$-type G-N inequality and show closed form solutions for some special parameters $d,p,q$ and $m$.

In Section 5, we derive the Euler-Lagrange equations, which are a class of generalized Lane-Emden equations whose unique solution achieves the equality for the $L^m$-type G-N inequality (\ref{sharpine}) in the parameter range (\ref{qandm}) and hence we provide the best constant of this inequality. Thanks to Strauss's lemma and compactness in $L^{m}(\mathbb{R}^d)$, we obtain existence of the minimizer by a direct method of calculus variation.

 In Sections 6 and 7, some closed form solutions to $L^m$-type G-N inequality are given. Our results recover all the known results with some extensions.

 In Section 8, we show the limit behavior in the best constant and the solution to the Euler-Lagrange equations as $m\to\infty$ for $d=1$, which indicates the connection between the  $L^m$-type G-N inequality (\ref{h sharp inequality1}) and  $L^{\infty}$-type G-N inequality (\ref{LinftyMlinftyq}).

 In Appendix A, we use a modified Strauss's inequality to prove existence of the minimizer in $L^{m+1}(\mathbb{R}^d)$ for the case $q\geq p-1$ (i.e. to prove Proposition \ref{h0minimizer}). In Appendix B, we provide some details to verify agreements with previous known formulas of the best constants for some special parameters.

\section{Euler-Lagrange equations for $L^{\infty}$-type G-N inequalities}

When $p>d$, the $L^{\infty}$-type G-N inequality \cite[pp.176, (2.3.50)]{maz2011sobolev} takes the following form
\begin{eqnarray}\label{linftyest}
\|u\|_{L^{\infty}}\leq C_{q,\infty,p} \|\nabla u\|^{\theta}_{L^p}\|u\|^{1-\theta}_{L^{q+1}},\quad 0<\theta=\frac{pd}{pd+(q+1)(p-d)}<1.
\end{eqnarray}
Our main propose is to derive the best constant $ C_{q,\infty,p}$ in terms of nonnegative solutions to a FBP if $q < p-1$ and in terms of positive solutions to an elliptic equation for $q \ge p -1$.

Following a standard method, the minimization problem is established in the solution space
$$
Y=\left\{u|~u\in L^{q+1}(\mathbb{R}^d),\quad \nabla u\in L^p(\mathbb{R}^d) \right\},
$$
and we know that there is a positive constant $\alpha$ such that
\begin{eqnarray}\label{ibetamini}
\alpha=\inf_{u\in Y}\frac{\|u\|^{1-\theta}_{L^{q+1}}\|\nabla u\|^{\theta}_{L^{p}}}{\|u\|_{L^{\infty}}}.
\end{eqnarray}
First, it can be directly checked that the minimization problem (\ref{ibetamini}) is equivalent to the following minimization problem
\begin{eqnarray}\label{ibetamini1}
\alpha=\inf_{u\in Y,~\|u\|_{L^{\infty}}=1}G(u),
\end{eqnarray}
where
\begin{eqnarray}\label{Gh}
G(u):=\|u\|^{1-\theta}_{L^{q+1}}\|\nabla u\|^{\theta}_{L^{p}}.
\end{eqnarray}
Thanks to the rearrangement technique (see \cite[Chapter 3]{lieb2001analysis})
\begin{eqnarray}\label{normequ}
\|h^*\|_{L^p}=\|h\|_{L^p}\,,\quad 1\leq p \leq \infty,
\end{eqnarray}
and the P\'{o}lya-Szeg\H{o} inequality \cite{PS,Brothers1988},
\begin{eqnarray}\label{derinormequ}
\|\nabla h^*\|_{L^p}\,\leq \|\nabla h\|_{L^p},\quad 1\le p\leq \infty,
\end{eqnarray}
similar to \cite[Lemma 2.1]{LW1}, we know that the minimization problem (\ref{ibetamini1}) is equivalent to the following minimization problem
 \begin{eqnarray}\label{iradminiexi}
\alpha=\inf_{u\in Y^*_{rad}}G(u),
\end{eqnarray}
where $Y^*_{rad}$ is a non-negative radial symmetric decreasing function space
$$
Y^*_{rad}=\left\{u(r)\geq 0\big|\,\lim_{r\to 0^+} u(r)=1,\, \,  u'(r)\leq 0~~ a.e.,\,\, \int_0^{\infty} \left(|u|^{q+1}+| u'|^{p}\right) r^{d-1}\, dr<\infty\right\}.
$$
For any $u\in Y^*_{rad}$, we take always $u(0)=1$.
\begin{prop}\label{iLcriticalequa} Assume that $\bar u(r)\in Y^*_{rad}$ is a critical point of $G(u)$, then there is $\lambda_0>0$ such that the rescaling function $u(r) =\bar u(\lambda_0 r)$ satisfies the following equation in the classical sense
\begin{eqnarray}
(| u'|^{p-2} u')' + \frac{d-1}{r} | u'|^{p-2} u' =  u^q,\,\,  0<r<R,\label{ixequation}
\end{eqnarray}
and the boundary conditions
\begin{eqnarray}
 \lim_{r\to 0^+} u(r)=1,\quad\lim_{r\to R^-} u(r)=0,\label{ixinitial}
\end{eqnarray}
for some $0< R\leq +\infty$.
\end{prop}
\begin{proof}
{\bf Step 1.} Re-scaling and admissible variation

Let $u_1(r):=\bar u(\lambda_1 r)$, $\lambda_1>0$ be a re-scaling parameter to be determined by (\ref{inftya1}). Noticing the scaling invariant of $G(u)$ for $u_1(r)=\bar u(\lambda_1 r)$, we have
 \begin{eqnarray}\label{Ginvarant}
G(u_1)=G(\bar u).
\end{eqnarray}
Hence if $\bar u(r)\in Y^*_{rad}$ is the critical point of $G(u)$, then $u_1$ is also a critical point ($\frac{\delta G(u_1)}{\delta u}=0$), and by choosing $\lambda_1$, it holds that
\begin{eqnarray}\label{inftya1}
\|u_1\|_{L^{q+1}}=1,\quad \|\nabla u_1\|^p_{L^p}=:a_1.
\end{eqnarray}

Since $u_1 \in Y^*_{rad}$, we have that $u_1(r)$ is continuous in $[0,\infty)$\footnote{If $u_1 \in Y^*_{rad}$, we know that for any $0<a<b<+\infty$, $u_1 \in W^{1,p}_{rad}([a,b])$. Hence $u_1(r)$ is continuous in $(0,\infty)$. The continuity at $x=0$ is given by $\lim_{r\to 0^+} u_1(r)=1$. }.
Denote
$$R_1:=\inf\{r>0|u_1(r)=0\}\in \mathbb{R}^{+}\cup \{+\infty\}.$$
For any admissible variation $\phi\in C_0^{\infty}(0,R_1)$ at $u_1$,
i.e., there is a $\varepsilon_0>0$ such that for any $0<|\varepsilon|<\varepsilon_0$, one has $u_1+\varepsilon \phi\in Y_{rad}^*$. Then from a direction computation and using \eqref{inftya1}, we have
$$
\frac{d}{d\varepsilon}\Big|_{\varepsilon=0}G(u_1+\varepsilon \phi)= S_d\int_0^{R_1}\theta a_1^{\frac{\theta}{p}-1}\left(r^{d-1}|u'_1|^{p-2} u'_1\right)\phi'(r)+(1-\theta)a_1^{\frac{\theta}{p}} u_1^qr^{d-1}\phi(r)\, dr = 0.
$$
This implies that $u_1$ satisfies the following generalized Lane-Emden equation in the distribution sense
\begin{eqnarray}
&&-\theta\left(r^{d-1}|u'_1|^{p-2} u'_1\right)'+(1-\theta)a_1 r^{d-1} u_1^q=0, \quad \mbox{ in }(0,R_1),\label{iequationh0}\\
&&\lim_{r\to 0^+} u_1(r)=1,\quad\lim_{r\to R_1^-} u_1(r)=0,
\end{eqnarray}
where $0< R_1\leq +\infty$.

{\bf Step 2. }Normalization

We re-scale the function $u_1$ as $u(r)=u_1(\lambda r)$, where $\lambda$ will be given in (\ref{ipara1}). From (\ref{Ginvarant}), we know that $ u$ is also a critical point of $G(u)$ in $Y^*_{rad}$. From (\ref{iequationh0}) we deduce that $u$ satisfies the following equation
\begin{eqnarray*}
-\theta\lambda^{-p}\left(r^{d-1}| u'|^{p-2} u'\right)'+(1-\theta)a_1 r^{d-1} u^q=0, \quad 0<r<\frac{R_1}{\lambda}=: R.
\end{eqnarray*}
Taking
\begin{eqnarray}\label{ipara1}
\lambda =\left(\frac{\theta}{(1-\theta) a_1}\right)^{1/p},~~\mbox{ i.e. }\theta\lambda^{-p} =(1-\theta)a_1,
\end{eqnarray}
we have that $u$ satisfies (\ref{ixequation})-(\ref{ixinitial}) in the distribution sense.

{\bf Step 3.} $u$ satisfies (\ref{ixequation})-(\ref{ixinitial}) in the classical sense.

 The propose of this step is to prove that solutions of (\ref{ixequation})-(\ref{ixinitial}) in the distribution sense are also classical solutions in any closed interval of $(0,R)$. We need only to show that the equation (\ref{ixequation}) is uniformly elliptic, i.e., to prove $|u'(r)|\geq C$ for some $C>0$ in any closed interval of $(0,R)$. That will be a consequence of the following claim.

 \noindent
 {\it Claim:} If there is $0<r_*\leq R$ such that $u'(r_*)=0$, then $u(r_*)=0$.

  \noindent
  {\bf Proof of Claim:} If not, then $u(r_*)>0$. Noticing that $\lim_{r\to 0^+} u(r)=1$ and $u'(r)\leq 0$, then for any fixed $r>r_*$, we have $\int_{r_*}^{r} s^{d-1}u^{q}(s)\, dr<\infty$. Moreover, by the continuity of $u(r)$ in $(0,\infty)$, we know that there is $r^*:~r_*<r^*<\infty$ such that if $r\in [r_*,r^*)$, $u(r)>0$. Hence integrating (\ref{ixequation}) from $r_*$ to $r^*$, we deduce
 \begin{eqnarray}\label{rstar}
(r^*)^{d-1}|u'(r^*)|^{p-2}u'(r^*)=\int_{r_*}^{r^*}s^{d-1}u^q(s)\,ds.
\end{eqnarray}
Notice that $(r^*)^{d-1}|u'(r^*)|^{p-2}u'(r^*)\leq 0$ due to $u'(r^*)\leq 0$ and the right side of the above equation is positive. That is a contradiction. This completes the proof of this claim.

Since $R:=\inf\{r>0|u(r)=0\}\in \mathbb{R}^{+}\cup \{+\infty\}$, then we have $r_*= R$ by the claim. So, we have that $|u'(r)|>0$ in $(0,R)$. Therefore, from regularity of solutions to the elliptic equation, we know that any weak solution of the above equation (\ref{iequationh0}) in $Y_{rad}^*$ is also a classical solution in any closed interval of $(0,R)$. Hence the equation (\ref{ixequation}) holds in the classical sense.

\end{proof}

In the following, we show the three important results: (i) For the case $q<p-1$, all critical points of $G(u)$ in $Y^*_{rad}$ satisfy the free boundary problem (\ref{iMinfFBVP1})-(\ref{iMinfFBVPbound}). Here we need to derive a Pohozaev type identity and use it to prove that the contact angle is zero. (ii) For the case $q\geq p-1$, we derive the Euler-Lagrange equations (\ref{iMinfchachy1})-(\ref{iMinfchachybound}) for the critical points of $G(u)$ in $Y^*_{rad}$. Solutions to the Euler-Lagrange equations (\ref{iMinfchachy1})-(\ref{iMinfchachybound}) are positive and have decay properties at the infinity. (iii) We show that the solution to the Euler-Lagrange equations in $Y^*_{rad}$ is a critical point of $G(u)$  up to a re-scaling.

\subsection{Case $q<p-1$: Compact support, zero contact angle and free boundary problem}
In this subsection, we prove that all critical points of $G(u)$ satisfy the free boundary problem (\ref{iMinfFBVP1})-(\ref{iMinfFBVPbound}). First, we show that a solution to (\ref{ixequation})-(\ref{ixinitial}) has a compact support in $[0,+\infty)$ (see Proposition \ref{compact support}). Next we show that the solution has a zero-contact-angle at the boundary of the compact support (see Lemma \ref{lem21} and Lemma \ref{ieneangle}). Finally we use the zero-contact-angle result to derive a complete free boundary problem (\ref{iMinfFBVP1})-(\ref{iMinfFBVPbound}) (see Proposition \ref{icor}).

\begin{prop}\label{compact support}Assume that $\bar u(r)\in Y^*_{rad}$ is a critical point of $G(u)$. Let $u(r) =\bar u(\lambda_0 r)$ for some $\lambda_0>0$ satisfy (\ref{ixequation})-(\ref{ixinitial}). If $p>1$, $0\leq q<p-1$, then there is $ R\in (0,\infty)$ such that $ u( R)=0$.
\end{prop}
\begin{proof}
For a radial decreasing non-negative function $u\in Y^*_{rad}$, there only exist two cases: (i) there exists a finite $R$ such that $ u(R)=0$; (ii)  $ u(r) >0$ for all $r>0$, and hence $u(r)\to 0$, $u'(r)\to 0$ as $r\to \infty$.

Inspired by the work \cite[Theorem 5.1]{Peletier1986}, using a contradiction method, we show that the second case can not happen.
Indeed, if (ii) holds, then $u>0$ is a solution to the following problem
 \begin{eqnarray}\label{iequa2}
&& (|u'|^{p-2}u')' + \frac{d-1}{r} |u'|^{p-2}u'  = u^{q},\quad 0<r<\infty,\label{iequa2}\\
&&\lim_{r\to 0^+}  u(r)=1,\quad \lim_{r\to \infty}u(r)=0.\label{iequa2bound}
\end{eqnarray}
Multiplying $u'$ to the both sides of (\ref{iequa2}), we get
\begin{eqnarray}\label{ienergy}
 \frac{d}{dr}\left(\frac{p-1}{p}|u'(r)|^p-\frac{u^{q+1}(r)}{q+1}\right)+\frac{d-1}{r}|u'(r)|^p=0.
\end{eqnarray}
 Integrating (\ref{ienergy}) from $r$ to $+\infty$ and utilizing the fact $u(r)\to 0$ and $u'(r)\to 0$ as $r\to \infty$, we have
 \begin{eqnarray}\label{iHenergypd}
 \frac{p-1}{p}|u'(r)|^p-\frac{u^{q+1}(r)}{q+1}=\int_r^{+\infty}\frac{d-1}{s}|u'(s)|^p\,ds.
\end{eqnarray}
Hence from (\ref{iHenergypd}), it holds that
 \begin{eqnarray}\label{iuderi}
 -u'(r)\geq \left(\frac{p}{(p-1)(q+1)}\right)^{1/p}u^{\frac{q+1}{p}}(r).
\end{eqnarray}
Using the method of separation of variable for (\ref{iuderi}) and integrating the result inequality from $0$ to $r$, $r\in (0,+\infty)$, we obtain
 \begin{eqnarray*}
 \int_{u(r)}^{1}u^{-\frac{q+1}{p}}\,du\geq \left(\frac{p}{(p-1)(q+1)}\right)^{1/p} r\quad\mbox{for all }r>0,
\end{eqnarray*}
which gives
\begin{eqnarray}\label{ircontrol1}
 r\leq \left(\frac{p}{(p-1)(q+1)}\right)^{-1/p}\frac{p}{p-q-1}\left(1-u^{\frac{p-q-1}{p}}(r)\right).
\end{eqnarray}
Noticing that $p>1$, $q<p-1$, by (\ref{ircontrol1}) we have
\begin{eqnarray}\label{ircontrol}
r\leq \left(\frac{p}{(p-1)(q+1)}\right)^{-1/p}\frac{p}{p-q-1}.
\end{eqnarray}
Taking $r\to+\infty$, we obtain a contradiction from (\ref{ircontrol}). Hence the second case can not happen, i.e., there exists a finite $R$ such that $u( R)=0$.

\end{proof}

Now we show that solutions to (\ref{ixequation})-(\ref{ixinitial}) have a zero contact angle at the boundary of the compact support by constructing an auxiliary energy functional
\begin{eqnarray}\label{isteadyfree0}
\mathcal{G}( u):=\frac{p-d}{p}\int_{\mathbb{R}^d} |\nabla u|^p\,dx- \frac{d}{q+1} \int_{\mathbb{R}^d} u^{q+1}\,dx.
\end{eqnarray}

\begin{lem}\label{lem21}
Let $\bar u(r)\in Y^*_{rad}$ be a critical point of $G(u)$. Then there is $\lambda_0>0$ such that the re-scaling function $ u(r)=\bar u(\lambda_0 r)$ is a zero point of the energy functional $\mathcal{G}( u)$ defined in (\ref{isteadyfree0}), i.e.,
\begin{eqnarray}\label{Gre=0}
\mathcal{G}(u)=0.
\end{eqnarray}

\end{lem}
\begin{proof}
From (\ref{inftya1}), the re-scaling function $u_1(r)=\bar u(\lambda_1 r)$, $\lambda_1>0$ satisfies
\begin{eqnarray}\label{re}
1=\int_{\mathbb{R}^d}u_1^{q+1}\,dx,\quad a_1=\int_{\mathbb{R}^d}|\nabla u_1|^p\,dx.
\end{eqnarray}

Let $ u(r)=u_1(\lambda r)$, where $\lambda$ is given by (\ref{ipara1}). Thus from (\ref{re}) we deduce
$$
\int_{\mathbb{R}^d} u^{q+1}\,dy=\frac{1}{\lambda^d},\quad \int_{\mathbb{R}^d}|\nabla u|^p\,dy=a_1\lambda ^{p-d}.
$$
Hence
\begin{align*}
\mathcal{G}( u)&=\frac{p-d}{p}\int_{\mathbb{R}^d} |\nabla u|^p\,dx- \frac{d}{q+1} \int_{\mathbb{R}^d} u^{q+1}\,dx\\
&=\frac{p-d}{p}a_1\lambda ^{p-d}-\frac{d}{q+1} \lambda^{-d}.
\end{align*}
Using (\ref{ipara1}) and the definition (\ref{Gh}) of $\theta$, we have
\begin{align*}
\mathcal{G}( u)=\frac{p-d}{p}\left(\frac{\theta}{1-\theta}\right)\lambda^{-p}\lambda ^{p-d}-\frac{d}{q+1} \lambda^{-d}=0,
\end{align*}
 i.e., (\ref{Gre=0}) holds.

\end{proof}

\begin{lem}\label{ieneangle}
Let $u(r)$ be a solution to the problem (\ref{ixequation})-(\ref{ixinitial}) in $Y_{rad}^*$. Assume that $u(r)$ has a touchdown point $R$ {\rm(}i.e. $u(R)=0${\rm )}. Then
the following relation between the energy functional defined by (\ref{isteadyfree0}) and the contact angle holds
\begin{eqnarray}\label{ifrvsangle}
\mathcal{G}( u) = \frac{(p-1) S_d}{p} \lim_{r\to R^-}r^d|u'(r)|^p,
\end{eqnarray}
where $S_d$ is the surface area of $d$-dimensional unit ball.
\end{lem}

\begin{proof}
Now we prove (\ref{ifrvsangle}) by using a similar idea to the proof of the Pohozaev identity. Introduce the energy function
\begin{eqnarray}\label{iHenergy}
H(r):= \frac{p-1}{p}|u'(r)|^p-\frac{u^{q+1}(r)}{q+1}.
\end{eqnarray}
Using (\ref{ienergy}), we have the following energy-dissipation relation
\begin{eqnarray}\label{iH}
\frac{dH(r)}{dr} + \frac{d-1}{r} |u'(r)|^{p} = 0.
\end{eqnarray}
Multiplying $r^d$ to (\ref{iH}) and integrating the result equation from $r$ to $R_0$, for any fixed $0<R_0<R$, we obtain that
 $$
 R_0^d H(R_0) -  r^d H(r) -d \int_r^{R_0} s^{d-1} H(s) \,ds+ (d-1) \int_r^{R_0} |u'(s)|^p s^{d-1}\,ds = 0.
 $$
By (\ref{iHenergy}), the above equation can be written as the following form
\begin{align}\label{i0est1}
r^d\left(\frac{p-1}{p}|u'(r)|^p-\frac{u^{q+1}(r)}{q+1}\right)=&R_0^d H(R_0)-\frac{p-d}{p} \int_r^{R_0} s^{d-1}| u'(s)|^p\,ds \nonumber\\
&+ \frac{d}{q+1} \int_r^{R_0} s^{d-1} u^{q+1}(s)\,ds.
\end{align}
Since $u(r)\in Y_{rad}^*$, we have that the limit of the right side of (\ref{i0est1}) exists as $r\to 0^+$. Hence taking the limit for the both side of (\ref{i0est1}), we have
\begin{align}\label{i0est}
\frac{p-1}{p}\lim_{r\to 0^+} r^d|u'(r)|^p=&R_0^d H(R_0)-\frac{p-d}{p} \int_0^{R_0} s^{d-1}| u'(s)|^p\,ds \nonumber\\
&+ \frac{d}{q+1} \int_0^{R_0} s^{d-1} u^{q+1}(s)\,ds.
\end{align}
Notice that $r^{d}|u'(r)|^{p}\geq 0$. Hence there is a constant $C\geq 0$ such that
$$
\lim_{r\to 0^+}r^{d}|u'(r)|^{p}=C.
$$

Now we claim $C=0$. If $C>0$, then there is $\delta>0$ such that
$$
r^{d}|u'(r)|^{p}\geq \frac{C}{2}\quad\mbox{ for } 0<r\leq \delta,
$$
which means
$$
r^{d-1}|u'(r)|^{p}\geq \frac{C}{2} r^{-1}\quad\mbox{ for } 0<r\leq \delta.
$$
Integrating above inequality from $0$ to $\delta$, we deduce
\begin{eqnarray*}
\infty>\int_0^{\delta} s^{d-1}| u'(s)|^p\,ds\geq \frac{C}{2}\int_0^{\delta}r^{-1}\,dr=+\infty.
\end{eqnarray*}
This is a contradiction. Hence it holds that
\begin{eqnarray}\label{i0est2}
\lim_{r\to 0^+}r^{d}|u'(r)|^{p}=0.
\end{eqnarray}

Therefore, using (\ref{iHenergy}), (\ref{i0est2}) and taking the limit for (\ref{i0est}) as $R_0\to R^-$, we have
\begin{align*}
\frac{p-1}{p} \lim_{R_0\to R^{-}} R_0^d |u'(R_0)|^p=&\lim_{R_0\to R^{-}} R_0^d H(R_0)\\
=&\frac{p-d}{p} \int_0^R r^{d-1}| u'(r)|^p\,dr - \frac{d}{q+1} \int_0^R r^{d-1} u^{q+1}(r)\,dr\\
=& \frac{p-d}{p} \frac{1}{S_d}\int_{B_R(0)} |\nabla u|^p\,dx- \frac{d}{q+1} \frac{1}{S_d}\int_{B_R(0)} u^{q+1}\,dx=\frac{1}{S_d}\mathcal{G}(u).
\end{align*}
Hence (\ref{ifrvsangle}) holds.
\end{proof}

Finally, we show that all critical points of $G(u)$ satisfy the free boundary problem (\ref{iMinfFBVP1})-(\ref{iMinfFBVPbound}) up to a re-scaling.

\begin{prop}\label{icor}Assume $p>1$, $0\leq q<p-1$.
Let $\bar u(r)\in Y^*_{rad}$ be a critical point of $G(u)$. Then there is $\lambda_0>0$ such that the re-scaling function $ u(r)= \bar u(\lambda_0 r)$ satisfies the free boundary problem (\ref{iMinfFBVP1})-(\ref{iMinfFBVPbound}).
\end{prop}
\begin{proof}
As a direct consequence of (\ref{ifrvsangle}) and $\mathcal{G}( u)=0$, one knows that $ u'( R)=0$. In the other words, the contact angle is zero. This case is the so-called complete wetting regime in Young's law \cite{giacomelli2008smooth}.
\end{proof}

\subsection{Case $q\geq p-1$: Positivity and decay property} In this subsection, we show that solutions to (\ref{ixequation}) and (\ref{ixinitial}) are positive (see Proposition \ref{iqlargep-1}). And decay properties of solutions to the problem (\ref{iMinfchachy1})-(\ref{iMinfchachybound}) are proved in Proposition \ref{idecay}.
\begin{prop}\label{iqlargep-1} Assume $p>1$, $ q\geq p-1$. Let $u(r)$ be a solution of (\ref{ixequation})-(\ref{ixinitial}). Then $u(r)>0$ for any $0<r<\infty$.
\end{prop}
\begin{proof}
 Now we only need to prove that $ R=\infty$ for $p>1$, $q\geq p-1$. If not, $R<\infty$. By Proposition \ref{icor}, we have $u'(R)=0$. Multiplying $r^{d-1}$ to the equation (\ref{ixequation}) and using $u'(r)\leq 0$, we have
 $$( r^{d-1} |u'(r)|^{p-1})' + r^{d-1} u^q(r)  = 0,~~ 0<r<R.$$
 We extend the function $u$ to $u=0$ for $r\geq R$.
 Let $\Omega_{\varepsilon}:=\mathbb{R}^d\setminus \overline{B_{\varepsilon}(0)}, \forall \varepsilon>0$ be a domain without the origin.
For any $\phi(x) =\phi(|x|)  \ge 0$ and $\phi\in C^{\infty}_c (\Omega_{\varepsilon})$, it holds that
$$\int_{\varepsilon}^{\infty} (- \phi'(r) r^{d-1} |u'(r)|^{p-1} +\phi(r) r^{d-1} u^q(r)  ) dr =  0.$$
And noticing $\nabla u\in L^p(\mathbb{R}^d)$, we have
$$ \int_{\Omega_{\varepsilon}} ( \nabla\phi\cdot \nabla u |\nabla u|^{p-2}  +\phi u^q) dx = 0.$$
Hence we have for any $\bar R>R>0$
 $$
 \Delta_{p}  u = u^{q} \quad\mbox{ in } \mathcal{D}( B_{\bar R}(0)\setminus \overline{B_{\varepsilon}(0)}).
 $$
 Moreover, by Step 3 in the proof of Proposition \ref{iLcriticalequa} and $u'(R)=0$, we know $u\in C^1( B_{\bar R}(0)\setminus \overline{B_{\varepsilon}(0)})$.
   Positivity of $u$ in $ B_{\bar R}(0)\setminus \overline{B_{\varepsilon}(0)}$ is a direct consequence of the Strong Maximum Principle given by Pucci and Serrin \cite[Thoerem 1.1.1]{pucci2007maximum}.
To prove this positivity, we only need to verify the necessary and sufficient condition for the Strong Maximum Principle: $f(s)>0$ for $s\in (0,\delta)$ and $\int_{0^+}\frac{ds}{H^{-1}(F(s))}=\infty$ (in the same notations as that in \cite{pucci2007maximum}, $f(s)=s^q$, $F(s)=\frac{s^{q+1}}{q+1}$, $H(s)=\frac{p-1}{p}s^p$). While the condition  $\int_{0^+}\frac{ds}{H^{-1}(F(s))}=\infty$ holds if and only if $q\geq p-1$.

 Therefore, positivity in $ B_{\bar R}(0)\setminus \overline{B_{\varepsilon}(0)}$ is a contradiction with  $u\equiv0$ in $B_{\bar R}(0)\setminus \overline{B_{R}(0)}$.
\end{proof}

\begin{prop}\label{idecay}
Let $u(r)$ be a solution of the problem (\ref{iMinfchachy1})-(\ref{iMinfchachybound}). Then $u(r)$ satisfies the following decay estimate
\begin{eqnarray}\label{idecayrate}
\lim_{r\to \infty} r^{d-1}|u'(r)|^{p-1}=0.
\end{eqnarray}
Moreover, $u(r)$ satisfies the following decay rates
\begin{itemize}
\item [(i)] for $q>p-1$,
it holds that
\begin{eqnarray}\label{ipdaq1oinftypro22}
u(r)+r|u'(r)|\leq  C_{p,q}r^{-\frac{p}{q+1-p}} \quad \mbox{for } r> 0;
\end{eqnarray}
\item [(ii)] for $q=p-1$, it holds that
\begin{eqnarray}\label{ipq1oinftypro22}
u(r)+|u'(r)|\leq C_p e^{-(p-1)^{-\frac{1}{p}}r}\quad \mbox{for } r> 0.
\end{eqnarray}
 \end{itemize}
\end{prop}
\begin{proof}
{\bf Step 1.} We prove the decay estimate (\ref{idecayrate}).
Since the function $u(r)$ satisfies the equation (\ref{iMinfchachy1}), hence we have
$$
(r^{d-1}|u'(r)|^{p-1})'=-u^{q} r^{d-1}<0~~\mbox{ for any } r>0.
$$
So, $r^{d-1}|u'(r)|^{p-1}$ is decreasing in $r$. Notice that $r^{d-1}|u'(r)|^{p-1}\geq 0$. Hence there is a constant $C\geq 0$ such that
$$
\lim_{r\to \infty}r^{d-1}|u'(r)|^{p-1}=C.
$$
Now we claim $C=0$. If $C>0$, we have
$$
r^{d-1}|u'(r)|^{p-1}\geq C~~ \mbox{ for any } r>0,
$$
which means
$$
-u'(r)\geq C r^{-\frac{d-1}{p-1}}~~\mbox{ for any } r> 0.
$$
Integrating above inequality from $r$ to $\infty$ for any $r>0$ and using the fact $\lim_{r\to \infty} u(r)=0$, we obtain
\begin{eqnarray}
u(r)\geq C\frac{p-1}{p-d}r^{1-\frac{d-1}{p-1}}\Big|_{r}^{\infty}=+\infty.
\end{eqnarray}
This is a contradiction, i.e., (\ref{idecayrate}) holds.

{\bf Step 2.} The decay rate of $u$.
From (\ref{iHenergy}) and (\ref{iH}), we have
\begin{eqnarray}\label{iupie}
\frac{d}{dr}|u'(r)|^p -\frac{p}{(p-1)(q+1)}\frac{d}{dr}u^{q+1}(r)+ \frac{p(d-1)}{p-1}\frac{1}{r}|u'(r)|^{p} = 0.
\end{eqnarray}
Integrating (\ref{iupie}) from $r$ to $\infty$ and using $u(r),~u'(r)\to 0$ as $r\to \infty$, we obtain
$$
|u'(r)|^p-\frac{p}{(p-1)(q+1)}u^{q+1}(r)=\frac{p(d-1)}{p-1}\int_r^{\infty} \frac{1}{s} |u'(s)|^{p}\,ds.
$$
Thus
$$
|u'(r)|^p-\frac{p}{p-1}\frac{u^{q+1}(r)}{q+1}\geq 0\quad \mbox{ for } r>0,
$$
i.e.,
$$
-u'(r)\geq \left(\frac{p}{(p-1)(q+1)}\right)^{\frac{1}{p}}u^{\frac{q+1}{p}}(r).
$$
Using the method of separation of variable for above formula and integrating the result inequality from $r$ to $+\infty$, $r\in(0,\infty)$, we deduce
\begin{eqnarray}
&&u(r)\leq  C_{p,q}r^{-\frac{p}{q+1-p}} \quad \mbox{for } r> 0, ~~q>p-1;\label{iudecay}\\
&&u(r)\leq e^{-(p-1)^{-\frac{1}{p}}r}\quad \mbox{for } r> 0, ~~q=p-1.\label{iq=udecay}
\end{eqnarray}

{\bf Step 3.} The refined decay rate of $u'(r)$.
Again multiplying $r^k$, $k=\frac{p(d-1)}{p-1}$ in both sides of (\ref{iupie}), it holds that
\begin{eqnarray*}
\frac{d}{dr}(r^k |u'(r)|^p) -\frac{p}{(p-1)(q+1)}r^k\frac{d}{dr}u^{q+1}(r) = 0.
\end{eqnarray*}
The decay rate in (\ref{idecayrate}) implies $\lim_{r\to \infty} r^k |u'(r)|^p =0$. Hence integrating the above equality from $r$ to $\infty$ gives
$$
r^k |u'(r)|^p+\frac{p}{(p-1)(q+1)}\int_r^{\infty} s^k\frac{d}{ds}u^{q+1}(s)\,ds=0.
$$
Using (\ref{iudecay}), we can directly check that $\lim_{r\to \infty} r^k u^{q+1}(r) =0$ due to $k-\frac{p(q+1)}{q+1-p}<0$ for $p>d$. Hence using
the integration by parts, we have
$$
r^k |u'(r)|^p=\frac{p}{(p-1)(q+1)}r^ku^{q+1}(r)+\frac{p}{(p-1)(q+1)}k\int_r^{\infty} s^{k-1}u^{q+1}(s)\,ds.
$$
Thus we get
\begin{eqnarray}\label{iderqua1}
 |u'(r)|^p=\frac{p}{(p-1)(q+1)}u^{q+1}(r)+\frac{p}{(p-1)(q+1)}kr^{-k}\int_r^{\infty} s^{k-1}u^{q+1}(s)\,ds.
\end{eqnarray}
Using (\ref{iudecay}), (\ref{iq=udecay}) and (\ref{iderqua1}), a direct computation gives that for any $r>0$
\begin{eqnarray}
 &&r|u'(r)| \leq C_{p,q} r^{-\frac{p}{q+1-p}}\quad \mbox{ for } q>p-1;\label{iudercon}\\
 && |u'(r)|\leq C_{p} e^{-(p-1)^{-\frac{1}{p}}r}\quad \mbox{ for } q=p-1.\label{iq=udercon}
\end{eqnarray}
Hence (\ref{iudecay}) and (\ref{iudercon}) give (\ref{ipdaq1oinftypro22}). Formulas (\ref{iq=udecay}) and (\ref{iq=udercon}) imply (\ref{ipq1oinftypro22}).
\end{proof}

\subsection{Solutions to Euler Lagrange equations are critical points of $G(u)$}
Since the zero contact angle in the free boundary condition (\ref{iMinfFBVPbound}) provides a $C^1$ zero extension for the case $q<p-1$, we can recast the free boundary problem (\ref{iMinfFBVP1})-(\ref{iMinfFBVPbound})
 into the problem (\ref{iMinfchachy1})-(\ref{iMinfchachybound}) as in the following lemma.
\begin{lem}\label{icor21}
Let $u(r)$ be a solution to the free boundary problem (\ref{iMinfFBVP1})-(\ref{iMinfFBVPbound}), and $u(r)=0$ for $r\geq R$. Then the zero extension solution $u(r)\in C^1(0,\infty)$ is a nonnegative solution to the following problem in the distribution sense
\begin{eqnarray}
&& (|u'|^{p-2}u')' + \frac{d-1}{r} |u'|^{p-2}u' = u^{q},~~0<r<+\infty,\label{idistrequ}\\
&& \lim_{r\to 0^+} u(r)=1,\,\,\lim_{r\to\infty} u(r)=0.\label{idistrbound}
 \end{eqnarray}
\end{lem}
\begin{proof}
Since $u$ is a solution to the free boundary problem (\ref{iMinfFBVP1})-(\ref{iMinfFBVPbound}), by Step 3 in the proof of Proposition \ref{iLcriticalequa}, we know that $u$ is also a classical solution in $(0,R)$. Notice that $ u'( R)=0$, which allows us to make a $C^1$-zero extension, i.e., extend it to $u(r)=0$ for $r\geq R$. Thus we have that the solution $u$ is a $C^1$-nonnegative solution to (\ref{idistrequ})-(\ref{idistrbound}) in $(0,\infty)$.
\end{proof}

\begin{prop}\label{pdeuiscri}
Let $u(r)$ be a solution to (\ref{idistrequ})-(\ref{idistrbound}) in $Y_{rad}^*$. Then for any $\lambda>0$, the re-scaling function $u_{\lambda}(r)=u(\frac{r}{\lambda})$ is a critical point of $G(u)$ in $Y_{rad}^*$.
\end{prop}
\begin{proof}
{\bf Step 1.} In this step, we show that $\mathcal{G}(u)=0$, $\mathcal{G}(u)$ is defined by (\ref{isteadyfree0}).

Since $u$ satisfies the equation (\ref{idistrequ})-(\ref{idistrbound}) and the decay estimates (\ref{ipdaq1oinftypro22})-(\ref{ipq1oinftypro22}), hence by (\ref{ifrvsangle}), we have $\mathcal{G}(u)=0$, i.e.,
\begin{eqnarray}\label{ifudeng0}
\int_{\mathbb{R}^d} |\nabla u|^p\,dx= \frac{pd}{(q+1)(p-d)} \int_{\mathbb{R}^d} u^{q+1}\,dx.
\end{eqnarray}

{\bf Step 2.} We prove that for any $\lambda >0$, the re-scaling function $u_{\lambda}(r)=u(\frac{x}{\lambda})$ is a critical point of $G(u)$ in $Y_{rad}^*$.

In fact, it is directly verified that for any admissible variation $\phi\in C_c^1(0,\infty)$ at $u_{\lambda}$ (i.e., there is a $\varepsilon_0>0$ such that for any $|\varepsilon|<\varepsilon_0$ one has $u_{\lambda}+\varepsilon \phi_{\lambda}\in Y_{rad}^*$), we have
\begin{align}
&\frac{1}{G(u_{\lambda})}\frac{d}{d\varepsilon}\Big|_{\varepsilon=0} G(u_{\lambda}+\varepsilon \phi_{\lambda})\nonumber\\
=&-\theta\|\nabla u_{\lambda}\|^{-p}_{L^p}\int_0^{\infty}\phi_{\lambda}'|u'_{\lambda}|^{p-1}s^{d-1}\, ds+(1-\theta)\|u_{\lambda}\|_{L^{q+1}}^{-q-1}\int_0^{\infty}\phi_{\lambda} u_{\lambda}^{q}s^{d-1}\, ds\nonumber\\
=&-\theta \|\nabla u\|^{-p}_{L^p}\int_0^{\infty}\phi'|u'|^{p-1}r^{d-1}\, dr+(1-\theta)\|u\|_{L^{q+1}}^{-q-1}\int_0^{\infty}\phi(r) u^{q}(r)r^{d-1}\, dr.
\end{align}
Together with (\ref{ifudeng0}), we deduce
\begin{eqnarray*}
\frac{1}{G(u_{\lambda})}\frac{d}{d\varepsilon}\Big|_{\varepsilon=0} G(u_{\lambda}+\varepsilon \phi_{\lambda})
=-\theta \|\nabla u\|^{-p}_{L^p}\left(\int_0^{\infty}\phi'|u'|^{p-1}r^{d-1}\, dr-\int_0^{\infty}\phi(r) u^{q}(r)r^{d-1}\, dr\right).
\end{eqnarray*}
Noticing that $u$ is a distribution solution to (\ref{idistrequ})-(\ref{idistrbound}) in $Y_{rad}^*$, then it holds that
$$
\frac{1}{G(u_{\lambda})}\frac{d}{d\varepsilon}\Big|_{\varepsilon=0} G(u_{\lambda}+\varepsilon \phi_{\lambda})=0.
$$
Hence any re-scaling function of $u$ is a critical point of $G(u)$ in $Y_{rad}^*$.

\end{proof}

\section{Existence and uniqueness for Euler-Lagrange equations in $L^{\infty}$ case}
In this section, we prove existence and uniqueness of solutions to the Euler-Lagrange equations (\ref{idistrequ})-(\ref{idistrbound}) of  $L^{\infty}$-type G-N inequalities. We also show that the Euler-Lagrange equations are equivalent to some Thomas-Fermi type equations.

\subsection{Existence}
  In this subsection, we prove existence of solutions $u(r)$ to the problem (\ref{idistrequ})-(\ref{idistrbound}). First, we show that there is a singularity of $u'(r)$ at $r=0$: $u'(r)\sim C r^{-\frac{d-1}{p-1}}$ at $r\rightarrow0$ (Proposition \ref{existenceth}). We then prove existence through a limit of a sequence of solutions in the exterior domain $(r_i, \infty)$, $r_i \to 0$. The main ingredients of the convergence proof are: (i) Comparison principle (Lemma \ref{compare}); (ii) Uniform low bound nearby $r=0$ (Lemma \ref{0estimate}); (iii) Application of the Dini theorem.

  We introduce the following exterior Dirichlet problem, which was studied in \cite{pucci2007maximum}
 \begin{eqnarray}
&&(|u'|^{p-2}u')' + \frac{d-1}{r} |u'|^{p-2}u' = u^q,\quad r_0<r<\infty,\label{extequation}\\
&&  u(r_0)=1,\quad\lim_{r\to \infty} u(r)=0.\label{extinitial}
\end{eqnarray}
From \cite[Theorem 4.3.1]{pucci2007maximum} and \cite[Theorem 4.3.2]{pucci2007maximum}, we know that the problem (\ref{extequation})-(\ref{extinitial}) has a unique solution $u(r)\in C^1[r_0,\infty)$ satisfying $u'(r)<0$ when $u(r)>0$. Furthermore, {\it this solution $u(r)$ is non-increasing in $[r_0,+\infty)$}, although this statement is not directly stated in \cite[Theorem 4.3.1]{pucci2007maximum}, the non-increase of $u$ is a consequence in their proof (\cite[p.94,  line 1-4]{pucci2007maximum}). See also the proof of Proposition \ref{existenceth} below. We refer to $u(r)$ as a $C^1$ non-increasing solution.

The following proposition is to give a characterization of singularity of $u'(r)$ at $r=0$.
\begin{prop}\label{existenceth}
For $p>d\geq 1$, $q>0$, and any $r_0>0$, the non-increasing solution $u(r)$ to the problem (\ref{extequation})-(\ref{extinitial}) satisfies
  \begin{eqnarray}\label{uqest}
 r_0^{d-1}|u'(r_0)|^{p-1} =\int_{r_0}^{\infty}r^{d-1}u^q\,dr<\infty.
   \end{eqnarray}
\end{prop}
\begin{proof}
From the proof of \cite[Theorem 4.3.1]{pucci2007maximum},  we know that the non-increasing solution $u(r)$ to the problem (\ref{extequation})-(\ref{extinitial}) is the limit of a non-increasing function sequence $\{u_j(r)\}_1^{\infty}$, which is a solution to the following truncated exterior problem
 \begin{eqnarray}
&&(|u'|^{p-2}u')' + \frac{d-1}{r} |u'|^{p-2}u' = u^q,\,\,  r_0<r<r_0+j,\label{xxquation}\\
&&  u(r_0)=1,\quad u(r_0+j)=0,\label{xxinitial}\\
&& u'(r)\leq 0, \mbox { in } [r_0,r_0+j],
\end{eqnarray}
and satisfies $u_j(r)<u_{j+1}(r)\leq 1$ for $r>r_0$. This implies that
\begin{enumerate}[(i)]
\item $|u'_{j}(r_0)|\leq |u'_{j-1}(r_0)|\leq \cdots \leq |u'_{1}(r_0)|$;
\item there is $u(r)$ such that $\lim_{j\to \infty}u_j(r)=u(r)$.
\end{enumerate}
Multiplying $r^{d-1}$ to (\ref{xxquation}) and integrating the result equation from $r_0$ to $r_0+j$, we obtain
\begin{eqnarray}\label{equal}
(r_0+j)^{d-1}|u_j'(r_0+j)|^{p-1}-r_0^{d-1}|u_j'(r_0)|^{p-1}+\int_{r_0}^{r_0+j}r^{d-1}u_j^q\,dr=0,
\end{eqnarray}
which means
\begin{eqnarray}\label{Ij}
I_j:=\int_{r_0}^{r_0+j}r^{d-1}u_j^q\,dr\leq r_0^{d-1}|u_j'(r_0)|^{p-1}\leq r_0^{d-1}|u_1'(r_0)|^{p-1}\quad \mbox{ for any } j\in \mathbb{N}^+.
\end{eqnarray}
Extending $u_j(r)=0$ for $r\geq r_0+j$, we have $\int_{r_0}^{\infty}r^{d-1}u_j^q\,dr<r_0^{d-1}|u_1'(r_0)|^{p-1}$.  Hence by the Monotone Convergence Theorem, we have
$$\int_{r_0}^{\infty}r^{d-1}u^q\,dr=\lim_{j\to \infty} \int_{r_0}^{\infty}r^{d-1}u_j^q\,dr\leq r_0^{d-1}|u_1'(r_0)|^{p-1} .$$
Taking the limit for (\ref{equal}) and using (\ref{idecayrate}), we get (\ref{uqest}).
\end{proof}

In order to show existence of solutions to the problem (\ref{idistrequ})-(\ref{idistrbound}) (see Theorem \ref{existence} below), first we prove the following three lemmas.
 \begin{lem}{\rm(Comparison principle)}\label{compare}
Let $u_1$ and $u_2$ be $C^1$ non-increasing solutions to the  exterior problem  (\ref{extequation})-(\ref{extinitial}) with $r_0=r_1$ and  $r_0=r_2$  respectively. Then if $r_1< r_2$, we have $u_2(r)>u_1(r)$ when $u_2(r)>0$.
\end{lem}
\begin{proof}
Since $u_1(r_1)=1$ and $u_1'(r)<0$ when $u_1(r)>0$, we have $u_1(r)<1$ in $(r_1,r_2]$. Hence  $u_2(r_2)>u_1(r_2)$ due to $u_2(r_2)=1$. Using a contradiction method, we assume that there is $r_*: ~r_2<r_*<\infty$ such that $u_2(r_*)=u_1(r_*)=:m^*>0$. Considering the following problem
 \begin{eqnarray}
&&(|u'|^{p-2}u')' + \frac{d-1}{r} |u'|^{p-2}u' = u^q,\,\,  r_*<r<\infty,\label{xextequation}\\
&&  u(r_*)=m^*,\quad\lim_{r\to \infty} u(r)=0,\label{xextinitial}
\end{eqnarray}
we know that solutions $u_1(r)$ and $u_2(r)$ defined in $[r_*,\infty)$ are $C^1$ non-increasing solutions to (\ref{xextequation})-(\ref{xextinitial}). The uniqueness of the $C^1$ non-increasing solution to the problem (\ref{xextequation})-(\ref{xextinitial}) implies that $u_1(r)=u_2(r)$ for $r\in [r_*,\infty)$. By ODE theory, we know that $u_1(r)=u_2(r)$ for any $r\geq r_2$, which is a contradiction with $u_2(r_2)>u_1(r_2)$.
\end{proof}

\begin{lem}\label{lmmucon} Let $u$ be the $C^1$ non-increasing solution to the problem  (\ref{extequation})-(\ref{extinitial}). Define $\mu_{r_0}:=\int_{r_0}^{+\infty}r^{d-1} u^q(r)\,dr$. Then $\mu_{r_1}\leq \mu_{r_2}+\frac{r_2^d-r_1^d}{d}$ for any $r_2\geq r_1\geq r_0$.
\end{lem}
\begin{proof}
Since $u'(r)\leq 0$ in $(r_0,+\infty)$, we have $u(r)\leq 1$ in $[r_1,r_2]$. Hence for any $r_1,r_2$ satisfying $r_0\leq r_1\leq r_2$, a direct computation gives
\begin{align}\label{fmu1}
\mu_{r_1}=&\left(\int_{r_1}^{r_2}+\int_{r_2}^{+\infty}\right)r^{d-1} u^q(r)\,dr\nonumber\\
\leq & \int_{r_1}^{r_2}r^{d-1} \,dr+\int_{r_2}^{+\infty}r^{d-1} u^q(r)\,dr =\frac{r_2^d-r_1^d}{d}+\mu_{r_2}.
\end{align}
\end{proof}
\begin{lem}{\rm(Uniform low bound)}\label{0estimate}
Let $u$ be the $C^1$ non-increasing solution to the problem  (\ref{extequation})-(\ref{extinitial}). Then for any $r>r_0$, there is $C>0$ independent of $r_0$ and $r$ such that
 \begin{eqnarray}\label{fmu0}
u(r)\geq 1-C\left( r^{\frac{p-d}{p-1}}-r_0^{\frac{p-d}{p-1}}\right).
\end{eqnarray}
\end{lem}
\begin{proof}
Multiplying $r^{d-1}$ to (\ref{extequation}) and integrating the result equation from $r_0$ to $\infty$, we deduce
\begin{eqnarray}\label{identity}
r_0^{d-1}|u'(r_0)|^{p-1} =\int_{r_0}^{\infty}r^{d-1}u^{q}(r) \,dr=\mu_{r_0}.
 \end{eqnarray}
Again multiplying $r^{d-1}$ to (\ref{extequation}) and integrating it from $r_0$ to $r$, and using (\ref{identity}), we obtain
$$
r^{d-1}|u'(r)|^{p-1}-\mu_{r_0}=-\int_{r_0}^r r^{d-1}u^q(r)\,dr.
$$
From above equation with $r>r_0$ and $u\geq0$, we have
\begin{eqnarray}\label{fmu2}
-u'(r)\leq r^{-\frac{d-1}{p-1}}\mu_{r_0}^{\frac{1}{p-1}}.
\end{eqnarray}
Integrating (\ref{fmu2}) from $r_0$ to $r$, we deduce  \begin{eqnarray}\label{fmu3}
u(r)\geq 1-\frac{p-1}{p-d}\mu_{r_0}^{\frac{1}{p-1}}\left( r^{\frac{p-d}{p-1}}-r_0^{\frac{p-d}{p-1}}\right).
\end{eqnarray}
By Lemma \ref{lmmucon}, we know that for a fixed $r_*>r_0$, $\mu_{r_0}\leq \mu_{r_*}+\frac{r_*^d-r_0^d}{d}\leq \mu_{r_*}+\frac{r_*^d}{d}$.
Denote $C:=\frac{p-1}{p-d}\left(\mu_{r_*}+\frac{r_*^d}{d}\right)^{\frac{1}{p-1}}$. Hence (\ref{fmu3}) implies that (\ref{fmu0}) holds true.

\end{proof}

\begin{thm}(Existence)\label{existence}
Assume that exponents $p>d\geq 1$ and $q>0$, then there is a $C^1$ non-increasing solution to the problem (\ref{idistrequ})-(\ref{idistrbound}) and it satisfies
 \begin{enumerate}[(i)]
\item  $u'(r)<0$ for $u(r)>0$;
\item for any $u\in L^{q+1}(\mathbb{R}^d)$ and $\nabla u\in L^{p}(\mathbb{R}^d)$, it holds that
\begin{eqnarray}\label{hengdengs}
\int_{0}^{\infty}r^{d-1}|u'|^p\,dr+\int_{0}^{\infty}r^{d-1}u^{q+1}\,dr=\lim_{r\to 0^+}r^{d-1}|u'|^{p-1}=\int_{0}^{\infty}r^{d-1}u^{q}\,dr.
\end{eqnarray}
\end{enumerate}
\end{thm}
\begin{proof} Let $u_i(r)$ be the $C^1$ non-increasing solution to the problem  (\ref{extequation})-(\ref{extinitial}) with $u_i(r_i)=1$, $r_i>r_{i+1}>0$ for any $i\in \mathbb{N}^+$, and $r_{i+1}\to 0^+$ as $i\to +\infty$.

{\bf Step 1.} We prove that
there is a continuous, non-negative and non-increasing function $u(r)$ such that
 \begin{eqnarray}
 u_i(r)\to u(r),\quad u'_i(r)\to  u'(r), ~~\mbox{ for all } r>0,
 \end{eqnarray}
 and they converges uniformly in any interval $[a,b]$ for $0<a<b<\infty$.

Notice that $\{u_i(r)\}_{i_0}^{\infty}$ is continuous, non-negative and non-increasing sequence and bounded below in $[a,\infty)$, $a>0$ ($0<r_{i_0}<a$).
Hence by the Dini theorem the sequence $\{u_i(r)\}_{i_0}^{\infty}$ converges uniformly on every
compact interval $[a,b]$ of $(0,\infty)$ to a non-negative, non-increasing, continuous
function $u(r)$, i.e.,
 \begin{eqnarray}
 u_i(r)\to u(r)\quad ~~\mbox{ for all } r>0,\quad \mbox{ as } i\to +\infty,
 \end{eqnarray}
 and they converges uniformly in any interval $[a,b]$ for $0<a<b<\infty$. Since $u_i(r)$ is a non-negative and non-increasing function in $r$, hence $u(r)$ is also a non-negative and non-increasing function.

Moreover, let $u_j(r)$ be a solution to the equation (\ref{extequation}) with $u(r_j)=1$. Noticing
$$
\lim_{r\to \infty}r^{d-1}|u'(r)|^{p-1}=0
$$
 due to (\ref{idecayrate}), we have
$$
r^{d-1}|u_j'(r)|^{p-1}=\int_{r}^{\infty}s^{d-1}u_j^{q}(s)\,ds>\int_{r}^{\infty}s^{d-1}u_{j+1}^{q}(s)\,ds=r^{d-1}|u_{j+1}'(r)|^{p-1}, ~~r>r_j,
$$
which means $u'_j(r)<u'_{j+1}(r)\leq 0,~~r>r_j$. Hence by the Dini theorem we have
\begin{eqnarray}
 u'_j(r)\to u'(r)\quad ~~\mbox{ for all } r>0,
 \end{eqnarray}
 and they converges uniformly in any interval $[a,b]$ for $0<a<b<\infty$.

{\bf Step 2.} We prove $\lim_{r\to\infty}u(r)=0$ and $\lim_{r\to 0^+}u(r)=1$.

Since $\lim_{r\to \infty}u_i(r)=0$ by (\ref{extinitial}), we have $\lim_{r\to\infty}u(r)=0$. On the other hand, by Lemma \ref{0estimate} we have
$$
\lim_{r\to 0^+}\lim_{i\to \infty}u_i(r)\geq 1.
$$
Together with $\lim_{r\to 0^+}\lim_{i\to \infty} u_i(r)\leq 1$ gives that
$\lim_{r\to 0^+} u(r)= 1$. Thus the limit function $u(r)$ satisfies the boundary condition (\ref{idistrbound}).

{\bf Step 3.} We prove that the limit function $u(r)$ is the required radial solution of (\ref{idistrequ}) in the distribution sense.

In fact, for any $\phi \in C^1_c(0,+\infty)$, we have
\begin{eqnarray}
-\int_0^{+\infty} \phi' r^{d-1}|u_i'|^{p-1}\,dr+\int_0^{+\infty} \phi r^{d-1}u_i^{q}\,dr=0.
 \end{eqnarray}
Using the uniform convergence property of $u_i$ and  $u_j'$ from Step 1, we obtain
\begin{eqnarray}
-\int_0^{+\infty} \phi' r^{d-1}|u'|^{p-1}\,dr+\int_0^{+\infty} \phi r^{d-1}u^{q}\,dr=0,
 \end{eqnarray}
i.e., the limit function $u(r)$ is the required radial solution of (\ref{idistrequ}) in the distribution sense.

{\bf Step 4.} Regularity.

From (\ref{idistrequ}), a similar process to obtaining  (\ref{iuderi}) gives that $u'(r)<0$ in the set $\{r|u(r)>0\}$, i.e, the case (i) holds.

Similar to (\ref{identity}), we have
 \begin{eqnarray*}
r^{d-1}|u'(r)|^{p-1} =\int_r^{\infty}s^{d-1}u^{q}(s) \,ds=\mu_r,
 \end{eqnarray*}
 where $\mu_r$ is defined in Lemma \ref{lmmucon}. By Lemma \ref{lmmucon}, we know that $\mu_r$  has a uniform upper bound independent of $r$. Thus
 \begin{eqnarray}\label{hengdengs3}
\lim_{r\to 0^+} r^{d-1}|u'(r)|^{p-1} =\int_0^{\infty}s^{d-1}u^{q}(s) \,dr<\infty.
 \end{eqnarray}
On the other hand, multiplying $u(r)$ in the both sides of the equation (\ref{ixequation}), integrating it from $r$ to $\infty$, and using the facts $u(r)\to 0$, $r^{d-1}|u'|^{p-1}\to 0$ as $r\to \infty$, we have
\begin{eqnarray}\label{hengdengs2}
 \int_r^{\infty}s^{d-1}|u'(s)|^{p} \,dr+\int_r^{\infty}s^{d-1}u^{q+1}(s) \,dr=r^{d-1}u(r)|u'(r)|^{p-1}.
 \end{eqnarray}
Hence using (\ref{hengdengs3}), (\ref{hengdengs2}) and $\lim_{r\to 0^+}u(r)=1$ gives that (\ref{hengdengs}) holds.

 This completes the proof of Theorem \ref{existence}.
\end{proof}
\begin{rem}
Theorem \ref{existence} proved existence of $C^1$ non-increasing solutions to (\ref{idistrequ})-(\ref{idistrbound}) for the case $q>0$. Existence for the case $q=0$ will be established in Proposition \ref{prop2} below by giving an exact closed form solution.
\end{rem}

\subsection{Uniqueness }
In this subsection, we prove uniqueness of solutions to (\ref{idistrequ})-(\ref{idistrbound}) by following works \cite[Theorem 1]{KLS1996} given by Franchi, Lanconelli and Serrin. Let $u(r)$ and $v(r)$ be two $C^1$ non-increasing solutions to the problem (\ref{idistrequ})-(\ref{idistrbound}). By (i) of Theorem \ref{existence}, we have $u'(r)<0$ when $u(r)>0$, and hence both $u(r)$ and $v(s)$ possess inverse functions in those supports.
We denote respectively  by $r(u)$ and $s(v)$ the inverse functions of $u(r)$ and $v(s)$, defined on the
interval $(0, 1]$.

The following two lemmas are special cases from results in [7]. We supply a proof to show how their proof is used in our special cases.

\begin{lem}\label{lm1}\cite[Lemma 3.3.1]{KLS1996}
Assume $q\geq 0$ and $d>1$.
If $r(u)>s(u)$ in some
open interval $(0,1)$, then $r(u)-s(u)$ can have at most one critical point in $(0,1)$. Moreover if such a critical point exists, it must be a strict maximum point.
\end{lem}
\begin{proof} By the equation (\ref{idistrequ}), it is immediately verified that the function $r=r(u)$ satisfies the
equation
$$
(p-1)r_{uu}-\frac{d-1}{r}r_u^2-|r_u|^{p+1}u^q=0,~~0<u<1,
$$
and the same equation holds for $s(u)$. Hence by subtracting one from another, we get
\begin{eqnarray}\label{r-sequa}
(p-1)(r-s)_{uu}-(d-1)\left(\frac{r_u^2}{r}-\frac{s_u^2}{s}\right)-\left(|r_u|^{p+1}-|s_u|^{p+1}\right)u^q=0.
\end{eqnarray}
Now we suppose that $u=u_*\in (0,1)$ is a critical point of $r(u)-s(u)$, then $r_u=s_u<0$ at $u=u_*$. Thus from (\ref{r-sequa}), we have that
$$
(p-1)(r-s)_{uu}=(d-1)r_u^2\left(\frac{1}{r}-\frac{1}{s}\right)<0, \mbox{ at } u=u_*,
$$
where the last inequality used the fact $r(u)>s(u)$ in $(0,1)$.
Hence we get that all critical points must be maximum points, which implies that $r(u)-s(u)$ has at most one critical point in $(0,1)$.
\end{proof}

\begin{lem}\label{lm2}\cite[Lemma 3.3.2]{KLS1996}
Assume $q\geq 0$ and $d>1$.
If $r(u)-s(u)$ has two zero points
in $(0, 1]$, denoting them as $\xi_0$ and $\xi_1$, then $r(u)=s(u)$ for all $u$ between $\xi_0$ and $\xi_1$ .
\end{lem}
\begin{proof} Inspired by \cite[Lemma 3.3.2]{KLS1996}, use the contradiction method to prove this lemma.
Without loss of generality,
we assume that $\xi_0<\xi_1$ and $r(u)>s(u)$ for all $u\in (\xi_0, \xi_1)$. By Lemma \ref{lm1}, we know that $r(u)-s(u)$ has at most one critical point in $(\xi_0, \xi_1)$. Since $r(\xi_0)-s(\xi_0)=r(\xi_1)-s(\xi_1)=0$, then there is at least one critical point in $(\xi_0, \xi_1)$. Suppose that $\xi_2\in (\xi_0, \xi_1)$ is a unique critical point satisfying $(r-s)'(\xi_2)=0$. From (\ref{r-sequa}), we have
\begin{eqnarray}
(p-1)(r''(u)-s''(u))=(d-1)(r'(u))^2\left(\frac{1}{r(u)}-\frac{1}{s(u)}\right) <0, \quad \mbox{at } u=\xi_2,
\end{eqnarray}
where the last inequality used the fact $r(\xi_2)>s(\xi_2)$.
Hence by the continuity of $r'(u)$, we know that there is a $\delta>0$ such that $(r-s)'(u)<0$ in $(\xi_2,\xi_2+\delta)$. Since the critical point of $r(u)-s(u)$ is unique in $(\xi_0,\xi_1)$, we have
\begin{eqnarray}\label{danjian}
(r-s)'(u)<0,\quad \mbox{in } (\xi_2,\xi_1), \quad \mbox{ i.e. } |r'(u)|>|s'(u)|.
\end{eqnarray}
Denote $r_1=r(\xi_1)$ and $r_2=r(\xi_2)$. Multiplying $r^{d-1}$ to the equation (\ref{idistrequ}) and integrating the result equation from $r_1$ to $r_2$, we deduce
\begin{eqnarray}\label{jifen}
\int_{r_1}^{r_2} \frac{d}{dr}\left( r^{d-1}|u'|^{p-2}u'(r)\right)\,dr=\int_{r_1}^{r_2} r^{d-1} u^q(r)\,dr=\int_{\xi_2}^{\xi_1}r(u)^{d-1}\frac{u^q}{|u'(r(u))|}\,du.
\end{eqnarray}
Hence it holds that
\begin{eqnarray}\label{deffru}
r_2^{d-1}|u'(r_2)|^{p-2}u'(r_2)-r_1^{d-1}|u'(r_1)|^{p-2}u'(r_1) = \int_{\xi_2}^{\xi_1}r(u)^{d-1}\frac{u^q}{|u'(r(u))|}\,du.
\end{eqnarray}

Similar for $v$, denote $s_1=s(\xi_1)$ and $s_2=s(\xi_2)$. We have the same formula
\begin{eqnarray}\label{deffrv}
s_2^{d-1}|v'(s_2)|^{p-2}v'(s_2)-s_1^{d-1}|v'(s_1)|^{p-2}v'(s_1) = \int_{\xi_2}^{\xi_1}s(u)^{d-1}\frac{u^q}{|v'(s(u))|}\,du.
\end{eqnarray}
Due to $r_1=r(\xi_1)=s(\xi_1)=s_1$ and $r'(\xi_2)=s'(\xi_2)$, subtracting (\ref{deffrv}) from (\ref{deffru}) gives that
\begin{align}\label{subtraction}
\left(s_2^{d-1}-r_2^{d-1}\right)|u'(r_2)|^{p-1}+&r_1^{d-1}\left(|u'(r_1)|^{p-1}-|v'(s_1)|^{p-1}\right)\nonumber\\
&=\int_{\xi_2}^{\xi_1}u^q\left(\frac{r(u)^{d-1}}{|u'(r(u))|}-\frac{s(u)^{d-1}}{|v'(s(u))|}\right)\,du.
\end{align}
Since (\ref{danjian}), $s_2<r_2$ and $r(u)>s(u)$ for all $u\in (\xi_0, \xi_1)$, we directly verify that both terms on the left side of (\ref{subtraction}) are strictly negative, while the right side of (\ref{subtraction}) is non-negative. This is a contradiction. Hence the assumption is not true, i.e., $r(u)=s(u)$ in $ (\xi_0, \xi_1)$.

\end{proof}

\begin{thm}(Uniqueness)\label{uniqueness}
 Assume $q\geq 0$ and $p>d>1$. Let $u$ and $v$ be two $C^1$ non-increasing solutions of the problem (\ref{idistrequ})-(\ref{idistrbound}).
Then $u(r)\equiv v(r)$ for any $0\leq r <\infty$.
\end{thm}

\begin{proof} We use a contradiction method to prove this theorem. If not, then $u(r)\not\equiv v(r)$ on $[0,\infty)$. Equivalently their inverse functions $r(u)\not\equiv r(u)$ in $(0,1]$. Hence there is $u_*\in (0,1)$ such that $r(u_*)\neq s(u_*)$. Then Lemma \ref{lm2} implies that $r(u), s(u)$ satisfy either $r(u)>s(u)$ or $r(u)<s(u)$ in $(0,1)$.

For the case $q\geq p-1$, without loss of generality, we suppose $r(u)>s(u)$ for $u\in (0,1)$, then $u(r)>v(r)$ for $r>0$. Multiplying $r^{d-1}$ to the equation (\ref{idistrequ}) and integrating the result equation form $r$ to $\infty$ and using (\ref{idecayrate}), we have
\begin{eqnarray}\label{ud}
 r^{d-1}|u'|^{p-1}=\int_r^{\infty}s^{d-1} u^q(s)\,ds.
\end{eqnarray}
The same process for $v(r)$ gives
\begin{eqnarray}\label{vd}
r^{d-1}|v'|^{p-1}=\int_r^{\infty}s^{d-1}v^q(s)\,ds.
\end{eqnarray}
Subtracting (\ref{vd}) from (\ref{ud})  gives that
$$
r^{d-1}|u'(r)|^{p-1}-r^{d-1}|v'(r)|^{p-1}=\int_r^{\infty}s^{d-1}(u^q(s)-v^q(s))\,ds.
$$
Since $u(r)>v(r)$ and $q\geq p-1>0$, then we have from the above equation
\begin{eqnarray}\label{dersign}
v'(r)>u'(r),\quad \mbox{for }r>0.
\end{eqnarray}
Integrating (\ref{dersign}), we obtain $\lim_{r\to 0^+}v(r)<\lim_{r\to 0^+}u(r)$, which is a contradiction with $\lim_{r\to 0^+}v(r)=\lim_{r\to 0^+}u(r)=1$.

For the case $0\leq q<p-1$, we suppose $r(u)>s(u)$ for $u\in (0,1)$. Then $u(r)>v(r)$ for $0<r<R_v$, and $u>0$, $v=0$ for $R_v<r<R_u$. Multiplying $r^{d-1}$ to the equation (\ref{idistrequ}) and integrating the result equation form $r$ to $R_v$, we have
\begin{align*}
r^{d-1}|v'|^{p-1}=&\int_r^{R_v}s^{d-1}v^q(s)\,ds\\
<& \left(\int_r^{R_v}+\int_{R_v}^{R_u}\right)s^{d-1}u^q(s)\,ds=r^{d-1}|u'|^{p-1}~~\mbox{ for } 0<r<R_v.
\end{align*}
Thus $v'(r)>u'(r)$ for $0<r<R_v$. In $(R_v,R_u)$, $u'(r)<0=v'(r)$. Hence we have
$$
0<\int_0^{R_u}(v'(r)-u'(r))\,dr=\int_0^{R_v}v'(r)\,dr-\int_0^{R_u}u'(r)\,dr=0,
$$
which is a contradiction.
\end{proof}
\begin{rem}
For the case $d=1$, uniqueness of $C^1$ non-increasing solutions to the problem (\ref{idistrequ})-(\ref{idistrbound}) is given by a direct computation in Proposition \ref{prop42} below.
\end{rem}

\subsection{Thomas-Fermi type equation}
This subsection shows that the non-increasing solution of the Euler Lagrange equation obtained above is equivalent to the radial non-increasing solution to a Thomas-Fermi type equation.

\begin{defn}\label{def1}
We call a function $u(|x|)$ a radial non-increasing weak solution to the Thomas-Fermi type equation (\ref{inonwhole})-(\ref{inonwholeboud}) if $u(|x|)$ satisfies
\begin{enumerate}[(i)]
\item $u(|x|)$ is a non-increasing function in $|x|$ and $\lim_{|x|\to 0^+}u(|x|)=1$,
\item $\nabla u\in L^p$, $u\in L^{q+1}$, and denote $a:=\|\nabla u\|^p_{L^p}+\|u\|^{q+1}_{L^{q+1}}$,
\item for any $\phi(|x|)\in C^{\infty}_c(\mathbb{R}^d)$, it holds that $(\nabla\phi,|\nabla u|^{p-2}\nabla u)+(\phi, u^q)=a(\phi, \delta_{x=0})$.
\end{enumerate}
\end{defn}

\begin{prop} \label{prop} Assume $p>d\geq 1$, then
\begin{enumerate}[(i)]
\item For the case $q\geq p-1$,
\begin{itemize}
\item [(a)] if $u(r)$ is the weak solution to the problem (\ref{iMinfchachy1})-(\ref{iMinfchachybound}) in $Y^*_{rad}$, then $u(|x|)$ is a radial non-increasing weak solution to a Thomas-Fermi type equation (\ref{inonwhole})-(\ref{inonwholeboud}).
\item [(b)] if $u(|x|)\in W^{1,p}(\mathbb{R}^d)$ is a radial non-increasing weak solution to the Thomas-Fermi type equation (\ref{inonwhole})-(\ref{inonwholeboud}), then $u(r)$ is also the solution of (\ref{iMinfchachy1})-(\ref{iMinfchachybound}) in $Y^*_{rad}$.
\end{itemize}
\item  For the case $q<p-1$,
\begin{itemize}
\item [(a)] if $u(r)$ is the solution to the free boundary problem (\ref{iMinfFBVP1})-(\ref{iMinfFBVPbound}) in $Y^*_{rad}$, then $u(|x|)$ is a radial non-increasing solution to a Thomas-Fermi type equation (\ref{iinftyfunction})-(\ref{iinftyfunctionboud}).
\item  [(b)] if $u(|x|)\in W^{1,p}(\mathbb{R}^d)$ is a radial non-increasing solution to the Thomas-Fermi type equation (\ref{iinftyfunction})-(\ref{iinftyfunctionboud}), then $u(r)$ is also the solution of (\ref{iMinfFBVP1})-(\ref{iMinfFBVPbound}) in $Y^*_{rad}$.
    \end{itemize}
\end{enumerate}
In particular, for $d=1$ we have $a=2 \left(\frac{p}{(q+1)(p-1)}\right)^{\frac{p-1}{p}}$.
\end{prop}
\begin{proof}
We first prove the case (i). For the case (a), suppose that $u(r)$ is the solution to (\ref{iMinfchachy1})-(\ref{iMinfchachybound}) in $Y^*_{rad}$. Hence we know that $u(|x|)\in W^{1,p}(\mathbb{R}^d)$ is radial non-increasing and satisfies $u(0)=1$ and $u(|x|)\to 0$ as $|x|\to \infty$. Hence the boundary condition (\ref{inonwholeboud}) holds.

For any test function $\phi(|x|)\in C_c^{\infty}(\mathbb{R}^d)$, it holds that
\begin{eqnarray*}
-(\nabla\phi,|\nabla u|^{p-2}\nabla u)-(\phi,u^q)=S_d\int_0^{\infty} \left(\phi'(r) |u'(r)|^{p-1}-\phi(r) u^q(r)\right)r^{d-1}\,dr.
\end{eqnarray*}
From Proposition \ref{iLcriticalequa}, we have that the solution is classical in $(0,\infty)$. Hence by integration by parts we have
\begin{align}\label{ilimit}
(\nabla\phi,|\nabla u|^{p-2}\nabla u)+(\phi,u^q)&=S_d\phi(0)\lim_{r\to 0^+}|u'(r)|^{p-1}r^{d-1}\nonumber\\
&\quad+S_d\int_0^{\infty}\phi \left(\left(|u'(r)|^{p-1}r^{d-1}\right)'+ u^qr^{d-1}\right)\,dr.
\end{align}
On the other hand, from Theorem \ref{existence} we know that
\begin{eqnarray}\label{ilimit1}
\lim_{r\to 0^+}S_dr^{d-1}|u'(r)|^{p-1}= S_d\int_{0}^{\infty}r^{d-1}|u'|^p\,dr+S_d\int_{0}^{\infty}r^{d-1}u^{q+1}\,dr=a.
\end{eqnarray}
Using (\ref{ilimit1}) and the equation (\ref{idistrequ}), from (\ref{ilimit}) we obtain
$$
(\nabla\phi,|\nabla u|^{p-2}\nabla u)+(\phi,u^q)=a\phi(0)=a(\phi, \delta_{x=0}).
$$
Hence we have that the following equation holds in the distribution sense
$$
\Delta_p u+ a \delta_{x=0}=u^q,
$$
Therefore, $u(|x|)$ is a radial non-increasing weak solution to a Thomas-Fermi type equation (\ref{inonwhole})-(\ref{inonwholeboud}).

Now we prove the case (b), assume that $u(|x|)\in W^{1,p}(\mathbb{R}^d)$ is a radial non-increasing weak solution of (\ref{inonwhole})-(\ref{inonwholeboud}) in Definition \ref{def1}, then $u(r):=u(|x|)$ satisfies (\ref{iMinfchachybound}) and for any test function $\phi(|x|)\in C_c^{\infty}(\mathbb{R}^d)$ satisfying $\phi(0)=0$, it holds that
\begin{align*}
0=-(\nabla\phi,|\nabla u|^{p-2}\nabla u)-(\phi,u^q)&=S_d\int_0^{\infty} \left(\phi'(r)|u'(r)|^{p-1}-\phi(r) u^q(r)\right)r^{d-1}\,dr\\
&=S_d\int_0^{\infty}\phi \left(\left(r^{d-1}u'|u'|^{p-2}\right)'- u^qr^{d-1}\right)\,dr.
\end{align*}
Hence $u(r)$ satisfies (\ref{iMinfchachy1}). This completes the proof of the case (i).

The proof of the case (ii) is exactly same with the case (i). Here we omit the details.

Finally, we determine the value of $a$ for $d=1$.
Since $u'<0$, (\ref{ilimit1}) implies $a=S_d\lim_{r\to 0^+}r^{d-1}|u'(r)|^{p-1}$. Thus multiplying $u'$ to the equation (\ref{iMinfchachy1}) with $d=1$, and integrating it from $r$ to $\infty$, and using the boundary condition (\ref{iMinfchachybound}), we deduce
$$
\frac{p-1}{p}|u'|^p=\frac{u^{q+1}}{q+1}.
$$
Noticing the fact $\lim_{r\to 0^+}u(r)=1$, we have
$$
\lim_{r\to 0^+}|u'(r)|=\left(\frac{p}{(p-1)(q+1)}\right)^{\frac{1}{p}}.
$$
Thus
$$
a=2\lim_{r\to 0^+}|u'(r)|^{p-1} = 2\left(\frac{p}{(q+1)(p-1)}\right)^{\frac{p-1}{p}}.
$$
This completes the proof of Proposition \ref{prop}.
\end{proof}

\section{Best constant for $L^{\infty}$-type G-N inequality}
This section is divided into three subsections. We give some close form solutions for the case $q=0$ in Subsection 4.1 or for the case $d=1$ in Subsection 4.2.  In Subsection 4.3, we use existence and uniqueness of the Euler-Lagrange equation (\ref{idistrequ})-(\ref{idistrbound}) for $q>0$ and $d>1$ in the previous section, together with one for $q=0$ in Subsection 4.1 and one for $d=1$ in Subsection 4.2, to derive the best constant of $L^{\infty}$-type G-N inequality.

\subsection{Existence, uniqueness and close form solution for $q=0$, $p>d\geq 1$}

In Theorem \ref{existenceth}, we require the condition $q>0$. For the case $q=0$, we use the close form solution to prove existence and uniqueness in the following proposition.
\begin{prop}\label{prop2}
Suppose $d\geq 1$, $p>d$ and $q=0$. Then there is a unique non-negative solution $u_{c,\infty}$ to the free boundary problem (\ref{iMinfFBVP1})-(\ref{iMinfFBVPbound}) and $u_{c,\infty}$ has the following closed form
\begin{eqnarray}
&&u_{c,\infty}(r)=d^{-\frac{p}{p-1}}R^{\frac{p}{p-1}}\left(\mathcal{B}\left(\frac{p-d}{d(p-1)},\frac{p}{p-1}\right)-B\left(\left(\frac{r}{R}\right)^d;
\frac{p-d}{d(p-1)},\frac{p}{p-1}\right)\right),\label{closedformq0}\\
&& R=d\left(\mathcal{B}\left(\frac{p-d}{d(p-1)},\frac{p}{p-1}\right)\right)^{-\frac{p-1}{p}}.\label{R}
\end{eqnarray}
\end{prop}
\begin{proof}
Let $r=Rs$, and $v(s)=u(Rs)$. Hence we have $v(1)=u(R)=0$ and $v(0)=u(0)=1$. Then from (\ref{iMinfFBVP1}) with $q=0$, we obtain that $v(s)$ satisfies the following equation
$$
-v'(s)/R=d^{-\frac{1}{p-1}}\left((R^d-R^ds^d)(Rs)^{1-d}\right)^{\frac{1}{p-1}}=d^{-\frac{1}{p-1}}R^{\frac{1}{p-1}}(1-s^d)^{\frac{1}{p-1}}s^{\frac{1-d}{p-1}}.
$$
Hence
\begin{eqnarray}\label{veq}
-v'(s)=d^{-\frac{1}{p-1}}R^{\frac{p}{p-1}}(1-s^d)^{\frac{1}{p-1}}s^{\frac{1-d}{p-1}}.
\end{eqnarray}
Integrating (\ref{veq}) from $s$ to $1$, we deduce
\begin{align*}
v(s)=&d^{-\frac{1}{p-1}}R^{\frac{p}{p-1}}\int_s^1 (1-r^d)^{\frac{1}{p-1}}r^{\frac{1-d}{p-1}}\,dr\\
=&d^{-\frac{p}{p-1}}R^{\frac{p}{p-1}}\left(\mathcal{B}\left(\frac{p-d}{d(p-1)},\frac{p}{p-1}\right)-B\left(s^d;\frac{p-d}{d(p-1)},\frac{p}{p-1}\right)\right).
\end{align*}
Hence we obtain the closed form solution $u_{c,\infty}$ to the free boundary problem (\ref{iMinfFBVP1})-(\ref{iMinfFBVPbound}) given by (\ref{closedformq0}).

With condition $v(0)=1$, we have explicit formula of $R$:
\begin{align*}
1=& d^{-1/(p-1)} R^{p/(p-1)} \int_0^1 (1-r^d)^{\frac{1}{p-1}}r^{\frac{1-d}{p-1}}\,dr\\
=&d^{-\frac{p}{p-1}}R^{\frac{p}{p-1}}\mathcal{B}\left(\frac{p-d}{d(p-1)},\frac{p}{p-1}\right),
\end{align*}
which means that (\ref{R}) holds.

\end{proof}

\subsection{The close form solution for $q\geq 0$ and $p>d=1$}

In this subsection, we present a result in the one dimensional case, for which there is a closed form solution and
 deduce the best constant $ C_{q,\infty,p}$ of the inequality (\ref{linftyest}) for $d=1$.
 \begin{prop}\label{prop42}
 Suppose $p>d=1$, and $q\geq 0$. Then the solution
  $u_{c,\infty}$ of the problem (\ref{idistrequ})-(\ref{idistrbound}) possesses the following closed form:
\begin{enumerate}[(i)]
\item for $q=p-1$,
\begin{eqnarray}\label{ipq1u}
u_{c,\infty}(r)=e^{-(p-1)^{-\frac{1}{p}}r};
\end{eqnarray}
\item for $q<p-1$,
\begin{eqnarray}
u_{c,\infty}(r)=\left(1-\frac{r}{R}\right)_{+}^{\frac{p}{p-q-1}}, ~~R=\frac{(p-1)^{1/p}(q+1)^{1/p}}{p^{1/p-1}(p-q-1)}, \mbox{ for } r>0; \label{ipxiaoq1u11}
\end{eqnarray}
\item for $q>p-1$,
\begin{eqnarray}
u_{c,\infty}(r):=\left(1+\frac{p^{1/p-1}(q+1-p)}{(p-1)^{1/p}(q+1)^{1/p}}r\right)^{-\frac{p}{q+1-p}}, \mbox{ for } r>0.\label{ipdayuq1u11}
\end{eqnarray}
\end{enumerate}
The best constant is given by $C_{q,\infty,p}=\left(\frac{p+(p-1)(q+1)}{2p}\right)^{\frac{p}{p+(p-1)(q+1)}}$.
\end{prop}
\begin{proof}
For $q=p-1$, multiplying $u'$ in the both sides of the equation (\ref{idistrequ}) with $d=1$ gives that
\begin{eqnarray}\label{igoshi1}
\frac{d}{dr}\left(\frac{p-1}{p}|u'|^{p}-\frac{u^{q+1}}{q+1}\right)=0.
\end{eqnarray}
Noticing $u\in W^{1,p}(\mathbb{R})$ and integrating (\ref{igoshi1}) from $r$ to $\infty$, we have
\begin{eqnarray}\label{igoshi2}
\frac{p-1}{p}|u'|^{p}-\frac{u^{q+1}}{q+1}=0,
\end{eqnarray}
which implies that
$$
u(r)=e^{\pm(p-1)^{-\frac{1}{p}}r}.
$$
 Hence $u(r)=e^{-(p-1)^{-\frac{1}{p}}r}$ is the unique solution satisfying $u(0)=1$ and $\lim_{r\to\infty}u(r)=0$.

 For $q\neq p-1$, solving (\ref{igoshi2}) and using boundary conditions $u(0)=1$ and $\lim_{r\to\infty}u(r)=0$, we can obtain (\ref{ipxiaoq1u11}) and (\ref{ipdayuq1u11}).

  Hence plugging $u_{c,\infty}$ into $M_c$ in (\ref{bestconstant}) below, we deduce
  $$
 M_c(u_{c,\infty}) =\int_{\mathbb{R}}u^{q+1}_{c,\infty}(x)\,dx=\frac{2(p-1)^{1/p}(q+1)^{1/p}}{p^{1/p-1}(p+(p-1)(q+1))},
  $$
  and thus it holds that
 \begin{align}
C_{q,\infty,p}=&
\left(\frac{(p-1)(q+1)}{p}\right)^{\frac{1}{p+(p-1)(q+1)}}M_c^{-\frac{p}{p+(p-1)(q+1)}}\nonumber\\
=&\left(\frac{p+(p-1)(q+1)}{2p}\right)^{\frac{p}{p+(p-1)(q+1)}}.
\end{align}
  This completes the proof of Proposition \ref{prop42}.
\end{proof}

\subsection{Main theorem for the best constant of $L^{\infty}$-type G-N inequality}
In this subsection, we utilize the results from above sections to prove the following Theorem \ref{Linftyhightdnagy1}. Particularly, the closed form solution is obtained in one dimensional case.
\begin{thm}\label{Linftyhightdnagy1}
Suppose $p>d\geq 1$, $q\geq 0$, $u\in L^{q+1}(\mathbb{R}^d)$ and $\nabla u\in L^{p}(\mathbb{R}^d)$. Then $u\in L^{\infty}(\mathbb{R}^d)$ and it satisfies the following inequality
\begin{eqnarray}\label{LinftyMlinftyq}
\|u\|_{L^{\infty}}\leq C_{q,\infty,p}\|u\|_{L^{q+1}}^{1-\theta}\|\nabla u\|_{L^p}^{\theta},\quad \theta=\frac{pd}{dp+(p-d)(q+1)},
\end{eqnarray}
where the best constant
\begin{eqnarray}\label{bestconstant}
C_{q,\infty,p}=\theta^{-\frac{\theta}{p}}(1-\theta)^{\frac{\theta}{p}}M_c^{-\frac{\theta}{d}}, ~~M_c=\int_{\mathbb{R}^n}|u_{c,\infty}|^{q+1}\,dx.
\end{eqnarray}
Here $u_{c,\infty}$ is the unique radial solution as described by the following two cases:
\begin{itemize}
\item if $q<p-1$, $u_{c,\infty}$ is the unique non-increasing solution of the free boundary problem (\ref{iMinfFBVP1})-(\ref{iMinfFBVPbound})
 and $u(r)=0$ for $r\geq R$.
 \item if $q\geq p-1$, $u_{c,\infty}$ is the unique positive solution to the problem (\ref{iMinfchachy1})-(\ref{iMinfchachybound}).
 \end{itemize}
Moreover, the case of equality holds if $u= u_{c,\infty}(\lambda|x-x_0|)$ for any $\lambda >0$, $x_{0}\in \mathbb{R}^d$.
\end{thm}

\begin{proof}

From Proposition \ref{iLcriticalequa}, Proposition \ref{icor} and Lemma \ref{icor21}, we have that any critical point of $G(u)$ in $Y_{rad}^*$ satisfies the problem (\ref{idistrequ})-(\ref{idistrbound}) up to a re-scaling. Conversely, in $Y_{rad}^*$, any non-negative solution $u$ of the problem (\ref{idistrequ})-(\ref{idistrbound}) is also a critical point of $G(u)$. Moreover any re-scaling function of $u$ is still a critical point of $G(u)$ in $Y_{rad}^*$ by Proposition \ref{pdeuiscri}.

By Theorem \ref{existence}, Theorem \ref{uniqueness} and Proposition \ref{prop42}, we know that the problem (\ref{idistrequ})-(\ref{idistrbound}) has a unique solution $u(r)$ for the case $q>0$. While for the case $q=0$, from Proposition \ref{prop2} we know that there is a unique close form solution $u_{c,\infty}(r)$ satisfying the free boundary problem (\ref{iMinfFBVP1})-(\ref{iMinfFBVPbound}). Thus the critical point of $G(u)$ is unique up to a re-scaling.

 Hence any re-scaling function set of $u_{c,\infty}$, $\big\{u_{\lambda}\big\}_{\lambda>0}$ contains all critical points of $G(u)$ and $G(u_{\lambda})\equiv G(u_{c,\infty})$. Notice that $G(u)$ do not have maximum. Hence all critical points  $\big\{u_{\lambda}\big\}_{\lambda>0}$ are minimizers of $G(u)$.

Next we derive the best constant $C_{q,\infty,p}$ for $q\geq 0$.
Since the solution $u_{c,\infty}(r)$ to the problem (\ref{idistrequ})-(\ref{idistrbound}) is a minimizer of $G(u)$. Hence from the problem (\ref{iradminiexi}) and the formula (\ref{Gre=0}), we have
$$
\alpha=\|u_{c,\infty}\|^{1-\theta}_{L^{q+1}}\|\nabla u_{c,\infty}\|^{\theta}_{L^{p}}=\left(\frac{\theta}{1-\theta}\right)^{\frac{\theta}{p}}\left(\|u_{c,\infty}(r)\|^{q+1}_{L^{q+1}}\right)^{\frac{\theta}{d}},
$$
Notice $C_{q,\infty,p}=\alpha^{-1}$, thus we have (\ref{bestconstant}).
\end{proof}

\section{Euler-Lagrange equations for the $L^m$-Type  G-N inequality}

In this section, we derive the best constant of the $L^m$-type  G-N inequality.
Following standard approach, the minimization problem is provided in the following solution space
\begin{eqnarray}\label{X}
X=\left\{u|~u\in L^{q+1}(\mathbb{R}^d),~~\nabla u\in L^p(\mathbb{R}^d) \right\},
\end{eqnarray}
and we prove that there is a positive constant $\beta$ satisfies
\begin{eqnarray}\label{betamini}
\beta=\inf_{u\in X}J(u),
\end{eqnarray}
where
\begin{eqnarray}\label{Jh}
J(u):=\frac{\left(\int_{\mathbb{R}^d}|u|^{q+1}\,dx\right)^{\frac{\gamma-p/2}{q+1}}\int_{\mathbb{R}^d}|\nabla u|^p\,dx}{\left(\int_{\mathbb{R}^d}|u|^{m+1}\,dx\right)^{\frac{\gamma+p/2}{m+1}}},
\end{eqnarray}
and
\begin{eqnarray}
\gamma=\frac{p(m+1)(q+1)+(p/2-1)d(m+q)-dqm+(p-1)d}{d(m-q)}.\label{a}
\end{eqnarray}
Again thanks to the rearrangement technique, the minimized problem (\ref{betamini}) is equivalent to the following minimized problem
 \begin{eqnarray}\label{radminiexi}
\beta=\inf_{u\in X^*_{rad}}J(u),
\end{eqnarray}
where $X^*_{rad}$ is a non-negative radial symmetric decreasing function space given by
$$
X^*_{rad}=\{u\geq 0\big|~u(x)=u(|x|),\,\, u'(r)\leq 0,~a.e.,\,\, u\in L^{q+1}(\mathbb{R}^d),\,\, \nabla u\in L^p(\mathbb{R}^d)\}.
$$

General strategy of
the derivation of the Euler-Lagrange equation for the critical points of $J(u)$ is similar to that of $L^{\infty}$-type G-N inequality in Section 2. However, there are some fine differences:
\begin{enumerate}[(i)]
\item Existence of a minimizer of the functional $J(u)$ can be directly obtained by the compactness in $L^m$ for $m<\infty$. See Proposition \ref{h0minimizer} below. As a direct consequence, we can obtain existence of solutions to the Euler-Lagrange equations.
\item For $m<\infty$, solutions to the corresponding Euler-Lagrange equation have no singularity at $r=0$. We use interior elliptic regularity to show that $u'(0)=0$. See Step 3 of the proof of Proposition \ref{barhmini}. In contrast, for the $L^{\infty}$-case, $\lim_{r\to 0^+}u'(r)=-\infty$ by Theorem \ref{existence}.
\item Uniqueness of solutions to the Euler-Lagrange equations can be obtained by verifying the conditions in \cite[Theorem 2]{pucci1998uniqueness} given by Pucci and Serrin.
\item For $q\geq p-1$, positivity of solutions to the problem (\ref{chachy1})-(\ref{chachybound}) can directly given by Compact Support Principle from Pucci and Serrin \cite[Thoerem 1.1.2]{pucci2007maximum}.
\end{enumerate}

Existence of a minimizer of the problem (\ref{radminiexi}) has been proved in the paper \cite[Proposition 2.1]{LW1} for the case $q<p-1$. The proof for the case $q\geq p-1$ is similar to that of the case $q<p-1$. However some modifications are needed in the proof for the case $q\geq p-1$, and hence we provide those necessary modifications in Appendix A. Here we recall \cite[Proposition 2.1]{LW1} and extend it for both cases: (i) the case $q<p-1$ and (ii) the case $q\geq p-1$.
\begin{prop}\label{h0minimizer}Assume that parameters $p,q$ and $m$ satisfy (\ref{qandm}). Then there exists a minimizer $u_0$ of $J(u)$ in $X^*_{rad}$ such that
\begin{eqnarray}\label{minimizer}
J(u_0)=\beta=\inf_{u\in X^*_{rad}} J(u), \quad \|u_0\|_{L^{q+1}}=\|u_0\|_{L^{m+1}}=1,\,\,\|\nabla u_0\|^p_{L^p}=\beta.
\end{eqnarray}
\end{prop}

The structure of the rest part of this section is similar to that of Section 2. For the parameter range (\ref{qandm}), in Subsections 5.1 we derive a Pohozaev type identity and use it to prove that the contact angle is zero for the case $q<p-1$, and thus we derive the Euler-Lagrange equation for critical points of $J(u)$ in $X^*_{rad}$ for any $q> 0$. In Subsection 5.2, we show respectively existence and uniqueness of the problem (\ref{FBVP1})-(\ref{FBVPbound}) and the problem (\ref{chachy1})-(\ref{chachybound}). In Subsection 5.3, we obtain the best constant for the $L^m$-type G-N inequality in Theorem \ref{hightdnagy1} below. Finally, we make some remarks on some non-radial solutions in Subsection 5.4.

\subsection{Euler-Lagrange equations}
We derive the Euler-Lagrange equations (\ref{FBVP1})-(\ref{FBVPbound}) for $q<p-1$ and (\ref{chachy1})-(\ref{chachybound}) for $q\geq p-1$ by the two steps: (i) derive the equation (\ref{xequation}) and the boundary condition (\ref{xinitial}) in Proposition \ref{barhmini}; (ii) prove a zero contact angle condition $u'(R)=0$ from Lemma \ref{lmfree0} and Lemma \ref{eneangle} for $q<p-1$.
\begin{prop}\label{barhmini}
Let $\bar u(x)\in X^*_{rad}$ be a critical point of $J(u)$. Then there exist $\mu_2, \lambda_2>0$ such that $ u(x)=\frac{1}{\mu_2} \bar u(\frac{1}{\lambda_2} x)$ satisfies
\begin{enumerate}[(i)]
\item the initial value problem, for some $\alpha>1$
\begin{eqnarray}
&&\nabla\cdot(|\nabla u|^{p-2}\nabla u)+ u^m= u^q,\,\,  \mbox{ in }\Omega,\label{xequation}\\
&& u(0)=\alpha,\quad  u'(0)=0,\label{xinitial}
\end{eqnarray}
where $\Omega=\{x\in \mathbb{R}^d,~|x|<R\}$, $R:=\inf\{r>0|~u(r)=0\}\in \mathbb{R}^{+}\cup \{+\infty\}$.
\item If $p>1$, $0<q<p-1$ and $q<m$, then there is a finite point $R\in (0,\infty)$ such that $ u(R)=0$.
\item If $p>1$, $q\geq p-1$ and $q<m$, then solutions to (\ref{xequation})-(\ref{xinitial}) are positive in $[0,\infty)$.
\end{enumerate}
\end{prop}
\begin{proof}
{\bf Step 1.} Re-scaling and variation

For any $\mu_1, \lambda_1>0$, we re-scale $\bar u(x)$ as $u_1(y):=\frac{1}{\mu_1}\bar u(\frac{y}{\lambda_1})$. Noticing the scaling invariant of $J(u)$ for $u_1=\frac{1}{\mu_1}\bar u(\frac{y}{\lambda_1})$, we have
 \begin{eqnarray}\label{Jinvarant}
J(u_1)=J(\bar u).
\end{eqnarray}
Hence if $\bar u(x)\in X^*_{rad}$ is a critical point of $J(u)$, then we obtain that $u_1(y)$ is also a critical point ($\frac{\delta J(u_1)}{\delta u}=0$).
Below we choose $\lambda_1, ~\mu_1$ such that it holds that
\begin{eqnarray}\label{a1}
\|u_1\|_{L^{q+1}}=\|u_1\|_{L^{m+1}}=1,\quad \|\nabla u_1\|^p_{L^p}=:a_1.
\end{eqnarray}

Since $u_1 \in X^*_{rad}$, it is observed that $u_1(r)$ is continuous in $(0,\infty)$.
Denote the support of $u_1$ as $\Omega_1:=\{x\in\mathbb{R}^d|\, u_1>0\}$. Since $u_1(r)$ is a radial decreasing function, one knows that $0\in \Omega_1$ and $\Omega_1=\{x\in \mathbb{R}^d,~|x|<R_1\}$, where $R_1:=\inf\{r>0|u_1(r)=0\}\in \mathbb{R}^{+}\cup \{+\infty\}$.
For any radial symmetric function $\phi\in C_0^{\infty}(\Omega_1)$\footnote{Notice that for the case $m=\infty$, the admissible variation $\phi$ requests $\phi(0)=0$ in Section 2. As a consequence, there is a singularity at $x=0$. Hence the term $a\delta_{x=0}$ appears in the Euler-Lagrange equations (\ref{iinftyfunction}).},
we can show that $\phi$ be an admissible variation at $u_1$, i.e., there is a $\varepsilon_0>0$ such that for any $|\varepsilon|<\varepsilon_0$ one has $u_1+\varepsilon \phi\geq 0$. Then from a direct computation and using \eqref{a1}, we have for $\forall~ \phi \in C_c^{\infty}(\Omega_1)$
$$
\frac{d}{d\varepsilon}\Big|_{\varepsilon=0}J(u_1+\varepsilon \phi)=\int_{\mathbb{R}^d}p(|\nabla u_1|^{p-2}\nabla u_1)\cdot \nabla \phi +\left((\gamma-\frac{p}{2})a_1 u_1^q-(\gamma+\frac{p}{2})a_1 u_1^m\right)\phi\, dx = 0.
$$
This implies that $u_1$ satisfies the following generalized Lane-Emden equation
\begin{eqnarray}\label{equationh0}
\nabla\cdot(|\nabla u_1|^{p-2}\nabla u_1)-\frac{\gamma-p/2}{p}a_1u_1^q+\frac{\gamma+p/2}{p}a_1 u_1^m=0 \quad \mbox{ in } \Omega_1.
\end{eqnarray}

{\bf Step 2. }Normalization

We re-scale function $u_1$ as $ u(y)=\frac{1}{\mu}u_1(\frac{y}{\lambda})$, where $\mu, \lambda>0$ will be defined in (\ref{mulambda}). From (\ref{Jinvarant}), we know that $ u$ is also a critical point of $J(u)$ in $X^*_{rad}$. From (\ref{equationh0}) we deduce that $u$ satisfies the following equation in $\Omega:=\lambda\Omega_1$, $R:=\lambda R_1$
$$
\mu^{p-1} \lambda^p \nabla\cdot(|\nabla u|^{p-2}\nabla u)-\frac{\gamma-\frac{p}{2}}{p}a_1\mu^q u^q+\frac{\gamma+\frac{p}{2}}{p}a_1 \mu^m u^m=0, \mbox{ in } \Omega.
$$
Taking
\begin{eqnarray}\label{mulambda}
\mu=\left(\frac{\gamma-p/2}{\gamma+p/2}\right)^{1/(m-q)}~~(m>q),\quad
 \lambda =\left(\frac{\gamma-p/2}{p}\right)^{1/p}\mu^{\frac{q-p+1}{p}}a_1^{1/p},
\end{eqnarray}
one has
\begin{eqnarray}\label{para1}
\mu^{p-1} \lambda^p=\frac{\gamma-p/2}{p}a_1 \mu^q=\frac{\gamma+p/2}{p}a_1 \mu^m.
\end{eqnarray}
Hence $ u$ satisfies
\begin{eqnarray}\label{steadyequa}
\nabla\cdot(|\nabla u|^{p-2}\nabla u)- u^q+  u^m=0,~~ \mbox{ in } \Omega.
\end{eqnarray}

{\bf Step 3.} In this step, we prove that the critical point $ u$ satisfies the following initial value problem
\begin{eqnarray}
&&(|u'|^{p-2}u')' + \frac{d-1}{r} |u'|^{p-2}u' + u^{m} = u^{q},\quad  0<r<R, \label{steadyequ}\\
&&u(0)=\alpha,\quad u'(0)=0,\label{initial}
\end{eqnarray}
for some $\alpha>1$.

First we claim that
$u(0)>1$. If not, we have $ u(0)\leq 1$. The decreasing property of $ u$ in $r$ implies that for any fixed $R_0>0$, $x\in B(0,R_0)$, $ u(x)\leq 1$. Thus we have
\begin{eqnarray}\label{equa1}
\nabla\cdot(|\nabla u|^{p-2}\nabla u)= u^q(1- u^{m-q})\geq 0,\quad \mbox{ in } B(0,R_0),
\end{eqnarray}
which implies that the maximum of $\bar u$ is reached at $|x|=R_0$ by the maximum principles for divergence structure elliptic differential inequalities \cite[Theorem 3.2.1]{pucci2007maximum}. Since $u$ is the radial decreasing continuous function, it holds that $ u(x) \equiv  u(R)$ in $B(0,R)$ for any $R>0$. Plugging this constant solution into (\ref{equa1}), one knows that $ u(x ) \equiv 1$ in $B(0,R)$ for any $R>0$. Hence $ u(x ) \equiv 1$ in $\mathbb{R}^d$.
It is contradictory to the integrability of $u$.

Next, we will prove $ u$ is a smooth function at $r=0$, and satisfies $u'(0)=0$.

In fact, since $u$ is a radial decreasing, continuous, integrable function and $ u(0)>1$, there exists a $R_0$ such that $u(R_0)=1$. Let $\tilde{u}= u -1$. Then $\tilde{u}\in L^{\infty}(B(0, R_0))$ satisfies the following equation
\begin{eqnarray}\label{tilde h equa}
&&\nabla\cdot(|\nabla \tilde u|^{p-2}\nabla \tilde u) =(\tilde u+1)^q-(\tilde u+1)^m \in L^{\infty}(B(0, R_0)),\\
&&\tilde{u}(x)=0,\quad |x|=R_0.
\end{eqnarray}
In \cite{tolksdorf1984regularity}, 
 the $C_{loc}^{1}(B(0, R_0))$ regularity result is assured to the weak solutions $u\in W^{1,p}(B(0, R_0))\cap L^{\infty}(B(0, R_0)) $ of \eqref{tilde h equa}.
Thus we denote the peak value as $\alpha= u(0)>1$, and we have $ u'(0)=0$.

{\bf Step 4.} There exists a finite $R$ such that $ u(R)=0$ for $q<p-1$.

For a radial decreasing non-negative function $u\in X^*_{rad}$, there only exist two cases: (i) there exists a finite $R$ such that $u(R)=0$; (ii)  $u(r) >0$ for all $r\geq 0$, and hence $u(r)\to 0$, $ u'(r)\to 0$ as $r\to \infty$.

 Similar to Step 3 in the proof of Proposition \ref{iLcriticalequa}, we show that the second case can not happen by a contradiction.
Indeed, if (ii) holds, then $ u> 0$ is a solution to the following problem
 \begin{eqnarray}\label{equa2}
&& (|u'|^{p-2}u')' + \frac{d-1}{r} |u'|^{p-2}u' + u^{m} = u^{q},\label{equa2}\\
&& u'(0)=0,\quad \lim_{r\to \infty}u(r)=0.\label{equa2bound}
\end{eqnarray}
Repeating the process to obtaining (\ref{iuderi}), we have
\begin{eqnarray}\label{deri}
 \frac{p-1}{p}|u'|^p+\frac{u^{m+1}}{m+1}-\frac{u^{q+1}}{q+1}\geq 0.
\end{eqnarray}
By using (\ref{deri}) and $u'(0)=0$, we deduce $u(0)\geq\alpha_c:=\left( \frac{m+1}{q+1}\right)^{\frac{1}{m-q}}>0$.
Thus there is a $R_0\geq 0$ such that $u(R_0)=\alpha_c$ and $0<u\leq \alpha_c$ for $r\geq R_0$.
Hence when $r\geq R_0$, we solve from (\ref{deri})
 \begin{eqnarray}\label{uderi}
 u'(r)\leq -\left(\frac{p}{p-1}\right)^{1/p}\left(\frac{u^{q+1}(r)}{q+1}-\frac{u^{m+1}(r)}{m+1}\right)^{1/p}.
\end{eqnarray}
Using the method of separation of variables for (\ref{uderi}) and integrating the result inequality from $R_0$ to $r$, $r\in (R_0,+\infty)$, we obtain
 \begin{eqnarray*}
 \int_{\alpha_c}^{u(r)}\frac{du}{\left(\frac{u^{q+1}}{q+1}-\frac{u^{m+1}}{m+1}\right)^{1/p}}\leq \left(\frac{p}{p-1}\right)^{1/p}(R_0-r) \quad \mbox{for all }r> R_0,
\end{eqnarray*}
which gives
\begin{eqnarray}\label{rcontrol}
 r\leq R_0+\left(\frac{p-1}{p}\right)^{1/p}\int_{u(r)}^{\alpha_c}\frac{du}{\left(\frac{u^{q+1}}{q+1}-\frac{u^{m+1}}{m+1}\right)^{1/p}} \quad \mbox{for all }r\geq R_0.
\end{eqnarray}
Since $0<u(r)\leq \alpha_c$ for $r\geq R_0$, we have that the right side of (\ref{rcontrol}) satisfies
\begin{eqnarray*}
\int_{u(r)}^{\alpha_c}\frac{du}{\left(-\frac{u^{m+1}}{m+1}+\frac{u^{q+1}}{q+1}\right)^{\frac{1}{p}}}
\leq\int_{0}^{\alpha_c}\frac{du}{\left(-\frac{u^{m+1}}{m+1}+\frac{u^{q+1}}{q+1}\right)^{\frac{1}{p}}} .\end{eqnarray*}
A simple computation leads to
\begin{eqnarray}\label{integration1}
\int_{0}^{\alpha_c}\frac{du}{\left(-\frac{u^{m+1}}{m+1}+\frac{u^{q+1}}{q+1}\right)^{\frac{1}{p}}}
=\frac{(m+1)^{\frac{p-q-1}{p(m-q)}}(q+1)^{\frac{1}{p}-\frac{p-q-1}{p(m-q)}}}{m-q}
\mathcal{B}\left(\frac{p-1}{p},\frac{p-(q+1)}{p(m-q)}\right).
\end{eqnarray}
Noticing that $p>1$, $q<p-1$ and $m>q$, we have that $\mathcal{B}\left(\frac{p-1}{p},\frac{p-(q+1)}{p(m-q)}\right)$ is finite. Taking $r\to +\infty$, we obtain a contradiction from (\ref{rcontrol}) and (\ref{integration1}). Hence the second case can not happen, i.e., there exists a finite $R$ such that $u(R)=0$.

{\bf Step 5.} Positivity for $q\geq p-1$.

  The positivity of solutions to (\ref{xequation})-(\ref{xinitial}) is a direct consequence of Compact Support Principle from Pucci and Serrin \cite[Thoerem 1.1.2]{pucci2007maximum}. Here we only need to verify the necessary and sufficient condition for the Compact Support Principle: in the same notations as that in \cite{pucci2007maximum}, $H(s)=\frac{p-1}{p}s^p$, $F(s)=\frac{u^{q+1}}{q+1}-\frac{u^{m+1}}{m+1}$,  $\int_{0^+}\frac{ds}{H^{-1}(F(s))}<\infty$ if and only if $q< p-1$.

\end{proof}
We construct an auxiliary energy functional
\begin{eqnarray}\label{steadyfree0}
\mathcal{F}( u):=\frac{d(p-1)+p}{p}\int_{\mathbb{R}^d}|\nabla u|^p\,dx-\frac{dm}{m+1}\int_{\mathbb{R}^d} u^{m+1}\,dx+\frac{dq}{q+1}\int_{\mathbb{R}^d}u^{q+1}\,dx.
\end{eqnarray}
In the following two lemmas, we prove that for the case $q<p-1$, solutions to (\ref{xequation})-(\ref{xinitial}) have a zero contact angle at the boundary of the compact support.

\begin{lem}\label{lmfree0}
Let $\bar u(x)\in X^*_{rad}$ be a critical point of $J(u)$. Then there are $\mu_2, \lambda_2>0$ such that the re-scaling function $u(x)=\frac{1}{\mu_2} \bar u(\frac{1}{\lambda_2} x)$ is a zero point of the energy functional defined in (\ref{steadyfree0}), i.e.,
\begin{eqnarray}\label{fre=0}
\mathcal{F}(u)=0.
\end{eqnarray}
\end{lem}
\begin{proof}
From (\ref{a1}), the re-scaling function $u_1(y)=\frac{1}{\mu_1}\bar u(\frac{y}{\lambda_1})$, $\mu_1, \lambda_1>0$ satisfies
\begin{eqnarray}
1=\int_{\mathbb{R}^d}u_1^{q+1}\,dx,~~1=\int_{\mathbb{R}^d}u_1^{m+1}\,dx,~~a_1=\int_{\mathbb{R}^d}|\nabla u_1|^p\,dx.\label{3}
\end{eqnarray}
Let $ u(y)=\frac{1}{\mu}u_1(\frac{y}{\lambda})$, where $\mu, \lambda>0$ was defined in (\ref{mulambda}). Then by (\ref{3}) we obtain
\begin{eqnarray}\label{1}
\int_{\mathbb{R}^d} u^{q+1}\,dy=\frac{\lambda^d}{\mu^{q+1}},\quad \int_{\mathbb{R}^d} u^{m+1}\,dy=\frac{\lambda ^d}{\mu^{m+1}},\quad \int_{\mathbb{R}^d}|\nabla u|^p\,dy=\frac{a_1\lambda ^{d-p}}{\mu^p}.
\end{eqnarray}
Hence
\begin{align*}
F(u)&=\frac{d(p-1)+p}{p}\int_{\mathbb{R}^d}|\nabla u|^p\,dy-\frac{dm}{m+1}\int_{\mathbb{R}^d} u^{m+1}\,dy+\frac{dq}{q+1}\int_{\mathbb{R}^d} u^{q+1}\,dy\\
&=\frac{\lambda^d}{\mu^{q+1}}\left(\left(\frac{d(p-1)}{p}+1\right)\frac{a_1}
{\mu^{p-q-1}\lambda^{p}}-\frac{dm}{m+1}\frac{1}{\mu^{m-q}}+\frac{dq}{q+1}\right).
\end{align*}
Using (\ref{mulambda}) and (\ref{para1}), we have
\begin{align*}
&\left(\frac{d(p-1)}{p}+1\right)\frac{a_1}{\mu^{p-q-1}\lambda^{p}}-\frac{dm}{m+1}\frac{1}{\mu^{m-q}}+\frac{dq}{q+1}\\
=&\left(\frac{d(p-1)}{p}+1\right)\frac{p}{\gamma-p/2}-\frac{dm}{m+1}\frac{\gamma+p/2}{\gamma-p/2}+\frac{dq}{q+1}\\
=&\frac{1}{\gamma-\frac{p}{2}}\left(d(p-1)+p-\frac{dm}{m+1}(\gamma+\frac{p}{2})+\frac{dq}{q+1}(\gamma-\frac{p}{2})\right)=0,
\end{align*}
where the last equality used the definition (\ref{a}) of $\gamma$. Hence we obtain (\ref{fre=0}).

\end{proof}

Using (\ref{mulambda}) and (\ref{1}), a simple computation gives the following Corollary.
\begin{cor}\label{cor1}
Let $\bar u(x)\in X^*_{rad}$ be a critical point of $J(u)$. Then for $\mu_2, \lambda_2>0$ defined in Lemma \ref{lmfree0} the re-scaling function $u(x)=\frac{1}{\mu_2} \bar u(\frac{1}{\lambda_2} x)$ satisfies the following equalities
\begin{eqnarray}
&&\int_{\mathbb{R}^d}  u^{q+1}\,dy=p^{-\frac{d}{p}}(\gamma-\frac{p}{2})^{\frac{d}{p}-\frac{((p-1)d+p)-q(d-p)}{p(m-q)}}(\gamma+\frac{p}{2})^{\frac{((p-1)d+p)
-q(d-p)}{p(m-q)}}a_1^{\frac{d}{p}},\label{barformula1}\\
&&\int_{\mathbb{R}^d} u^{m+1}\,dy=p^{-\frac{d}{p}}(\gamma-\frac{p}{2})^{\frac{d}{p}-\frac{((p-1)d+p)-q(d-p)}{p(m-q)}-1}
(\gamma+\frac{p}{2})^{1+\frac{((p-1)d+p)-q(d-p)}{p(m-q)}}a_1^{\frac{d}{p}},\label{barformula2}\\
&&\int_{\mathbb{R}^d}|\nabla u|^p\,dy=p^{1-\frac{d}{p}}(\gamma-\frac{p}{2})^{\frac{d}{p}-1-\frac{((p-1)d+p)-q(d-p)}{p(m-q)}}
(\gamma+\frac{p}{2})^{\frac{((p-1)d+p)-q(d-p)}{p(m-q)}}a_1^{\frac{d}{p}}.\label{barformula3}
\end{eqnarray}
\end{cor}
\begin{lem}\label{eneangle}
Let $ u$ be a solution to the initial value problem (\ref{xequation})-(\ref{xinitial}). Assume that $u$ has a touchdown point $R$ {\rm(}i.e. $u(R)=0${\rm )}. Then
the following relation between the auxiliary energy functional $\mathcal{F}(u)$ and the contact angle holds
\begin{eqnarray}\label{frvsangle}
\mathcal{F}( u) = \frac{d(p-1) |B(0,R)|}{p} |u'(R)|^p.
\end{eqnarray}
\end{lem}
\begin{proof}
Similar to the proof of Lemma \ref{ieneangle}, we still need to introduce a function
\begin{eqnarray}\label{Henergy}
H(r):= \frac{p-1}{p}|u'(r)|^p+\frac{u^{m+1}(r)}{m+1}-\frac{u^{q+1}(r)}{q+1}.
\end{eqnarray}
Then by multiplying $u'(r)$ to (\ref{steadyequ}), we have the following energy-dissipation relation
\begin{eqnarray}\label{H}
\frac{dH(r)}{dr} + \frac{d-1}{r} |u'(r)|^{p} = 0.
\end{eqnarray}
Multiplying $r^d$ to (\ref{H}) and integrating the result equation from $0$ to $R$ (at $r=0$, $u(r)$ is well-define and smooth), one has
 $$
 R^d H(R) - d \int_0^R r^{d-1} H(r) \,dr+ (d-1) \int_0^R |u'(r)|^p r^{d-1}\,dr = 0.
 $$
Noticing that $H(R) = \frac{p-1}{p} |u'(R)|^p$ from (\ref{Henergy}), then with some simple computations, it holds that
\begin{align}\label{for1}
\frac{p-1}{p} R^d |u'(R)|^p=& d \int_0^R r^{d-1} H(r)\, dr - (d-1) \int_0^R |u'(r)|^p r^{d-1}\,dr\nonumber\\
=& (1-d/p) \int_0^R r^{d-1}| u'(r)|^p\,dr
 + \frac{d}{(m+1)} \int_0^R r^{d-1} u^{m+1}(r)\,dr\nonumber\\
& - \frac{d}{q+1} \int_0^R r^{d-1} u^{q+1}(r)\,dr\nonumber\\
=& (1-d/p) \frac{1}{S_d}\int_{B(0,R)} |\nabla u|^p\,dx
 + \frac{d}{(m+1)}\frac{1}{S_d} \int_{B(0,R)} u^{m+1}\,dx\nonumber\\
&- \frac{d}{q+1} \frac{1}{S_d}\int_{B(0,R)} u^{q+1}\,dx.
\end{align}
On the other hand, multiplying $u$ to (\ref{steadyequa}) and integrating in $B(0,R)$, we get
\begin{eqnarray}\label{for2}
\int_{B(0,R)} u^{q+1}\,dx = -  \int_{B(0,R)} |\nabla u|^p\,dx +  \int_{B(0,R)} u^{m+1}\,dx.
\end{eqnarray}
By (\ref{for1}) and (\ref{for2}), we have
\begin{align}\label{onehand}
\frac{p-1}{p} S_d R^d |u'(R)|^p= &\frac{(p-d)(q+1)+pd}{p(q+1)} \int_{B(0,R)} |\nabla u|^p\,dx\nonumber\\
 &\qquad - \frac{d(m-q)}{(m+1)(q+1)}  \int_{B(0,R)} u^{m+1}\,dx.
\end{align}
Moreover, using (\ref{steadyfree0}) and (\ref{for2}), we know that
\begin{eqnarray}\label{otherhand}
\mathcal{F}(u)=\frac{(p-d)(q+1)+pd}{p(q+1)} \int_{B(0,R)} |\nabla u|^p \, dx- \frac{d(m-q)}{(m+1)(q+1)}  \int_{B(0,R)} u^{m+1}\,dx.
\end{eqnarray}
Thus (\ref{onehand}) and (\ref{otherhand}) give (\ref{frvsangle}).
\end{proof}

Finally, we show that the minimizer of $J(u)$ in $X_{rad}^*$ is a solution of the free boundary problem (\ref{FBVP1})-(\ref{FBVPbound}) up to a re-scaling.
\begin{prop}\label{cor}Assume $p>1$, $0<q<p-1$ and $m>q$.
Let $\bar u(x)\in X^*_{rad}$ be a critical point of $J(u)$. Then there are $\mu,~\lambda>0$ such that the re-scaling function $ u(y)=\frac{1}{\mu} \bar u(\frac{1}{\lambda }y)$ satisfies the free boundary problem (\ref{FBVP1})-(\ref{FBVPbound}).
\end{prop}
\begin{proof}
As a direct consequence of (\ref{frvsangle}) and $\mathcal{F}( u)=0$, one knows that $u'(R)=0$. In the other words, the contact angle is zero.
\end{proof}

\subsection{Existence and uniqueness for the Euler-Lagrange equations}

 First, we prove existence and uniqueness for the FBP (\ref{FBVP1})-(\ref{FBVPbound}), $q<p-1$ in the following theorem.
\begin{thm}(Existence and uniqueness)\label{remunique}
Assume that exponents $p,m$ and $q$ satisfy the condition (\ref{qandm}) and $q<p-1$.
Then there is a unique solution $u(r)$ to the free boundary problem (\ref{FBVP1})-(\ref{FBVPbound}) in $X_{rad}^*$ and satisfies $\alpha:=u(0)>1$, $u(r)>0$ and $u'(r)<0$ for $0<r<R$.
\end{thm}
\begin{proof}
From Proposition \ref{cor}, we know that the re-scaling function $u$ of the minimizer $ \bar u$ of the functional $J(u)$ in $X_{rad}^*$ is a solution to the free boundary problem (\ref{FBVP1})-(\ref{FBVPbound}). Hence existence was proved.

Uniqueness for the case $\sigma>1$, $0<q<\sigma-1$ and $q<m<\sigma$ is proved by \cite[Theorem 2]{pucci1998uniqueness}. Since $p>\frac{2d}{d+2}$, we have $\sigma>1$. Consequently, we obtain that $u$ is a unique solution to the problem (\ref{FBVP1})-(\ref{FBVPbound}) under the assumptions of Theorem \ref{remunique}.

From (i) of Proposition \ref{barhmini}, we have $\alpha=u(0)>1$.

Finally, we prove $ u'(r)<0$ for $0<r<R$. The method for proving this property is standard \cite{Peletier1983}. For completeness, we provide a simple proof below.
If it is not true, then there exists (choice to be the first one) $r_0\in (0,R)$ such that $u'(r_0)=0$.

Since $u\in X_{rad}^*$ and $u(R)=0$ , we can show that the extremum points are never reached at the points of $ u(r)=1$. Therefore all extremum points must be either local minimum points ($ u''(r)>0$) or local maximum points ($ u''(r)<0$).

Solving the equation (\ref{FBVP1}) gives
\begin{eqnarray}\label{hhhfirstder}
  u'(r)=-\left(\frac{1}{r^{d-1}}\int_0^r s^{d-1} ( u^m(s)- u^{q}(s))\,ds\right)^{\frac{1}{p-1}}\quad \mbox{ for any } r\in (0,R).
\end{eqnarray}
Thus for any $r$ satisfying $u(r)\geq 1$, (\ref{hhhfirstder}) implies $u'(r)< 0$. Since $u(0)=\alpha_c>1$, the first extremum point $r_0>0$ must be a local minimum point, and satisfies $u(r_0)<1$.

 Since the extremum points must be either local minimum points or local maximum points,
 the next extremum point after $r_0$ must be a local maximum point, denote $r_1$, which satisfies $ u(r_1)>1$. Continuously, we can order all extremum points as a sequence $r_0,~ r_1, ~r_2, \cdots$, and at those points $ u(r)$ satisfies
  \begin{eqnarray}
 &&u(r_{2(k-1)})<1, \mbox{ for } k=1,2, \cdots\label{minimum}\\
 && u(r_{2k-1})>1, \mbox{ for } k=1,2, \cdots.\label{maximum}
\end{eqnarray}

 Now we claim that the sequence $\{ u(r_{2(k-1)})\}_{k=1}^\infty$ of local minimum is increasing and the sequence $\{u(r_{2k-1})\}_{k=1}^\infty$ of local maximum is decreasing.

From (\ref{H}), we deduce
 $$
 H(r_{2(k-1)})>H(r_{2k}),
 $$
 i.e.,
 $$
 \frac{1}{m+1} u^{m+1}(r_{2(k-1)} ) -  \frac{1}{q+1} u^{q+1}(r_{2(k-1)} )>\frac{1}{m+1} u^{m+1}(r_{2k} ) -  \frac{1}{q+1} u^{q+1}(r_{2k} ).
 $$
Notice that $f(s)=\frac{1}{m+1} s^{m+1} - \frac{1}{q+1} s^{q+1}$ is decreasing function for $s\in (0,1)$, and $f(s)$ is increasing function for $s\in (1,\infty)$. Hence (\ref{minimum}) indicates that $ u(r_{2(k-1)} )< u(r_{2k} )$. In a similar way, we have $ u(r_{2k+1})< u(r_{2k-1})$ by (\ref{maximum}). In other words, the solution is oscillatory around value $u=1$ and has decreasing amplitude, and hence it will never touchdown. This is a contradiction to the fact $u\in L^{q+1}(\mathbb{R}^d)$.
Hence $ u>0,~ u'(r)<0$ before touching down at $r=R$.

\end{proof}
Next we show existence and uniqueness for the problem (\ref{chachy1})-(\ref{chachybound}), $q\geq p-1$.
\begin{thm}(Existence and uniqueness)\label{remunique1}
Assume that exponents $p,m$ and $q$ satisfy the conditions (\ref{qandm}) and $q\geq p-1$.
Then there is a unique solution $u(r)$ to the problem (\ref{chachy1})-(\ref{chachybound}) and satisfies $\alpha:=u(0)>1$, $u(r)>0$ and $u'(r)<0$ for $0<r<\infty$.
\end{thm}
\begin{proof}
From Proposition \ref{barhmini}, we know that the critical points of $J(u)$ in $X^*_{rad}$ are solutions satisfying the initial value problem (\ref{xequation})-(\ref{xinitial}). Hence the minimizer $ u$ of $J(u)$ in $X^*_{rad}$ is a solution to (\ref{chachy1})-(\ref{chachybound}) and satisfies $\alpha>1$ and $ u(r)\to 0$ as $r\to \infty$.

Similar to the proof of Theorem \ref{remunique}, we can verify that the assumptions of this theorem exactly satisfies all the conditions in \cite[Theorem 2 ]{pucci1998uniqueness}, hence the solution $u$ to (\ref{chachy1})-(\ref{chachybound}) is unique.

By Compact Support Principle \cite{pucci2007maximum} (see Step 5 of the proof of Proposition \ref{barhmini}), we know $u(r)>0$ in $[0,+\infty)$. Since $ u\in X_{rad}^*$ and $ u(r)\to 0$ as $r\to \infty$, a similar proof to Theorem \ref{remunique} gives that $u'(r)<0$ for $0<r<\infty$.
\end{proof}

\begin{cor}
 The set of critical points of $J(u)$ is consisted by minimizers of $J(u)$. Moreover, minimizers are unique up to a re-scaling.
\end{cor}
\begin{proof}
Without loss of generality, we only prove this Corollary for the case $q<p-1$. Proposition \ref{cor} tells us that all critical points of $J(u)$ satisfy the free boundary problem (\ref{FBVP1})-(\ref{FBVPbound}) up to a re-scaling. By Theorem \ref{remunique}, we know that the free boundary problem (\ref{FBVP1})-(\ref{FBVPbound}) has at most a solution in $X^*_{rad}$. Hence critical points are unique up to a re-scaling. From Proposition \ref{h0minimizer}, we know that there is a minimizer for the functional $J(u)$. Thus the set of critical points of $J(u)$ is only consisted by minimizers of $J(u)$.
\end{proof}

\subsection{Best constant of the $L^m$-type G-N inequality}

In this subsection, we utilize the results from previous subsections to derive the best constant of the $L^m$-type G-N inequality, see the following theorem.
\begin{thm} \label{hightdnagy1}
 Suppose that $d\geq 1$, parameters $p, q$ and $m$ satisfy (\ref{qandm}), $u\in L^{q+1}({\mathbb{R}^d})$ and $\nabla u\in L^{p}({\mathbb{R}^d})$. Then $u\in L^{m+1}({\mathbb{R}^d})$ and satisfies the following $L^m$-type G-N inequality
 \begin{eqnarray}\label{h sharp inequality1}
\|u\|_{L^{m+1}}\leq  C_{q,m,p} \|u\|^{1-\theta}_{L^{q+1}}\|\nabla u\|^{\theta}_{L^p},\quad \theta=\frac{pd(m-q)}{(m+1)[d(p-q-1)+p(q+1)]}
\end{eqnarray}
with the best constant
\begin{eqnarray}\label{massvsbeta1}
C_{q,m,p}=\theta^{-\frac{\theta}{p}}(1-\theta)^{\frac{\theta}{p}-\frac{1}{m+1}}M_c^{-\frac{\theta}{d}},\quad M_c=\int_{\mathbb{R}^d}u_{c,m}^{q+1}\,dx.
\end{eqnarray}
Here $u_{c,m}\geq 0$ is described by the following two cases
\begin{enumerate}[(i)]
\item For $q<p-1$, $u_{c,m}$ is the unique solution for the free boundary problem (\ref{FBVP1})-(\ref{FBVPbound}) in $X^*_{rad}$, and satisfies $u_{c,m}(r)>0$, $u_{c,m}'(r)<0$ for $0<r<R$.
\item For $q\geq p-1$, $u_{c,m}$ is the unique positive solution for the problem (\ref{chachy1})-(\ref{chachybound}) in $X^*_{rad}$, and satisfies $u_{c,m}(r)>0$, $u_{c,m}'(r)<0$ for $0<r<+\infty$.
\end{enumerate}

Moreover, the equality holds in (\ref{h sharp inequality1}) if $f=A\, u_{c,m}(\lambda|x-x_0|)$ for any real numbers $A>0$, $\lambda >0$, $x_{0}\in \mathbb{R}^d$.
\end{thm}

\begin{proof}
From Proposition \ref{h0minimizer}, we know that there is a minimizer $\tilde{u}_{c,m}$ of the functional $J(u)$ in $X^*_{rad}$ satisfying
\begin{eqnarray}
1=\int_{\mathbb{R}^d}\tilde{u}_{c,m}^{q+1}\,dx,~~1=\int_{\mathbb{R}^d}\tilde{u}_{c,m}^{m+1}\,dx,~~\beta=\int_{\mathbb{R}^d}|\nabla \tilde{u}_{c,m}|^p\,dx.\label{min3}
\end{eqnarray}
such that it holds that
\begin{eqnarray}\label{minpro}
J(\tilde{u}_{c,m})=\inf_{u\in X} J(u)=\beta.
\end{eqnarray}

For $\mu,\lambda>0$ defined in (\ref{mulambda}), the re-scaling function $u_{c,m}(x)=\frac{1}{\mu}\tilde{u}_{c,m}(\frac{x}{\lambda})$
satisfies
\begin{enumerate}[(i)]
\item For $q<p-1$, $u_{c,m}$ is the unique solution for the free boundary problem (\ref{FBVP1})-(\ref{FBVPbound}) in $X^*_{rad}$, and satisfies $u_{c,m}(r)>0$, $u_{c,m}'(r)<0$ for $0<r<R$. See Proposition \ref{cor} and Theorem \ref{remunique}.
\item For $q\geq p-1$, $u_{c,m}$ is the unique positive solution for the problem (\ref{chachy1})-(\ref{chachybound}) in $X^*_{rad}$, and satisfies $u_{c,m}(r)>0$, $u_{c,m}'(r)<0$ for $0<r<+\infty$. See Proposition \ref{barhmini} and Theorem \ref{remunique1}.
\end{enumerate}
Moreover, $u_{c,m}$ satisfies the following equalities from Corollary \ref{cor1}
\begin{eqnarray}
\int_{\mathbb{R}^d}u_{c,m}^{q+1}\,dx=p^{-\frac{d}{p}}(\gamma-\frac{p}{2})^{\frac{d}{p}-\frac{((p-1)d+p)-q(d-p)}{p(m-q)}}(\gamma+\frac{p}{2})^{\frac{(p-1)d+p
-q(d-p)}{p(m-q)}}\beta^{\frac{d}{p}}.\label{cformula1}
\end{eqnarray}

Define $M_c:=\int_{\mathbb{R}^d}u^{q+1}_{c,m}\,dx$. Then from (\ref{cformula1}), we have
\begin{eqnarray}\label{massvsbeta}
\beta=p(\gamma-p/2)^{\frac{((p-1)d+p)-q(d-p)}{d(m-q)}-1}(\gamma+p/2)^{-\frac{((p-1)d+p)-q(d-p)}{d(m-q)}}M_c^{\frac{p}{d}}.
\end{eqnarray}

 Hence for any $u\in X$ the following inequality holds
\begin{eqnarray}\label{nagyinequality}
\left(\int_{\mathbb{R}^d}|u|^{m+1}\,dx\right)^{\frac{\gamma+\frac{p}{2}}{m+1}}
\leq \frac{1}{\beta}\, \Big(\int_{\mathbb{R}^d}|u|^{q+1}\,dx\Big)^{\frac{\gamma-\frac{p}{2}}{q+1}}\, \int_{\mathbb{R}^d}|\nabla u|^p\,dx.
\end{eqnarray}
 Using (\ref{minpro}) and the re-scaling invariance of $J(u)$, we have $J(A u_{c,m}(\lambda(x-x_0)))=J(u_{c,m}(x))=J(\tilde u_{c,m}(x))=\beta$. Hence the above equality holds if $u(x):=A u_{c,m}(\lambda(x-x_0))$ for any real numbers $A,\lambda>0$, $x_0\in\mathbb{R}^d$.

Denote
\begin{eqnarray}\label{Cqmp}
 \theta=\frac{p}{\gamma+p/2},\quad 1-\theta=\frac{\gamma-p/2}{\gamma+p/2},
\end{eqnarray}
where $\gamma$ is given by (\ref{a}). Thus
\begin{eqnarray}\label{Cqmp1}
C_{q,m,p}=\beta^{-\frac{1}{\gamma+p/2}}.
\end{eqnarray}
 Then (\ref{nagyinequality}) can be recast as (\ref{h sharp inequality1}).  Using the formulas for $\theta$ and $1-\theta$ in (\ref{Cqmp}), (\ref{massvsbeta1}) is a direct consequence of (\ref{massvsbeta}) and (\ref{Cqmp1}).

\end{proof}

\subsection{Remarks on non-radial solutions}

Now we prove that critical points of $J(u)$ in the non-radial case also satisfy a free boundary problem, and the $C^2-$solution of the free boundary problem is a radial solution if the region $\Omega$ satisfies some smooth conditions.
\begin{prop}
Let the non-negative function $\bar u\in C(\mathbb{R}^d)$ be a critical point of $J(u)$ and $\bar \Omega= \{x\in \mathbb{R}^d, \bar u(x)>0\}$. Assume $\bar\Omega $ is a bounded open star domain (with respect to the origin) with a smooth boundary $\bar \Gamma = \partial \bar \Omega$. Then there exist $\lambda, \mu >0$ such that the re-scaling function $u(x)=\frac{1}{\mu}\bar u(\frac{x}{\lambda})$ satisfies the following free boundary problem
\begin{eqnarray}
&&\nabla\cdot(|\nabla u|^{p-2}\nabla u)+u^m= u^q,\quad  \mbox{ in } \Omega, \label{noralsteadyequ}\\
&&u=\partial_n u=0,\mbox{ on } \Gamma,\label{noraliniboundary}
\end{eqnarray}
where $\Omega:=\{x\in \mathbb{R}^d, u(x)>0\}=\lambda \bar\Omega$, $\Gamma$ is a smooth boundary of $\Omega$.
\end{prop}
\begin{proof}
 Similar to Step 1, Step 2 and Step 3 in the proof of Proposition \ref{barhmini}, there exist $\lambda,\mu>0$ such that $u(x)=\frac{1}{\mu}\bar u(\frac{x}{\lambda})$ satisfies
\begin{eqnarray}
\nabla\cdot(|\nabla u|^{p-2}\nabla  u)+u^m= u^q,\mbox{ in } \Omega\label{steadyequ111}
\end{eqnarray}
and it is a zero point of $\mathcal{F}(u)$ defined in (\ref{steadyfree0}), i.e.,
\begin{eqnarray}
\mathcal{F}(u)=0.\label{freeenergy111}
\end{eqnarray}

We also know $\Omega$ is a bounded open star domain with respect to $0$. Let $n$ be the unit outward normal vector to $ \Gamma$. Hence $x\cdot n>0$ on $\Gamma$.

Noticing that $ u = 0$ on $\Gamma$, $u > 0$ in $\Omega$, one has
\begin{eqnarray}\label{equality}
\nabla u = -  |\nabla u| n.
\end{eqnarray}

Below we show the following Pohozaev type identity connecting the energy functional to the contact angle
\begin{eqnarray}\label{nonralfreangle}
\mathcal{F}(u)=\left(1-\frac{1}{p}\right)\int_{\Gamma}(x\cdot n)|\nabla u|^p\,ds.
\end{eqnarray}
Indeed, multiplying $\nabla\cdot (x  u)$ to (\ref{steadyequ111}), one has
\begin{eqnarray}\label{equality1}
\int_{\Omega}\nabla\cdot (x u)\nabla\cdot(|\nabla u|^{p-2}\nabla u)\,dx=\int_{\Omega}\nabla\cdot (x u)(u^q- u^m)\,dx.
\end{eqnarray}
Notice that
\begin{align}
\int_{\Omega}\nabla\cdot (x u)( u^q-u^m)\,dx=&-\int_{\Omega}(x u)\cdot\nabla( u^q-u^m)\,dx\nonumber\\
=&\frac{dq}{q+1} \int_{\Omega} u^{q+1}-\frac{dm}{m+1} \int_{\Omega}u^{m+1}\,dx\label{rightside}.
\end{align}
Using (\ref{equality}), we have
\begin{align}
&\int_{\Omega}\nabla\cdot (x u)\nabla\cdot(|\nabla  u|^{p-2}\nabla u)\,dx\nonumber\\
=&-\int_{\Omega}\nabla(\nabla\cdot (x u))\cdot(|\nabla u|^{p-2}\nabla u)\,dx+\int_{\Gamma}\nabla\cdot (x u)|\nabla u|^{p-2}\partial_n u\,ds\nonumber\\
=&\left(\frac{d}{p}-(d+1)\right)\int_{\Omega}|\nabla u|^p\,dx+\left(1-\frac{1}{p}\right)\int_{\Gamma}(x\cdot n)|\nabla u|^p\,ds.\label{leftside}
\end{align}
Hence from (\ref{equality1}), (\ref{rightside}), (\ref{leftside}) and the definition (\ref{steadyfree0}) of $\mathcal{F}(u)$,  we have
$$ \mathcal{F}(u) =  \left(1-\frac{1}{p}\right)\int_{\Gamma} (n \cdot x) |\nabla u|^p\,ds,$$
i.e., (\ref{nonralfreangle}) holds true.

Since $\mathcal{F}( u)=0$ and $n \cdot x>0$, we know that
\begin{eqnarray*}
\partial_n u=\nabla u\cdot n=0 \mbox{ a.e. on } \Gamma.
\end{eqnarray*}
Summarizing above process, $ u$ satisfies the free boundary problem
\begin{eqnarray*}
&&\nabla\cdot(|\nabla u|^{p-2}\nabla u) + u^m = u^q \mbox{ in } \Omega,\\
&& u = \partial_n u = 0 \mbox { a.e. on } \Gamma.
\end{eqnarray*}

\end{proof}

\begin{prop}\label{freeral}
Let $(u,\Omega=\{x\in \mathbb{R}^d, u(x)>0\})$ be a solution to the free boundary problem (\ref{noralsteadyequ})-(\ref{noraliniboundary}). Assume $u\in C^2(\overline\Omega)$, $\Omega$ is a bounded open domain with $C^2$ boundary (not be assumed simply connected). Then $\Omega$ is a ball and $u$ is radial symmetric about its center.
\end{prop}
\begin{proof}
The proof of Proposition \ref{freeral} is a direct application for \cite[Theorem 8.3.2]{pucci2007maximum} with $A(z,s)=s^{p-2}$, $f(z,s)=z^m-z^q$ (the notations are used in \cite{pucci2007maximum} ), all conditions in Theorem 8.3.2 are satisfied here.
\end{proof}

\section{Closed form solutions with compact support for $q<p-1$}

When $q<p-1$, we know that $u_{c,m}$ is the unique solution to the free boundary problem (\ref{FBVP1})-(\ref{FBVPbound}) as that discussed in Theorem \ref{remunique}. In this section we derive and document some closed form solutions $u_{c,m}$ for some special parameters $d,q,m$ and $p$.

\begin{thm} \label{onedtheorem}
 Suppose $d=1$, $p>1$ and $0< q< \min\{p-1,m\}$. Then
  $u_{c,m}$ possesses the following closed form
  \begin{eqnarray}\label{closedformqxiaopcom}
u_{c,m}(r)=\left(\frac{m+1}{q+1}B^{-1}\left(\bar C
(R-r);\frac{p-1-q}{p(m-q)},1-\frac{1}{p}\right)\right)^{\frac{1}{m-q}}, ~~0\leq r\leq R,
\end{eqnarray}
where
\begin{eqnarray}\label{Rvaluecom}
R=\left(\frac{p-1}{p}\right)^{\frac{1}{p}}\frac{(m+1)^{\frac{p-q-1}{p(m-q)}}(q+1)^{\frac{1}{p}-\frac{p-q-1}{p(m-q)}}}{m-q}
\mathcal{B}\left(1-\frac{1}{p},\frac{p-(q+1)}{p(m-q)}\right),
\end{eqnarray}
and $\bar C$ ia a constant given by
 \begin{eqnarray}\label{barCcom}
\bar C:=\left(\frac{p}{p-1}\right)^{1/p}(m-q)(m+1)^{\frac{1+q-p}{p(m-q)}}(q+1)^{\frac{p-m-1}{p(m-q)}}, 
\end{eqnarray}
and
 $u_{c,m}(0)=\left(\frac{m+1}{q+1}\right)^{1/(m-q)}=:\alpha_c$, $B^{-1}(x;a,b)$ is the inverse function of the incomplete beta function $B(x;a,b)$, which is defined as
\begin{eqnarray}\label{Binversecom}
B(x;a,b)=\int_0^xt^{a-1}(1-t)^{b-1}\,dt.
\end{eqnarray}
The best constant $C_{q,m,p}$ in (\ref{massvsbeta1}) has an exact value, which is given by
\begin{eqnarray}\label{betathe1com}
C_{q,m,p}=\left(\frac{2^p(p-1)^{1-p}\eta_1^{\frac{\eta_1}{m-q}}}{(m-q)^{2p-1}\eta_2^{\frac{\eta_2}{m-q}}} \mathcal{B}^p\left(\frac{\eta_2}{p(m-q)},\frac{2p-1}{p}\right)\right)^{-\ell},
\end{eqnarray}
where $\ell$ is defined by (\ref{candl1}), and
\begin{eqnarray}\label{delta1and2}
\eta_1=(p-1)(m+1)+p,~~\eta_2=(p-1)(q+1)+p.
\end{eqnarray}
\end{thm}
\begin{proof}
Let us recall the equation (\ref{FBVP1})-(\ref{FBVPbound}) for $d=1$ as
\begin{eqnarray}
&&(|u'|^{p-2}u')'+u^m=u^q,\quad 0<r<R,\label{1dRcsteadyequ}\\
&& u'(0)=0,\quad u(R)=u'(R)=0.\label{1diniboundary}
\end{eqnarray}
By the energy functional (\ref{Henergy}) and the energy-dissipation relation (\ref{H}), we know that for $d=1$, (\ref{1dRcsteadyequ}) possesses a first integral and it is a constant, i.e.,
\begin{eqnarray}\label{firstintex}
\frac{p-1}{p}|u'|^p+\frac{ u^{m+1}}{m+1}-
\frac{u^{q+1}}{q+1}=C.
\end{eqnarray}
Due to $u(R)=u'(R)=0$, then $C=0$. Hence the conditions $u(0)=\alpha_c$ and $ u'(0)=0$ imply
\begin{eqnarray}\label{alphac}
\frac{ \alpha_c^{m+1}}{m+1}-
\frac{\alpha_c^{q+1}}{q+1}=0,\mbox{ i.e. }, \alpha_c=\left(\frac{m+1}{q+1}\right)^\frac{1}{m-q}>1.
\end{eqnarray}
Solving (\ref{firstintex}) gives
\begin{eqnarray}\label{firstorder}
u'(r)=-\left(\frac{p}{p-1}\right)^{1/p} \left(\frac{u^{q+1}(r)}{q+1}-\frac{u^{m+1}(r)}{m+1}\right)^{1/p}.
\end{eqnarray}
Using the method of separation of variables for (\ref{firstorder}) and integrating the result equality in $(r,R)$, we have
\begin{eqnarray}\label{R-r}
\int_{0}^{u(r)}\frac{ds}{\left(-\frac{s^{m+1}}{m+1}+\frac{s^{q+1}}{q+1}\right)^{\frac{1}{p}}}
=\left(\frac{p}{p-1}\right)^{\frac{1}{p}} (R-r).
\end{eqnarray}
In (\ref{R-r}), taking $r=0$, some computations give
$$
R=\left(\frac{p-1}{p}\right)^{\frac{1}{p}}\frac{(m+1)^{\frac{p-q-1}{p(m-q)}}(q+1)^{\frac{1}{p}-\frac{p-q-1}{p(m-q)}}}{m-q}
\mathcal{B}\left(1-\frac{1}{p},\frac{p-q-1)}{p(m-q)}\right),
$$
which is exactly (\ref{Rvaluecom}).
Hence if $p>1$, $0<q<\min\{p-1,m\}$, we have $\mathcal{B}\left(1-\frac{1}{p},\frac{p-q-1}{p(m-q)}\right)<\infty$.
Moreover, solving (\ref{R-r}), it is deduced that
$$
u_{c,m}(r)=\left(\frac{m+1}{q+1}B^{-1}\left(\bar C
(R-r);\frac{p-q-1}{p(m-q)},1-\frac{1}{p}\right)\right)^{\frac{1}{m-q}},
$$
where $\bar C$ is given by (\ref{barCcom}). Hence the formula (\ref{closedformqxiaopcom}) holds.

Now we will derive the best constant $C_{q,m,p}$ by computing the minimum $\beta$ of the functional $J(u)$ for $d=1$. Similar to (\ref{cformula1}), from Corollary \ref{cor1} we also have
\begin{eqnarray}
\int_{\mathbb{R}^d}|\nabla  u_{c,m}|^p\,dx=p^{1-\frac{d}{p}}(\gamma-\frac{p}{2})^{\frac{d}{p}-1-\frac{((p-1)d+p)-q(d-p)}{p(m-q)}}
(\gamma+\frac{p}{2})^{\frac{((p-1)d+p)-q(d-p)}{p(m-q)}}\beta^{\frac{d}{p}}.\label{cformula3}
\end{eqnarray}
Due to (\ref{a}), we deduce
\begin{eqnarray}\label{form2}
\gamma-\frac{p}{2}=\frac{[(p-1)(m+1)+p](q+1)}{m-q},\quad \gamma+\frac{p}{2}=\frac{[(p-1)(q+1)+p](m+1)}{m-q}.
\end{eqnarray}
Hence by (\ref{form2}), the right side of (\ref{cformula3}) can be written as
\begin{align}\label{best1}
&p^{1-\frac{d}{p}}(\gamma-\frac{p}{2})^{\frac{d}{p}-1-C_0}
(\gamma+\frac{p}{2})^{C_0}\beta^{\frac{d}{p}}\\
=&
p^{1-\frac{d}{p}}\left(\frac{[(p-1)(m+1)+p](q+1)}{m-q}\right)^{\frac{d}{p}-1-C_0}
\left(\frac{[(p-1)(q+1)+p](m+1)}{m-q}\right)^{C_0}\beta^{\frac{d}{p}},\nonumber
\end{align}
where $C_0:=\frac{((p-1)d+p)-q(d-p)}{p(m-q)}$.
On the other hand, we compute the left side of (\ref{cformula3})
\begin{align}\label{best2}
\int_{\mathbb{R}^d}|\nabla  u_{c,m}|^p\,dx=&\int_{-R_c}^{R_c} |u'_{c,m}|^{p-2}u'_{c,m}\cdot u'_{c,m}  \,d x\nonumber\\
=&-2\int_0^{\alpha_c} |u'_{c,m}|^{p-2}u'_{c,m}\,d u_{c,m}\nonumber\\
=&2\int_0^{\alpha_c}\left(\frac{p}{p-1}\right)^{\frac{p-1}{p}} \left(\frac{u^{q+1}}{q+1}-\frac{u^{m+1}}{m+1}\right)^{\frac{p-1}{p}} \,d u\nonumber\\
=&\frac{2\left(\frac{p}{p-1}\right)^{\frac{p-1}{p}}(m+1)^{\frac{p+(p-1)(q+1)}{p(m-q)}}}{(m-q)(q+1)^{\frac{p+(p-1)(m+1)}{p(m-q)}}}
\mathcal{B}\left(\frac{p+(p-1)(q+1)}{p(m-q)},\frac{2p-1}{p}\right).
\end{align}
Then (\ref{cformula3}), (\ref{best1}) and (\ref{best2}) imply
\begin{eqnarray}\label{betathe}
\beta=\frac{2^p(p-1)^{1-p}\eta_1^{\frac{\eta_1}{m-q}}}{(m-q)^{2p-1}\eta_2^{\frac{\eta_2}{m-q}}}\left( \mathcal{B}\left(\frac{\eta_2}{p(m-q)},\frac{2p-1}{p}\right)\right)^p.
\end{eqnarray}
Using the relation $C_{q,m,p}=\beta^{-\frac{1}{\gamma+p/2}}$ in (\ref{Cqmp1}), then we obtain (\ref{betathe1com}).

Furthermore, from (\ref{massvsbeta}) we can obtain for $d=1$
$$
M_c=\frac{2(p-1)^{\frac{1-p}{p}}[(p-1)(m+1)+p](m+1)^{\frac{p+(p-1)(q+1)}{p(m-q)}}}{p^{1/p}(m-q)^2(q+1)^{\frac{m+1-p(q+2)}{p(m-q)}}}
\mathcal{B}\left(\frac{p+(p-1)(q+1)}{p(m-q)},\frac{2p-1}{p}\right).
$$

\end{proof}

\begin{rem}
 When $r\to R$, from (\ref{R-r}) we have that $u(r)\sim C(R-r)^{\alpha}$, $\alpha=\frac{p}{p-q-1}>1$, $C=\left(\frac{p-q-1}{p}\right)^{\frac{p}{p-1-q}}\left(\frac{p}{(p-1)(q+1)}\right)^{\frac{1}{p-1-q}}$. Hence $u(R)=u'(R)=0$ and the equation (\ref{1dRcsteadyequ})-(\ref{1diniboundary}) holds in the distribution sense in $\mathbb{R}^d$.
\end{rem}

\begin{thm}
 Suppose $d=1$, $p>1$, $0\leq q< p-1$, and $\frac{p-1-q}{p(m-q)}=1$, then $u_{c,m}$ is given by the following Barenblatt profile:
\begin{eqnarray}\label{uc}
u_{c,m}=\left(\frac{m+1}{q+1}\right)^{\frac{1}{m-q}}\left(1-\left(\frac{p-1}{p}\bar C\right)^{\frac{p}{p-1}}r^{\frac{p}{p-1}}\right)^{\frac{1}{m-q}},~~0\leq r<R,
\end{eqnarray}
where $\bar C$ is defined by (\ref{barCcom}) and $R$ is defined in (\ref{Rvaluecom}) and it satisfies $u_{c,m}(R)=0$.

The best constant is given by
\begin{eqnarray}\label{thebest}
C_{q,m,p}=\frac{(p-m-1)^{\frac{2p-1}{p}\theta}[p(m+1)]^{\frac{1}{m+1}}}{[2(p-1)]^{\theta}\eta_1^{\frac{\eta_1}{p(m+1)^2}}}
\left(\mathcal{B}\left(\frac{(m+1)(p-1)}{p-m-1},\frac{2p-1}{p}\right)\right)^{-\theta},
\end{eqnarray}
where $\theta=\frac{p-m-1}{(p-1)(m+1)^2}$ and $\eta_1$ is defined by (\ref{delta1and2}).
\end{thm}

\begin{proof}
In (\ref{closedformqxiaopcom}), taking $\frac{p-1-q}{p(m-q)}=1$, we have
 \begin{eqnarray}\label{specalcase1}
u_{c,m}(r)=\left(\frac{m+1}{q+1}B^{-1}\left(\bar C
(R-r);1,1-\frac{1}{p}\right)\right)^{\frac{1}{m-q}}, ~~0\leq r\leq R.
\end{eqnarray}
Noticing that
$$
B\left(x;1, 1-\frac{1}{p}\right)=\int_0^x(1-t)^{-\frac{1}{p}}\,dt=\frac{p}{p-1}\left(1-(1-x)^{\frac{p-1}{p}}\right),
$$
we deduce its inverse function
$$
B^{-1}\left((x;1, 1-\frac{1}{p}\right)=1-\left(1-\frac{p-1}{p}x\right)^{\frac{p}{p-1}}.
$$
Hence from (\ref{specalcase1}), we easily get
 \begin{eqnarray}\label{specal}
u_{c,m}(r)=\left(\frac{m+1}{q+1}\left(1-\left(1-\frac{p-1}{p}(\bar C(R-r))\right)^{\frac{p}{p-1}}\right)\right)^{\frac{1}{m-q}}, ~~0\leq r\leq R.
\end{eqnarray}
On the other hand, using (\ref{Rvaluecom}) and (\ref{barCcom}), we deduce
 \begin{eqnarray}\label{Rbarc}
R\bar C=\mathcal{B}\left(1-\frac{1}{p},1\right)=\frac{p}{p-1}.
 \end{eqnarray}
Therefore, (\ref{specal}) and (\ref{Rbarc}) imply that (\ref{uc}) holds.

Furthermore, in the special case $\frac{p-1-q}{p(m-q)}=1$ and $d=1$, we can compute that
$$q=\frac{pm-p+1}{p-1}, ~~q+1=\frac{pm}{p-1},~~m-q=\frac{p-m-1}{p-1},~~\ell=\frac{p-m-1}{p(p-1)(m+1)^2}.$$
Thus from (\ref{betathe1com}), we can compute that the best constant $C_{q,m,p}$ can be expressed by (\ref{thebest}).
\end{proof}

\begin{rem}
In the paper \cite{del2003optimal}, Del Pino and Dolbeault derived the best constant $\bar C_{q,m,p}$ (see (\ref{barC1})) by a variational method for a differential functional in sum form ($I_2(u)$ below)
  in the case $d\geq 2$ and $q=\frac{pm-p+1}{p-1}$. We will discuss the relation between the two functionals $I_2(u)$ and $I_2(u)$ below.

  Formally, if we take $d=1$ in (\ref{barC1}), we find that $\bar C_{q,m,p}$ exactly equals to $C_{q,m,p}$ given in (\ref{thebest}). Detail verifications are provided in Proposition \ref{d1equa} of Appendix B. In other words, we extend the results to the dimension $d=1$.
\end{rem}
Let
$$
I_1(u)=\frac{\|\nabla u\|^{\theta}_{L^p}\|u\|^{1-\theta}_{q+1}}{\|u\|_{L^{m+1}}},\quad I_2(u)=\frac{\frac{1}{p}\|\nabla u\|^p_{L^p}+\frac{1}{q+1}\|u\|^{q+1}_{q+1}}{\|u\|^a_{L^{m+1}}},
$$
where $a=\frac{(m+1)(dp+(p-d)(q+1))}{dp+p(m+1)-d(q+1)}$.
Indeed, if the both functionals $I_1(u)$ and $I_2(u)$ have minimizers, then minimizers of the functional $I_2(u)$ must be minimizers of the functional $I_1(u)$, see the following two lemmas.
\begin{lem} Let $u\in X$, where $X$ is given by (\ref{X}), then there is $\lambda>0$ such that $u$ and the corresponding re-scaling function $u_{\lambda}=\lambda^{-\frac{d}{m+1}}u\left(\frac{x}{\lambda}\right)$ satisfy
\begin{eqnarray}\label{I1I2}
KI^{a}_1(u)=\min_{\lambda}I_2(u_{\lambda}),
\end{eqnarray}
where the exponent $a=\frac{(m+1)(dp+(p-d)(q+1))}{dp+p(m+1)-d(q+1)}$ and the constant $K$ will be determined by (\ref{K}) below.
 \end{lem}
\begin{proof}
Noticing that $\|u_{\lambda}\|_{L^{m+1}}=\|u\|_{L^{m+1}}$ due to $u_{\lambda}(x)=\lambda^{-\frac{d}{m+1}}u(\frac{x}{\lambda})$. Then we deduce
 \begin{eqnarray}
I_2(u_{\lambda})= \frac{\left(\frac{1}{p}\lambda^{-\alpha}\|\nabla u\|^p_{L^p}+\lambda^{\beta}\frac{1}{q+1}\|u\|^{q+1}_{L^{q+1}}\right)}{\|u\|^a_{L^{m+1}}},
\end{eqnarray}
where $\alpha:=\frac{pd}{m+1}+p-d$ and $\beta:=d-\frac{q+1}{m+1}d$. For fixed $u\in X$, denote
$$
f(\lambda):=\frac{1}{p}\lambda^{-\alpha}\|\nabla u\|^p_{L^p}+\lambda^{\beta}\frac{1}{q+1}\|u\|^{q+1}_{L^{q+1}}.
$$

To compute the minimum of $f(\lambda)$, solving $f'(\lambda)=0$ gives that
$$
\lambda=\left(\frac{\alpha(q+1)}{p\beta}\frac{\|\nabla u\|^p_{L^p}}{\|u\|^{q+1}_{L^{q+1}}}\right)^{\frac{1}{\alpha+\beta}}
$$
 and hence
$$
\min_{\lambda}f(\lambda)=\frac{\alpha+\beta}{p\beta}\left(\frac{\alpha(q+1)}{p\beta}\right)^{-\frac{\alpha}{\alpha+\beta}}\|\nabla u\|^{\frac{p\beta}{\alpha+\beta}}_{L^p}\|u\|^{\frac{\alpha(q+1)}{\alpha+\beta}}_{L^{q+1}}.
$$
Defining
\begin{eqnarray}\label{K}
K:=\frac{\alpha+\beta}{p\beta}\left(\frac{\alpha(q+1)}{p\beta}\right)^{-\frac{\alpha}{\alpha+\beta}}
\end{eqnarray}
 and noticing the facts:
$$
\frac{p\beta}{\alpha+\beta}=a\theta, \quad  \frac{\alpha(q+1)}{\alpha+\beta}=a(1-\theta),
$$
we obtain (\ref{I1I2}).
\end{proof}
\begin{lem}
Let $u_{\infty}$ be a minimizer of the functional $I_2(u)$. Then it holds that
 \begin{eqnarray}\label{jminimizer}
I_1(u_{\infty})\leq I_1(u)\quad \mbox{for any } u\in X.
\end{eqnarray}
 \end{lem}
\begin{proof}
Since $u_{\infty}$ is a minimizer of the functional $I_2(u)$, then for any $u\in X$, they hold that
$$
I_2(u_{\infty})\leq I_2(u),\quad I_2(u_{\infty})\leq I_2(u_{\lambda})\quad \mbox{ for any } \lambda>0.
$$
Thus we have
\begin{eqnarray}\label{one}
I_2(u_{\infty})\leq \min_{\lambda}I_2(u_{\lambda}).
\end{eqnarray}
In particular, taking $u=u_{\infty}$ in (\ref{one}) gives
\begin{eqnarray}\label{one1}
I_2(u_{\infty})\leq \min_{\lambda}I_2(u_{\infty, \lambda}).
\end{eqnarray}
On the other hand, let $\lambda_0=1$, we know that
\begin{eqnarray}\label{other}
I_2(u_{\infty})= I_2(u_{\infty, \lambda_0})\geq \min_{\lambda}I_2(u_{\infty, \lambda}).
\end{eqnarray}
Hence (\ref{one1}) and (\ref{other}) imply
\begin{eqnarray}\label{other1}
I_2(u_{\infty})=  \min_{\lambda}I_2(u_{\infty, \lambda}).
\end{eqnarray}
Together with (\ref{one}) and (\ref{I1I2}), we deduce
$$
K I^a_1(u_{\infty})=\min_{\lambda}I_2(u_{\infty, \lambda})=I_2(u_{\infty})\leq \min_{\lambda}I_2(u_{\lambda})=KI^a_1(u),
$$
which means that (\ref{jminimizer}) holds.

\end{proof}

We remark in the following lemma that the form of Barenblatt profile (\ref{uc}) is also preserved on the $p$-Laplace operator in higher dimensions. This fact provides a closed-form solution for the free boundary problem (\ref{FBVP1})-(\ref{FBVPbound}) with some special parameters $p,q$ and $m$, see Lemma \ref{4.2}.
\begin{lem}\label{lmH}
Let $H(r)=\left(R^{\frac{p}{p-1}}-r^{\frac{p}{p-1}}\right)_{+}^{\alpha}$, $\alpha>1$. Then $H(r)$ satisfies the following equation
\begin{eqnarray}\label{laplaceH}
\Delta_p H=A H^{\frac{(\alpha-1)(p-1)-1}{\alpha}}-B H^{\frac{(\alpha-1)(p-1)}{\alpha}},
\end{eqnarray}
where $A=(p-1)(\alpha-1)\left(\frac{p}{p-1}\right)^p\alpha^{p-1}R^{\frac{p}{p-1}}$ and $B=\left(d+p(\alpha-1)\right)\left(\frac{p}{p-1}\right)^{p-1}\alpha^{p-1}$.
\end{lem}
\begin{proof}
For $r<R$, a direct computation gives that
\begin{align}
  H'(r)=&-\alpha\frac{p}{p-1} r^{\frac{1}{p-1}} \left(R^{\frac{p}{p-1}}-r^{\frac{p}{p-1}}\right)^{\alpha-1},\label{hfirstder}\\
 |H'(r)|^{p-2}H'(r)=&-\left(\alpha\frac{p}{p-1}\right)^{p-1} r \left(R^{\frac{p}{p-1}}-r^{\frac{p}{p-1}}\right)^{(\alpha-1)(p-1)},\\
 (|H'(r)|^{p-2}H'(r))'=&(p-1)(\alpha-1)\left(\frac{p}{p-1}\right)^p\alpha^{p-1}r^{\frac{p}{p-1}}\left(R^{\frac{p}{p-1}}-r^{\frac{p}{p-1}}\right)^{(\alpha-1)(p-1)-1}\nonumber\\
 &-\alpha^{p-1}\left(\frac{p}{p-1}\right)^{p-1} \left(R^{\frac{p}{p-1}}-r^{\frac{p}{p-1}}\right)^{(\alpha-1)(p-1)}.
 \end{align}
Hence
\begin{align}\label{deltaH}
\Delta_p H=&(|H'(r)|^{p-2}H'(r))'+\frac{d-1}{r}|H'(r)|^{p-2}H'(r)\nonumber\\
=&(p-1)(\alpha-1)\left(\frac{p}{p-1}\right)^p\alpha^{p-1}R^{\frac{p}{p-1}}\left(R^{\frac{p}{p-1}}-r^{\frac{p}{p-1}}\right)^{(\alpha-1)(p-1)-1}\nonumber\\
&-\left(d+p(\alpha-1)\right)\left(\frac{p}{p-1}\right)^{p-1}\alpha^{p-1} \left(R^{\frac{p}{p-1}}-r^{\frac{p}{p-1}}\right)^{(\alpha-1)(p-1)}.
\end{align}

Noticing that (\ref{hfirstder}) gives $H(R_{-})=H'(R_{-})=0$, then (\ref{laplaceH}) holds in the distribution sense in $\mathbb{R}^d$.
\end{proof}
\begin{lem}\label{4.2} Assume $p>1$ and $0\leq q<p-1$.
Let $m=\frac{(q+1)(p-1)}{p}$ and $\alpha=\frac{p-1}{p-m-1}$, then it holds
$$
\Delta_p H=AH^q-BH^m.
$$
Furthermore, take
\begin{eqnarray}
R= m^{-\frac{p-1}{p}}\frac{d}{(m+1)\theta}, \quad K=\left(\frac{d}{(m+1)\theta}\right)^{\frac{1}{m-p+1}}\left(\frac{p}{p-m-1}\right)^{\frac{p-1}{m-p+1}},\label{KR}
\end{eqnarray}
where $\theta$ is defined by (\ref{sharpine}). Then
\begin{eqnarray}\label{uc1}
u_{c,m}(r):=K H(r)=K \left(R^{\frac{p}{p-1}}-r^{\frac{p}{p-1}}\right)_{+}^{\alpha}
\end{eqnarray}
is the unique radial non-negative solution for the following problem
\begin{eqnarray}\label{u}
\Delta_p u=u^q-u^m
\end{eqnarray}
with $u'(0)=0$ and $u'(R)=u'(R)=0$.
\end{lem}
\begin{proof}
From $m=\frac{(q+1)(p-1)}{p}$ and $\alpha=\frac{p-1}{p-m-1}$, which means $\alpha=\frac{p}{p-1-q}>1$. Thus we deduce
\begin{eqnarray}\label{relation}
 \frac{(\alpha-1)(p-1)}{\alpha}= m,\quad \frac{(\alpha-1)(p-1)-1}{\alpha}= q.
 \end{eqnarray}
Hence (\ref{deltaH}) implies that
\begin{eqnarray}\label{deltaH1}
\Delta_p H&=&A H^{ q}-B H^{ m}.
\end{eqnarray}

Finally, we can directly verify that $u_{c,m}$ is the solution of (\ref{u}) and satisfies boundary conditions.
\end{proof}

\begin{rem}\label{rm4.3} Assume that $p>1$, $0\leq q<p-1$ and $m=\frac{(q+1)(p-1)}{p}$.
For $u_{c,m}$ defined in (\ref{uc1}), we have (see Lemma \ref{lmApp1})
\begin{eqnarray}\label{M_c}
M_c:=\int_{\mathbb{R}^d}u^{q+1}_{c,m}\,dx=d~\alpha(d)K^{q+1}R^{d+\frac{mp^2}{(p-1)(p-m-1)}}\frac{p-1}{p}\mathcal{B}\left( \frac{d(p-1)}{p}, \frac{pm}{p-m-1}+1\right)
\end{eqnarray}
and the best constant is given by
$$
C_{q,m,p}=\left(\frac{1-\theta}{\theta}\right)^{\frac{\theta}{p}}(1-\theta)^{-\frac{1}{m+1}}M_c^{-\frac{\theta}{d}},\quad
\theta=\frac{(p-1-m)d}{(m+1)[d(p-m-1)+pm]}.
$$
We recover the best constant given by Del Pino and Dolbeault in their celebrated work \cite[Theorem 3.1]{del2003optimal}. See Lemma \ref{lmApp} in Appendix B for detail verification.
\end{rem}

\begin{thm}\label{th63}
If $q=0$, $m=1$, $p>1$ and $d\geq 1$, then the following inequality holds
\begin{eqnarray}\label{specialineq}
\|u\|_{L^{2}}\leq  C_{0,1,p} \|u\|^{1-\theta}_{L^{1}}\|\nabla u\|^{\theta}_{L^p},\quad \theta=\frac{pd}{2(dp+p-d)}.
\end{eqnarray}
Here $C_{0,1,p}$ is the best constant given by
\begin{eqnarray}\label{specialbest}
C_{0,1,p}=\theta^{-\frac{\theta}{p}}(1-\theta)^{\frac{\theta}{p}-\frac{1}{2}}R^{-\theta} \omega_d^{-\frac{\theta}{d}},
\end{eqnarray}
where $\omega_d$ is the volume of the $d$-dimension unit ball $B_1$, $R$ is the first touch down point of $u_{c,m}$, where $u_{c,m}$ is the unique solution to the free boundary problem (\ref{FBVP1})-(\ref{FBVPbound}).
Particularly for $p=2$, $R^2$ is also the first eigenvalue of the Laplace operator with Neumann boundary in the ball $B_1$, and $C_{0,1,2}$ is consistent with the best constant of Nash's inequality, see \cite{carlen1993sharp}.
\end{thm}
\begin{proof}
In (\ref{h sharp inequality1}), taking $q=0$, $m=1$, $p>1$ and $d\geq 1$, we can obtain (\ref{specialineq}) with the best constant $C_{0,1,p}$:
\begin{eqnarray}\label{formulacom1}
C_{0,1,p}=\theta^{-\frac{\theta}{p}}(1-\theta)^{\frac{\theta}{p}-\frac{1}{2}}M_c^{-\frac{\theta}{d}},\quad M_c=\int_{\mathbb{R}^d}u_{c,m}\,dx.
\end{eqnarray}
Since $u_{c,m}$ satisfies the free boundary problem
\begin{eqnarray}
&& (|u'|^{p-2}u')' + \frac{d-1}{r} |u'|^{p-2}u' + u = 1 \quad\mbox{ for }0<r<R,\label{FBVP}\\
&& u'(0)=0,\,\, u(R)=u'(R)=0,\label{FBVP}
 \end{eqnarray}
then a simple computation gives
\begin{eqnarray}\label{formulacom2}
\int_{B(0,R)} u_{c,m}\,dx=\int_{B(0,R)}1\,dx=\omega_d R^d.
\end{eqnarray}
Hence from (\ref{formulacom1}) and (\ref{formulacom2}), we deduce (\ref{specialbest}).

Particularly, for $p=2$, we have $\theta=\frac{d}{d+2}$, $1-\theta=\frac{2}{d+2}$, and
$$
C_{0,1,2}=2^{\frac{d}{2(d+2)}}d^{-\frac{d+1}{d+2}}(d+2)^{\frac{1}{2}}R^{-\frac{d}{d+2}} \omega_d^{-\frac{1}{d+2}}.
$$
For this case $u_{c,m}$ is the solution to
\begin{eqnarray}
&& (u')' + \frac{d-1}{r} u' + u = 1 \quad\mbox{ for }0<r<R,\label{fbvpp2}\\
&& u'(0)=0,\,\, u(R)=u'(R)=0,\label{FBVPp2}
 \end{eqnarray}
It is directly verified that $\lambda=R^2$ and $u(r)=1-u_{c,m}(Rr)$ are the first eigenvalue and the corresponding eigenfunction of the Laplace operator with Neumann boundary in the ball $B_1$. The best constant $C_{0,1,2}$ is exactly same as the best constant of Nash's inequality, see \cite{carlen1993sharp}. This completes the proof of Theorem \ref{th63}.
\end{proof}

\section{Closed form positive solutions for $q\geq p-1$}

When $q\geq p-1$, we know that $u_{c,m}$ is the unique positive solution to the problem (\ref{chachy1})-(\ref{chachybound}) from the Compact Support Principle \cite{pucci2007maximum} as that discussed in the introduction. In this section we derive and document some closed form solutions $u_{c,m}$ for some special parameters $d,q,m$ and $p$.

\begin{thm} \label{hightdnagy1pos}
 Suppose $d=1$, $p>1$ and $p-1\leq q<m$. Then
  $u_{c,m}$ possesses the following closed form
  \begin{eqnarray}\label{closedform1pos}
u_{c,m}(r)=\left(\frac{m+1}{q+1}\right)^{\frac{1}{m-q}}\left(1-B^{-1}\left(\bar C r;1-\frac{1}{p},\frac{p-1-q}{p(m-q)}\right)\right)^{\frac{1}{m-q}},~~ r\geq 0,
\end{eqnarray}
where $\bar C$ is given by (\ref{barCcom}). The best constant $C_{q,m,p}$ is still given by (\ref{betathe1com}).

Moreover, for $m>2q+1$ there is $r_*>0$ (it is independent of $m$) such that $u_{c,m}$ satisfies the following estimates
\begin{itemize}
\item for $q>p-1$,
it holds that
\begin{eqnarray}\label{pdaq1oinftypro}
u(r)+r|u'(r)|\leq  Cr^{-\frac{p}{q+1-p}}, \quad \mbox{for } r\geq  2r_*,
\end{eqnarray}
\item for $q=p-1$, it holds that
\begin{eqnarray}\label{pq1oinftypro}
u(r)\leq e^{-Cr},\quad \mbox{for } r\geq 2r_*,
\end{eqnarray}
 \end{itemize}
 where $C$ is a constant independent of $m$.
\end{thm}

\begin{proof}
Integrating (\ref{firstorder}) from $0$ to $r$ for any $r\in (0,\infty)$, we deduce
\begin{align}\label{solveu}
\left(\frac{p}{p-1}\right)^{1/p}r=&\frac{(q+1)^{1/p}}{m-q}\left(\frac{m+1}{q+1}\right)^{\frac{p-q-1}{p(m-q)}}
\int_0^{1-\frac{q+1}{m+1}u^{m-q}}y^{-1/p}(1-y)^{\frac{p-q-1}{p(m-q)}-1}\,dy\\
=&\frac{(q+1)^{1/p}}{m-q}\left(\frac{m+1}{q+1}\right)^{\frac{p-q-1}{p(m-q)}} B\left(1-\frac{q+1}{m+1}u^{m-q};~1-\frac{1}{p},~\frac{p-q-1}{p(m-q)}\right).\nonumber
\end{align}
Hence we solve $u_{c,m}(r)$ given by (\ref{closedform1pos}) with $\bar C$ given in (\ref{barCcom}).

The same process to (\ref{cformula3})-(\ref{betathe}) shows that the best constant $C_{q,m,p}$ given by (\ref{betathe1com}).

Now we prove the decay properties of $u_{c,m}$. Since $u(r)\to 0$ as $r\to \infty$, then there exists a $0<r_0<\infty$ such that $u(r_0)=1$. From (\ref{solveu}), we have
\begin{eqnarray*}
\left(\frac{p}{p-1}\right)^{1/p}r_0=\frac{(q+1)^{1/p}}{m-q}\left(\frac{m+1}{q+1}\right)^{\frac{p-q-1}{p(m-q)}}
\int_0^{1-\frac{q+1}{m+1}}y^{-1/p}(1-y)^{\frac{p-q-1}{p(m-q)}-1}\,dy,
\end{eqnarray*}
which means that
\begin{eqnarray}\label{rqmp}
0<r_0\leq\left(\frac{p}{(p-1)(q+1)}\right)^{1-1/p}\left(\frac{m+1}{m-q}\right)^{1/p}.
\end{eqnarray}
By (\ref{rqmp}), we deduce if $m>2q+1$, then $0<r_0<2\left(\frac{p}{(p-1)(q+1)}\right)^{1-1/p}=:r_*$.
 Since $r'(r)<0$ for $r>0$, we have $0<u(r)<1$ for $r\geq r_*$, hence $\frac{u^{q+1}(r)}{q+1}-\frac{u^{m+1}(r)}{m+1}\geq u^{q+1}(r)\left(\frac{1}{q+1}-\frac{1}{m+1}\right)$ for $r\geq r_*$. Therefore from (\ref{firstorder}), we have that
\begin{eqnarray}\label{derqua2}
-u'(r)\geq \left(\frac{p}{p-1}\right)^{\frac{1}{p}}\left(\frac{1}{q+1}-\frac{1}{m+1}\right)^{\frac{1}{p}}u^{\frac{q+1}{p}}(r).
\end{eqnarray}
Using the method of separation of variable for (\ref{derqua2}) and integrating the result inequality from $r_*$ to r for any $r>r_*$, we obtain
$$
u^{\frac{p-q-1}{p}}(r)\geq 1+\frac{q+1-p}{p}\left(\frac{p}{p-1}\right)^{\frac{1}{p}}\left(\frac{1}{q+1}-\frac{1}{m+1}\right)^{\frac{1}{p}}(r-r_*).
$$
Taking $r>2r_*$, we have $r-r_*\geq \frac{r}{2}$. Therefore, for $r>2r_*$ we know that
\begin{align*}
u(r)\leq& \left(1+\frac{q+1-p}{p}\left(\frac{p}{p-1}\right)^{\frac{1}{p}}
\left(\frac{1}{q+1}-\frac{1}{m+1}\right)^{\frac{1}{p}}\frac{r}{2}\right)^{-\frac{p}{q+1-p}}\\
\leq&\left(\frac{q+1-p}{2p}\right)^{-\frac{p}{q+1-p}}\left(\frac{p}{p-1}\right)^{-\frac{1}{q+1-p}}
\left(\frac{1}{q+1}-\frac{1}{m+1}\right)^{-\frac{1}{q+1-p}}r^{-\frac{p}{q+1-p}}.
\end{align*}
Denoting
$$
C(q,m,p):=\left(\frac{q+1-p}{2p}\right)^{-\frac{p}{q+1-p}}\left(\frac{p}{p-1}\right)^{-\frac{1}{q+1-p}}
\left(\frac{1}{q+1}-\frac{1}{m+1}\right)^{-\frac{1}{q+1-p}},
$$
then when $m>2q+1$, we have
$$
C(q,m,p)\leq \left(\frac{q+1-p}{2p}\right)^{-\frac{p}{q+1-p}}\left(\frac{p}{p-1}\right)^{-\frac{1}{q+1-p}}
\left(\frac{1}{2(q+1)}\right)^{-\frac{1}{q+1-p}}=:C(q,p).
$$
Hence we obtain
\begin{eqnarray}\label{urcon}
u(r)\leq C(q,m,p) r^{-\frac{p}{q+1-p}}\leq C(q,p) r^{-\frac{p}{q+1-p}}, \quad \mbox{for } r\geq 2r_*.
\end{eqnarray}

Again from (\ref{firstorder}), we obtain for any $r>0$
\begin{eqnarray}\label{derqua21}
|u'(r)|= \left(\frac{p}{p-1}\right)^{\frac{1}{p}}\left(\frac{u^{q+1}}{q+1}-\frac{u^{m+1}}{m+1}\right)^{\frac{1}{p}}\leq \left(\frac{p}{(p-1)(q+1)}\right)^{\frac{1}{p}}u^{\frac{q+1}{p}}.
\end{eqnarray}
Combining (\ref{urcon}) and (\ref{derqua21}), we can deduce
\begin{eqnarray}\label{udercon}
 r|u'(r)| \leq \left(\frac{p}{(p-1)(q+1)}\right)^{\frac{1}{p}} C(q,p)^{\frac{q+1}{p}} r^{-\frac{p}{q+1-p}}~~\mbox{ for } r\geq 2r_*.
\end{eqnarray}
Hence (\ref{urcon}) and (\ref{udercon}) give (\ref{pdaq1oinftypro}).

For the case $q=p-1$, Integrating (\ref{derqua2}) from $r_*$ to r for any $r>r_*$, we obtain
$$
-\ln u(r)\geq \left(\frac{p}{p-1}\right)^{\frac{1}{p}}\left(\frac{1}{p}-\frac{1}{m+1}\right)^{\frac{1}{p}}(r-r_*).
$$
Taking $r>2r_*$, we have $r-r_*\geq \frac{r}{2}$. Therefore, for $r>2r_*$ we know that
\begin{eqnarray*}
\ln u(r)\leq -\left(\frac{p}{p-1}\right)^{\frac{1}{p}}\left(\frac{1}{p}-\frac{1}{m+1}\right)^{\frac{1}{p}}\frac{r}{2}.
\end{eqnarray*}
Let
\begin{eqnarray}\label{cmp}
C_{m,p}:=\frac{1}{2}\left(\frac{p}{p-1}\right)^{\frac{1}{p}}\left(\frac{1}{p}-\frac{1}{m+1}\right)^{\frac{1}{p}}.
\end{eqnarray}
Since $m>2q+1>2p-1$, we have $C_{m,p}>\frac{1}{2}\left(\frac{1}{2(p-1)}\right)^{\frac{1}{p}}=:C(p)$.
Then we obtain
\begin{eqnarray}\label{urconq=p-1}
u(r)\leq e^{-C_{m,p} r}\leq e^{-C(p) r} \quad \mbox{for } r\geq 2r_*,
\end{eqnarray}
i.e., (\ref{pq1oinftypro}) holds.
\end{proof}

\begin{rem}
For $0\leq r\leq R$, $u_{c,m}$ given in (\ref{closedformqxiaopcom}) and (\ref{closedform1pos}), respectively, are exactly same. This statement can be verified by the properties of the incomplete Beta function
$$
B(1-x;a,b)=\mathcal{B}(a,b)-B(x;b,a).
$$
\end{rem}

\begin{thm}
Suppose $d=1$, $p>1$, $p-1<q<m$, and $m=\frac{p(q-1)+1}{p-1}$. Then $u_{c,m}$ has the following closed form
\begin{eqnarray}\label{ucxxpos}
u_{c,m}(r)=\left(\frac{m+1}{q+1}\right)^{\frac{1}{m-q}}\left(1+\left(\frac{p-1}{p}\bar C\right)^{\frac{p}{p-1}} r^{\frac{p}{p-1}}\right)^{-\frac{1}{m-q}},~~ r\geq 0,
\end{eqnarray}
where $\bar C$ is given by (\ref{barCcom}).

The best constant is given by
\begin{eqnarray}\label{thebestpospe}
C_{q,m,p}=\frac{(q+1-p)^{\frac{2p-1}{p}\theta}\eta_2^{\frac{\eta_2(p-1)\theta}{p(q+1-p)}}}{[2(p-1)]^{\theta}[p(q+1)]^{\frac{(q+1)(p-1)\theta}{q+1-p}}}
\left(\mathcal{B}\left(\frac{(p-1)\eta_2}{p(q+1-p)},\frac{2p-1}{p}\right)\right)^{-\theta},
\end{eqnarray}
where $\theta=\frac{q+1-p}{q(p+(p-1)(q+1))}$ and $\eta_2$ is defined by (\ref{delta1and2}).
\end{thm}
\begin{proof}
In (\ref{closedform1pos}), taking $m=\frac{p(q-1)+1}{p-1}$, we have
\begin{eqnarray}\label{closed1}
u_{c,m}(r)=\left(\frac{m+1}{q+1}\right)^{\frac{1}{m-q}}\left(1-B^{-1}\left(\bar C r;1-\frac{1}{p},\frac{1}{p}-1\right)\right)^{\frac{1}{m-q}},~~ r\geq 0.
\end{eqnarray}
Notice that
\begin{eqnarray}\label{Binversexxpos}
B\left(x;1-\frac{1}{p},\frac{1}{p}-1\right)&=&\int_0^xt^{-\frac{1}{p}}(1-t)^{\frac{1}{p}-2}\,dt=\frac{p}{p-1}\left(\frac{x}{1-x}\right)^{\frac{p-1}{p}},
\end{eqnarray}
which implies that
$$
B^{-1}\left(x;1-\frac{1}{p},\frac{1}{p}-1\right)=1-\frac{1}{1+\left(\frac{p-1}{p}x\right)^{\frac{p}{p-1}}}.
$$
Plugging it into (\ref{closed1}), we deduce that $u_{c,m}$ has the explicit expression (\ref{ucxxpos}).

Furthermore, in the special case $m=\frac{p(q-1)+1}{p-1}$ and $d=1$, we can compute that
$$m+1=\frac{pq}{p-1},~~m-q=\frac{q+1-p}{p-1},~~\eta_1=p(q+1),$$
where $\eta_1$ was defined in (\ref{delta1and2}).
Thus from (\ref{betathe1com}), we can compute that the best constant $C_{q,m,p}$ can be expressed by (\ref{thebestpospe}).
\end{proof}

\begin{rem}
In the paper \cite{del2003optimal}, Del Pino and Dolbeault derived the best constant $\overline C_{q,m,p}$ (see (\ref{barC1po}))
 for the case $d\geq 2$ and $m=\frac{p(q-1)+1}{p-1}$. Formally, if we take $d=1$ in (\ref{barC1po}), we find that $\overline C_{q,m,p}$ exactly equals to $C_{q,m,p}$ given in (\ref{thebestpospe}). Detail verifications are provided in Proposition \ref{d1equa1} of Appendix B. In other words, we extend their results to the dimension $d=1$.
\end{rem}

We remark in the following lemma that the form of the profile (\ref{ucxxpos}) is also preserved on the $p$-Laplace operator in higher dimensions. This fact provides a closed-form solution for the problem (\ref{chachy1})-(\ref{chachybound}) for some special parameters $p,q$ and $m$, see Remark \ref{rm4.4}.

\begin{lem}\label{G}
Let $G(r)=\left(L+r^{\frac{p}{p-1}}\right)^{-\alpha}$ with $L>0$ and $\alpha>0$. Then
\begin{eqnarray}
\Delta_p G=C G^{\frac{(\alpha+1)(p-1)}{\alpha}}-D G^{\frac{(\alpha-1)(p-1)+1}{\alpha}},\label{deltaG}
\end{eqnarray}
where $C:=\left(p(\alpha+1)-d\right)\left(\frac{p}{p-1}\right)^{p-1}\alpha^{p-1}$ and $D:=L(p-1)(\alpha+1)\left(\frac{p}{p-1}\right)^p\alpha^{p-1}$.
\end{lem}
\begin{proof} Similar to Lemma \ref{lmH}, we can compute
\begin{align}
  G'(r)=&-\alpha\frac{p}{p-1} r^{\frac{1}{p-1}} \left(L+r^{\frac{p}{p-1}}\right)^{-\alpha-1},\\
 |G'(r)|^{p-2}G'(r)=&-\left(\alpha\frac{p}{p-1}\right)^{p-1} r \left(L+r^{\frac{p}{p-1}}\right)^{-(\alpha+1)(p-1)},\label{firsord}\\
 (|G'(r)|^{p-2}G'(r))'=&\left(\frac{p}{p-1}\right)^{p-1}\alpha^{p-1}\left(p(\alpha+1)-1\right)\left(L+r^{\frac{p}{p-1}}\right)^{-(\alpha+1)(p-1)}\nonumber\\
 &-L\alpha^{p-1}(\alpha+1)(p-1)\left(\frac{p}{p-1}\right)^{p} \left(L+r^{\frac{p}{p-1}}\right)^{-(\alpha+1)(p-1)-1}.\label{secondord}
 \end{align}
Hence (\ref{firsord}) and (\ref{secondord}) imply
\begin{align*}
\Delta_p G=&(|G'(r)|^{p-2}G'(r))'+\frac{d-1}{r}|G'(r)|^{p-2}G'(r)\nonumber\\
=&\left(\frac{p}{p-1}\right)^{p-1}\alpha^{p-1}\left(p(\alpha+1)-d\right)\left(L+r^{\frac{p}{p-1}}\right)^{-(\alpha+1)(p-1)}\nonumber\\
 &-L\alpha^{p-1}(\alpha+1)(p-1)\left(\frac{p}{p-1}\right)^{p} \left(L+r^{\frac{p}{p-1}}\right)^{-(\alpha+1)(p-1)-1}.
\end{align*}
This completes the proof of (\ref{deltaG}).
\end{proof}
\begin{lem}
Assume $p>1$, $q>p-1$, $m=\frac{p(q-1)+1}{p-1}$ and $\alpha=\frac{p-1}{q+1-p}>0$, then it holds that
$$
\Delta_p G=CG^q-DG^m,
$$
where $C,D$ are given in (\ref{deltaG}). Moreover, take
\begin{eqnarray}
K=\left(\frac{pq-d(q+1-p)}{q+1-p}\right)^{\frac{1}{q+1-p}}\left(\frac{p}{q+1-p}\right)^{\frac{p-1}{q+1-p}},~~L= q^{-1}\left(\frac{pq-d(q+1-p)}{q+1-p}\right)^{\frac{p}{p-1}}.\label{KRpo}
\end{eqnarray}
Then the following facts hold that
\begin{itemize}
\item For the case $p(\alpha+1)-d>0$,
we deduce from (\ref{deltaG1}) that
\begin{eqnarray}\label{uc1po}
u_{c,m}(r):=K H(r)=K \left(L+r^{\frac{p}{p-1}}\right)^{-\alpha}
\end{eqnarray}
is the unique radial positive solution for the following problem
\begin{eqnarray}\label{u1}
\Delta_p u=u^q-u^m
\end{eqnarray}
with $u'(0)=0$ and $\lim_{r\to \infty}u(r)=\lim_{r\to \infty}r'(r)=0$.
\item For $p(\alpha+1)-d=0$, G(r) is a solution of the following equation
\begin{eqnarray}\label{u2}
-\Delta_p G=L\alpha^{p-1}(\alpha+1)(p-1)\left(\frac{p}{p-1}\right)^p G^{\frac{pd-d+p}{d-p}}.
\end{eqnarray}
\end{itemize}

\end{lem}
\begin{proof}
Due to $\alpha=\frac{p-1}{q+1-p}$ and $m=\frac{p(q-1)+1}{p-1}$, we deduce
\begin{eqnarray}\label{relationG}
 (\alpha+1)(p-1)=\alpha q,\quad (\alpha+1)(p-1)+1=\alpha m.
 \end{eqnarray}
Hence (\ref{deltaG}) implies that
\begin{eqnarray}\label{deltaG1}
\Delta_p G=C G^{ q}-D G^{ m}.
\end{eqnarray}
Finally, we can directly verify that $u_{c,m}$ defined in (\ref{uc1po}) is a solution of (\ref{u1}) with boundary conditions $u'(0)=0$, $\lim_{r\to \infty}u(r)=\lim_{r\to \infty}r'(r)=0$ and $G(r)$ is also a solution to the equation (\ref{u2}).
\end{proof}

\begin{rem}\label{rm4.4}
 Assume that $p>1$, $q>p-1$, $m=\frac{p(q-1)+1}{p-1}$ and $\alpha=\frac{p-1}{q+1-p}>0$. The following arguments hold that
\begin{itemize}
\item If $p(\alpha+1)-d>0$, for $u_{c,m}$ defined in (\ref{uc1po}), we can compute $M_c$ (see Lemma \ref{lm5.3})
\begin{eqnarray}\label{M_cpo}
M_c=d~\alpha(d)K^{q+1}L^{\frac{d(p-1)}{p}-\alpha(q+1)}\frac{p-1}{p}\mathcal{B}\left(\alpha(q+1)- \frac{d(p-1)}{p}, \frac{d(p-1)}{p}\right),
\end{eqnarray}
and the best constant is given by
$$
C_{q,m,p}=\left(\frac{1-\theta}{\theta}\right)^{\frac{\theta}{p}}(1-\theta)^{-\frac{1}{m+1}}M_c^{-\frac{\theta}{d}},\quad
\theta=\frac{(q+1-p)d}{q[dp-(d-p)(q+1)]}.
$$
We recover the best constant given by Del Pino and Dolbeault in their celebrated work \cite[Theorem 1.2]{del2003optimal}. See Lemma \ref{lm5.4} in Appendix for detail verification.
\item For $p(\alpha+1)-d=0$, the solution $G(r)$ achieves the equality for the Sobolev inequality (\ref{sobolevineq})(see \cite{Talenti1976}),
where the best constant
$$
C_{d,p}=\pi^{-\frac{1}{2}}d^{-\frac{1}{p}}\left(\frac{p-1}{d-p}\right)^{1-\frac{1}{p}}
\left\{\frac{\Gamma(1+d/2)\Gamma(d)}{\Gamma(d/p)\Gamma(1+d-d/p)}\right\}^{1/d}.
$$
In fact, we easily verify that $C_{d,p}=C_{q,m,p}$ by taking $q=\frac{d(p-1)}{d-p}$, $m=\sigma$ in \eqref{massvsbeta1}.
\end{itemize}
\end{rem}

\section{the limit behavior as $m\to \infty$ in the best constant problem}

In this section, we show the limit behavior as $m\to \infty$ in the best constant problem for d=1, which implies that the closed relation between the functional inequalities (\ref{h sharp inequality1}) and  (\ref{LinftyMlinftyq}) in the one-dimension case. The result is given by the following theorem.
\begin{thm}\label{cor51}
Let $u_{c,m}$ and $u_{c,\infty}$ be respectively the non-increasing radial solution of the problem (\ref{FBVP1})-(\ref{FBVPbound}) and the problem (\ref{iMinfFBVP1})-(\ref{iMinfFBVPbound})(or the problem (\ref{chachy1})-(\ref{chachybound}) and the problem (\ref{iMinfchachy1})-(\ref{iMinfchachybound}) ). Let $R_m$ and $R_{\infty}$ be the free boundaries for  $u_{c,m}$ and $u_{c,\infty}$ respectively. Then the following facts hold for $d=1$
$$
u_{c,m}(r)\to u_{c,\infty}(r)\quad \mbox{ for any } r>0,\quad C_{q,m,p}\to C_{q,\infty,p},\quad \mbox{ as } m\to \infty.
$$
Moreover, for $q<p-1$ we have $R_m\to R_{\infty}$ as $m\to\infty$.
\end{thm}
\begin{proof}
{\bf Step 1.}
 In this step, we prove
\begin{eqnarray}\label{ulimit}
\lim_{m\to \infty}u_{c,m}(r)=u_{c,\infty}(r)\quad \mbox{ for any } r>0.
\end{eqnarray}

For any fixed $0<u<1$, denote $r_{c,m}(u)$ as a inverse function of $u_{c,m}(r)$. We recall (\ref{solveu}), i.e.,
\begin{eqnarray}\label{solveu1}
\left(\frac{p}{p-1}\right)^{1/p}r_{c,m}(u)=\frac{(q+1)^{1/p}}{m-q}\left(\frac{m+1}{q+1}\right)^{\frac{p-q-1}{p(m-q)}}
\int_0^{1-\frac{q+1}{m+1}u^{m-q}}y^{-1/p}(1-y)^{\frac{p-q-1}{p(m-q)}-1}\,dy.
\end{eqnarray}
For the case $q=p-1$, we have $\frac{p-q-1}{p(m-q)}=0$, (\ref{solveu1}) becomes into
\begin{eqnarray}\label{solveu2}
\left(\frac{1}{p-1}\right)^{1/p}r_{c,m}(u)=\frac{1}{m-q}
\int_0^{1-\frac{q+1}{m+1}u^{m-q}}y^{-1/p}(1-y)^{-1}\,dy.
\end{eqnarray}
Since $\left(\frac{m+1}{q+1}\right)^\frac{1}{m-q}\to 1$ as $m\to\infty$, for $m$ sufficient large, we have that $0<u< \left(\frac{m+1}{q+1}\right)^\frac{1}{m-q}$.

Taking the limit for (\ref{solveu2}) as $m\to \infty$ and using the L'H\^opital's rule, we deduce that
\begin{align}\label{solveu3}
\left(\frac{1}{p-1}\right)^{1/p}\lim_{m\to\infty}r_{c,m}(u)=&\lim_{m\to\infty}\frac{\int_0^{1-\frac{q+1}{m+1}u^{m-q}}y^{-1/p}(1-y)^{-1}\,dy}{m-q}\nonumber\\
=&\lim_{m\to\infty}\left(1-\frac{q+1}{m+1}u^{m-q}\right)^{-1/p}\left(\frac{1}{m+1}-\ln u\right)\nonumber\\
=&-\ln u.
\end{align}
Hence
$$
\lim_{m\to\infty}r_{c,m}(u)=-\left(p-1\right)^{1/p}\ln u.
$$
We denote the above limit function as $r_{c,\infty}:(0,1)\mapsto (0,\infty)$, $u\to r_{c,\infty}(u):=-\left(p-1\right)^{1/p}\ln u$.
Denote inverse function as $u_{c,\infty}:(0,\infty)\mapsto (0,1)$, $r\to u_{c,\infty}(r)=e^{-(p-1)^{-\frac{1}{p}}r}$, and $u_{c,\infty}(r)$ is the solution to the equation (\ref{iMinfchachy1}) by (\ref{ipq1u}).
Noticing $u_{c,m}(0)\to 1$ as $m\to \infty$ and $u_{c,m}(r)>0$ is continuous and strictly decreasing in $(0,\infty)$, hence we have
$$
\lim_{m\to\infty} u_{c,m}(r)=u_{c,\infty}(r)\quad \mbox{ for any } r\in(0,+\infty).
$$

For the case $q<p-1$, let $\tilde u_{c,m}(r)=\frac{u_{c,m}(r)}{u_{c,m}(0)}$, $u_{c,m}(0)=\left(\frac{m+1}{q+1}\right)^{\frac{1}{m-q}}$. Then $\tilde u_{c,m}(0)=1$, and $\tilde u_{c,m}(r)$ satisfies the following equation
\begin{eqnarray}
&& (|\tilde{u}'|^{p-2}\tilde{u}')' + \left(\frac{m+1}{q+1}\right)^{\frac{m+1-p}{m-q}}\tilde{u}^m= \left(\frac{m+1}{q+1}\right)^{\frac{q+1-p}{m-q}}\tilde{u}^{q}, ~~0<r<R_m,\label{tiinftyfunR11}\\
&& \tilde{u}(0)=1,\quad \tilde{u}(R_m)=\tilde{u}'(R_m)=0.\label{tiinftyfunRbound11}
\end{eqnarray}

For any fixed $0<\tilde{u}<1$, denote the inverse function of $\tilde{u}_{c,m}(r)$ as $r_{c,m}(\tilde{u})$. Similar to (\ref{firstintex}), we know that
\begin{eqnarray}\label{rctilde}
\frac{p-1}{p}\left|\frac{1}{r_{c,m}'(\tilde{u})}\right|^p+\left(\frac{m+1}{q+1}\right)^{\frac{m+1-p}{m-q}}\frac{1}{m+1}\tilde{u}^{m+1}-
\left(\frac{m+1}{q+1}\right)^{\frac{q+1-p}{m-q}}\frac{1}{q+1}\tilde{u}^{q+1}=0.
\end{eqnarray}
From (\ref{rctilde}), we deduce that
\begin{eqnarray}\label{cderve}
-\left(\frac{p}{p-1}\right)^{\frac{1}{p}}r_{c,m}'(\tilde{u})
=\frac{1}{\left(\frac{u_{c,m}(0)^{q+1-p}}{q+1}\tilde{u}^{q+1}-\frac{u_{c,m}(0)^{m+1-p}}{m+1}\tilde{u}^{m+1}\right)^{1/p}}.
\end{eqnarray}
Furthermore, integrating (\ref{cderve}) respect to $\tilde{u}$ from $\tilde{u}$ to $1$ and plugging $u_{c,m}(0)=\left(\frac{m+1}{q+1}\right)^{\frac{1}{m-q}}$ give that
\begin{eqnarray}\label{r}
\left(\frac{p}{p-1}\right)^{\frac{1}{p}}r_{c,m}(\tilde{u})
=\int_{\tilde{u}}^1\frac{1}{\left(\frac{u_{c,m}(0)^{q+1-p}}{q+1}s^{q+1}\right)^{1/p}\left(1-s^{m-q}\right)^{1/p}}\,ds.
\end{eqnarray}
Making variable substitution $y=1-s^{m-q}$, we have
\begin{eqnarray}\label{r1}
\left(\frac{p}{p-1}\right)^{\frac{1}{p}}r_{c,m}(\tilde{u})
=\frac{(q+1)^{1/p}}{m-q}\left(\frac{m+1}{q+1}\right)^{\frac{p-q-1}{p(m-q)}}\int_{0}^{1-\tilde{u}^{m-q}}y^{-\frac{1}{p}}(1-y)^{\frac{p-q-1}{p(m-q)}-1}\,dy.
\end{eqnarray}
Noticing that $\left(\frac{m+1}{q+1}\right)^{\frac{p-q-1}{p(m-q)}}\to 1$ as $m\to\infty$, thus it holds
\begin{eqnarray}\label{gongshi}
\left(\frac{p}{(p-1)(q+1)}\right)^{1/p}\lim_{m\to\infty}r_{c,m}(\tilde{u})
=\lim_{m\to\infty}\frac{\int_{0}^{1-\tilde{u}^{m-q}}y^{-\frac{1}{p}}(1-y)^{\frac{p-q-1}{p(m-q)}-1}\,dy}{m-q}.
\end{eqnarray}
Since $0<\tilde{u}<1$, we get $1-\tilde{u}^{m-q}>\frac{1}{2}$ if $m$ is large. Hence
\begin{align}\label{gongshi1}
\lim_{m\to\infty}\frac{\int_{0}^{1-\tilde{u}^{m-q}}y^{-\frac{1}{p}}(1-y)^{\frac{p-q-1}{p(m-q)}-1}\,dy}{m-q}
=&\lim_{m\to\infty}\frac{\left(\int_{0}^{1/2}+\int_{1/2}^{1-\tilde{u}^{m-q}}\right)y^{-\frac{1}{p}}(1-y)^{\frac{p-q-1}{p(m-q)}-1}\,dy}{m-q}\nonumber\\
=&\lim_{m\to\infty}\frac{\int_{1/2}^{1-\tilde{u}^{m-q}}y^{-\frac{1}{p}}(1-y)^{\frac{p-q-1}{p(m-q)}-1}\,dy}{m-q}\nonumber\\
=&\lim_{m\to\infty}\frac{-p\int_{1/2}^{1-\tilde{u}^{m-q}}y^{-\frac{1}{p}}d(1-y)^{\frac{p-q-1}{p(m-q)}}}{p-q-1}.
\end{align}
Using the integration by parts, we know that
\begin{align}\label{gongshi2}
\int_{1/2}^{1-\tilde{u}^{m-q}}y^{-\frac{1}{p}}d(1-y)^{\frac{p-q-1}{p(m-q)}}
=& y^{-\frac{1}{p}}(1-y)^{\frac{p-q-1}{p(m-q)}}\Big|_{1/2}^{1-\tilde{u}^{m-q}}\nonumber\\
&-\int_{1/2}^{1-\tilde{u}^{m-q}}
\left(y^{-\frac{1}{p}}\right)'(1-y)^{\frac{p-q-1}{p(m-q)}}\,dy.
\end{align}
We can directly check
\begin{eqnarray}\label{gongshi3}
\lim_{m\to\infty}\int_{1/2}^{1-\tilde{u}^{m-q}}
\left(y^{-\frac{1}{p}}\right)'(1-y)^{\frac{p-q-1}{p(m-q)}}\,dy=\lim_{m\to\infty}\int_{1/2}^{1-\tilde{u}^{m-q}}
\left(y^{-\frac{1}{p}}\right)'\,dy=1-2^{\frac{1}{p}}.
\end{eqnarray}
Plugging (\ref{gongshi2}) and (\ref{gongshi3}) into (\ref{gongshi1}), we obtain
\begin{eqnarray}\label{gongshi3'}
\lim_{m\to\infty}\frac{\int_{0}^{1-\tilde{u}^{m-q}}y^{-\frac{1}{p}}(1-y)^{\frac{p-q-1}{p(m-q)}-1}\,dy}{m-q}
=\frac{p}{p-q-1}\left(1-\tilde{u}^{\frac{p-q-1}{p}}\right).
\end{eqnarray}
Therefore, (\ref{gongshi}) and (\ref{gongshi3'}) imply that
\begin{align}\label{gongshi4}
\lim_{m\to\infty}r_{c,m}(\tilde{u})
=&\left(\frac{p}{(p-1)(q+1)}\right)^{-1/p}\frac{p}{p-q-1}\left(1-\tilde{u}^{\frac{p-q-1}{p}}\right)\nonumber\\
=& R_{\infty}\left(1-\tilde{u}^{\frac{p-q-1}{p}}\right),
\end{align}
where $R_{\infty}$ is the same to $R$ defined in (\ref{ipxiaoq1u11}).

We denote the above limit function as $r_{c,\infty}:(0,1)\mapsto (0,R_{\infty})$:
$$
\tilde{u}\to r_{c,\infty}(\tilde{u}):=R_{\infty}\left(1-\tilde{u}^{\frac{p-q-1}{p}}\right).
$$
Denote its inverse function as $u_{c,\infty}:(0,R_{\infty})\mapsto (0,1)$. Then it is given by
\begin{eqnarray}\label{ucinfty}
r\to u_{c,\infty}(r)=\left(1-\frac{r}{R}\right)^{\frac{p}{p-q-1}}\quad\mbox{ for }0<r<R_{\infty}.
\end{eqnarray}

Next we show that $\lim_{m\to \infty}R_{m}=R_{\infty}$, $R_{m}$ is the same to $R$ given by (\ref{Rvaluecom}). Indeed,
\begin{align}\label{Rvaluecom11}
\lim_{m\to \infty}R_m=&\lim_{m\to \infty}\left(\frac{p-1}{p}\right)^{\frac{1}{p}}\frac{(m+1)^{\frac{p-q-1}{p(m-q)}}(q+1)^{\frac{1}{p}-\frac{p-q-1}{p(m-q)}}}{m-q}
\int_{0}^{1}y^{-\frac{1}{p}}(1-y)^{\frac{p-q-1}{p(m-q)}-1}\,dy\nonumber\\
=&\left(\frac{(p-1)(q+1)}{p}\right)^{\frac{1}{p}}\lim_{m\to \infty}\frac{
\int_{0}^{1}y^{-\frac{1}{p}}(1-y)^{\frac{p-q-1}{p(m-q)}-1}\,dy}{m-q}.
\end{align}
Due to $q<p-1$ we have
\begin{eqnarray}\label{remader}
\lim_{m\to \infty}\frac{
\int_{1-\tilde{u}^{m-q}}^{1}y^{-\frac{1}{p}}(1-y)^{\frac{p-q-1}{p(m-q)}-1}\,dy}{m-q}= \lim_{m\to \infty}\frac{
\int_{1-\tilde{u}^{m-q}}^{1}(1-y)^{\frac{p-q-1}{p(m-q)}-1}\,dy}{m-q}=\frac{p}{p-q-1}\tilde{u}^{\frac{p-q-1}{p}}.
\end{eqnarray}
Combining (\ref{gongshi3'}) and (\ref{remader}), we have
$$
\lim_{m\to \infty}R_m=\left(\frac{p}{(p-1)(q+1)}\right)^{-1/p}\frac{p}{p-q-1}=R_{\infty}.
$$
Noticing that $u_{c,m}(0)\to 1$ as $m\to \infty$ and $u_{c,m}(r)>0$ is continuous and strictly decreasing in $(0,R_m)$, hence we have from (\ref{gongshi4})
$$
\lim_{m\to\infty} u_{c,m}(r)=u_{c,\infty}(r)\quad \mbox{ for any } r\in(0,R_{\infty}).
$$

For the case $q>p-1$, the proof of (\ref{ulimit}) is exactly same with the case $q<p-1$. We omit the details.

{\bf Step 2.} We prove that
\begin{eqnarray}\label{mtoinftypbudengq1}
\lim_{m\to\infty}C_{q,m,p}=C_{q,\infty,p}.
\end{eqnarray}

For the case $q=p-1$,
define
\begin{equation*}
  U(r):=\left\{
  \begin{aligned}
  &\left(\frac{m+1}{q+1}\right)^{\frac{1}{m-q}},  && \mbox{ if } 0<r\leq r'_*,\\
  &e^{-C(p)r}, && \mbox{ if } r>r'_*.
  \end{aligned}
     \right.
\end{equation*}
 From (\ref{pq1oinftypro}), we have $u_{c,m}(r)\leq U(r)$ for any $r>0$.
We directly compute $\|U(|x|)\|_{L^{q+1}}<\infty$.

For the case $q>p-1$, define
\begin{equation}
  U(r):=\left\{
  \begin{aligned}
  &\left(\frac{m+1}{q+1}\right)^{\frac{1}{m-q}},  && \mbox{ if } 0<r\leq \bar r_*,\\
  &C(p,q) r^{-\frac{p}{q+1-p}}, && \mbox{ if } r>\bar r_*.
  \end{aligned}
     \right.
\end{equation}
We have $u_{c,m}(r)\leq U(r)$ for any $r>0$ due to (\ref{pdaq1oinftypro}) and $u_{c,m}(r)\leq u_{c,m}(0)=\left(\frac{m+1}{q+1}\right)^{\frac{1}{m-q}}$. A direct computation gives $\|U(r)\|_{L^{q+1}}<\infty$.

For the case $q<p-1$, we know that the solution $u_{c,m}(r)$ has a finite support $(0,R_m)$, $R_m\to R_{\infty}$ as $m\to \infty$, where $R_m$ is defined by (\ref{Rvaluecom}). Noticing that $u_{c,m}(r)\leq \left(\frac{m+1}{q+1}\right)^{\frac{1}{m-q}}\leq 2$ for $m$ large, we have $\int_0^{R_{\infty}}2^{q+1}\,dr\leq C(p,q)$. Then for the above three cases, the dominated convergence theorem implies that
\begin{eqnarray}\label{pbedengq1u}
\lim_{m\to \infty}M_c(u_{c,m})=\lim_{m\to \infty}\int_{\mathbb{R}}u^{q+1}_{c,m}(x)\,dx=\int_{\mathbb{R}}u^{q+1}_{c,\infty}(x)\,dx= M_c(u_{c,\infty}).
\end{eqnarray}
Hence from (\ref{bestconstant}) and (\ref{massvsbeta1}), we obtain (\ref{mtoinftypbudengq1}).
\end{proof}

\appendix

\section{Proof of Proposition 1.1 for the case $q\geq p-1$}

 In this section we use a modified Strauss' inequality to prove Proposition 1.1 for the case $q\geq p-1$. The proof for the case $q\geq p-1$ is similar to that of the case $q<p-1$ in the paper \cite[subsection 2.1]{LW1}. However some modifications of the proof are needed, and hence here we provide those necessary modifications.
\begin{lem}(Modified Strauss' inequality)
Assume $d\geq 2$ and $u\in X^*_{rad}$. Then for a.e. $x\in \mathbb{R}^d\setminus\{0\}$, the following inequality holds
\begin{eqnarray}\label{straussinequ1}
|u(x)|\leq C(d,p,q) |x|^{\frac{1-d}{\nu}}\|u\|^{\frac{\nu-1}{\nu}}_{L^{q+1}(\mathbb{R}^d)}\|\nabla u\|^{\frac{1}{\nu}}_{L^p(\mathbb{R}^d)},\quad \nu=\frac{(q+1)(p-1)}{p}+1.
\end{eqnarray}
\end{lem}
\begin{proof}
The condition $u\in X^*_{rad}$ implies $u\in L^{q+1}(\mathbb{R}^d)$, $\nabla u\in L^{p}(\mathbb{R}^d)$ and $u(x)=u(r)$, where $r=|x|$, and parameters $q,p$ are defined by (\ref{qandm}). Hence
\begin{eqnarray}\label{der}
\frac{d}{dr}\left(r^{d-1}|u(x)|^{\nu}\right)=\nu |u|^{\nu-2}u\frac{du}{dr}r^{d-1}+(d-1)|u|^{\nu}r^{d-2}.
\end{eqnarray}
Integrating (\ref{der}) from $r$ to $\infty$, one has
\begin{align}\label{der1}
r^{d-1}|u(r)|^{\nu}=&-\int_r^{\infty}\nu |u|^{\nu-2}u\frac{du}{ds}s^{d-1}\,ds-\int_r^{\infty}(d-1)|u|^{\nu}s^{d-2}\,ds\nonumber\\
\leq& \nu\int_{\mathbb{R}^d} |u|^{\nu-1}|\nabla u|\,dx.
\end{align}
Using the H\"older inequality, we obtain from (\ref{der1})
\begin{eqnarray}\label{der2}
r^{d-1}|u(x)|^{\nu}\leq C(d,p,q)\|u\|^{\nu-1}_{L^{q+1}}\|\nabla u\|_{L^{p}}.
\end{eqnarray}
Thus (\ref{straussinequ1}) holds.
\end{proof}

\begin{proof}[\bf The proof of Proposition 1.1]
First, we know that $J(u)$ defined by (\ref{Jh}) has the following scaling invariant:
\begin{eqnarray}\label{invariant}
J(u_{\mu,\lambda})= J(u),\quad u_{\mu,\lambda}=\mu u(\lambda x),\quad\mu,\lambda>0.
\end{eqnarray}

By (\ref{radminiexi}) and (\ref{invariant}), there exists a minimizing sequence $\{u_{k}\}\in X^*_{rad}$ satisfying
\begin{eqnarray}\label{normlized}
\int_{\mathbb{R}^d}u_k^{q+1}\,dx=\int_{\mathbb{R}^d}u^{m+1}_k\,dx=1
\end{eqnarray}
 such that
\begin{eqnarray}\label{minimizerseq}
\lim_{k\to \infty}J(u_{k})= \lim_{k\to \infty} \int_{\mathbb{R}^d}|\nabla u_k|^p\, dx=\beta=\inf_{u\in X^*_{rad}}J(u).
\end{eqnarray}

For $d\geq 2$, there exist a subsequence (is still denoted as $u_{k}$) and $u_0\in  X^*_{rad}$ such that
\begin{eqnarray}
&& u_k\rightharpoonup u_0, \quad\mbox{ in } L^{q+1}\cap L^{m+1}(\mathbb{R}^d),\\
&& \nabla u_k\rightharpoonup \nabla u_0, \quad\mbox{ in } L^p(\mathbb{R}^d).
\end{eqnarray}
Hence by Fatou's lemma, we have
\begin{eqnarray}
&&\|\nabla u_0\|_{L^{p}}\leq \liminf_{k\rightarrow\infty} \|\nabla u_k\|_{L^{p}}=\beta^{1/p},\label{nablawek}\\
&&\|u_0\|_{L^{q+1}}\leq \liminf_{k\rightarrow\infty} \|u_k\|_{L^{q+1}}=1.\label{q+1weak}
\end{eqnarray}

 On the one hand, (\ref{straussinequ1}) tells us for any radial function $u_k\in X^*_{rad}$, it holds that
\begin{eqnarray}\label{straussineq}
 |u_k|\leq C |x|^{\frac{1-d}{\nu}}\|u_k\|^{\frac{\nu-1}{\nu}}_{L^{q+1}}\|\nabla u_k\|^{\frac{1}{\nu}}_{L^{p}}~~ \mbox{ for } |x|>0,\,d\geq 2,\,\nu=\frac{(q+1)(p-1)}{p}+1,
\end{eqnarray}
 where $C$ is only dependent of $d,q$ and $p$. By (\ref{straussineq}), we know that for $s>\frac{\nu d}{d-1}$, it holds
$$\int_{|x|>R}u_k^s\,dx \leq  C\int_{|x|>R} |x|^{-\frac{s(d-1)}{\nu}}\,dx = C\int^{+\infty}_{R} r^{-\frac{s(d-1)}{\nu}+(d-1)}\,dr=C R^{d-\frac{s(d-1)}{\nu}}.$$
Noticing $d-\frac{s(d-1)}{\nu}<0$, for any $\varepsilon>0$, there is $R^1_{\varepsilon}$ such that
$$
\int_{|x|>R^1_{\varepsilon}}u_k^s\,dx\leq \varepsilon.
$$
Since $q<\sigma-1$, we have $q+1<\frac{\nu d}{d-1}$. Thus if $q+1<s\leq \frac{\nu d}{d-1}$, using the interpolation inequality, it holds that
$$
\|u_k\|^s_{L^s(B^c(0,R))}\leq \|u_k\|^{s(1-\theta)}_{L^{q+1}(B^c(0,R))}\|u_k\|^{s\theta}_{L^{\gamma}(B^c(0,R))}\quad \mbox{ for } \gamma>\frac{\nu d}{d-1}.
$$
So, there exists a $R^2_{\varepsilon}$ such that
$$\|u_k\|^s_{L^s(B^c(0,R^2_{\varepsilon}))}\leq \varepsilon.$$
Taking $R=2\max\{R^1_{\varepsilon},R^2_{\varepsilon}\}$, one has
\begin{eqnarray}\label{outersmall}
\|u_k\|^s_{L^s(B^c(0,R))}\leq \varepsilon.
\end{eqnarray}

On the other hand, for any fixed $R>0$, $u_k\in X^*_{rad}$ implies that $u_k\in L^p(B(0,R+1))$ due to $q\geq p-1$. The Sobolev embedding theorem gives
$$
W^{1,p}(B(0,R+1))\hookrightarrow\hookrightarrow L^s(B(0,R+1)),
$$
provided that $1< s<p^*=\frac{pd}{d-p}$ if $p<d$, and $s\geq 1$ for $p\geq d$.
Together with (\ref{outersmall}) implies that there is a strong convergence subsequence of $u_k$ (still denoted as $u_k$)
\begin{eqnarray}\label{strongcon}
  u_k\rightarrow u_0,\quad \mbox{ in }  L^{s}(\mathbb{R}^d).
\end{eqnarray}
 Since $1<m+1<\frac{pd}{d-p}$ for $p<d$, and $1<m+1<\infty$ for $p\geq d$, we know from (\ref{strongcon})
 \begin{eqnarray}
 &&\int_{\mathbb{R}^d}|u_k-u_0|^{m+1}\,dx\rightarrow 0, \mbox{ as } k\rightarrow\infty,\label{lm+1strong1}\\
 && \|u_0\|_{L^{m+1}}= \lim_{k\rightarrow\infty} \| u_k\|_{L^{m+1}}=1.\label{lm+1strong2}
\end{eqnarray}
From (\ref{nablawek}), (\ref{q+1weak}) and (\ref{lm+1strong2}), we deduce
$$
J(u_0)=\frac{\left(\int_{\mathbb{R}^d}u_0^{q+1}\,dx\right)^{\frac{a-p/2}{q+1}}\int_{\mathbb{R}^d}|\nabla u_0|^p\,dx}{\left(\int_{\mathbb{R}^d}u_0^{m+1}\,dx\right)^{\frac{a+p/2}{m+1}}}\leq\beta.
$$
Noticing $J(u_0)\geq \beta$ by the definition of $\beta$, one knows that
\begin{eqnarray}\label{limitfunctionva}
J(u_0)=\beta,\quad\|u_0\|_{L^{q+1}}=\|u_0\|_{L^{m+1}}=1,\quad \|\nabla u_0\|^p_{L^{p}}=\beta.
\end{eqnarray}

For the 1-d case, Sz. Nagy \cite{Nagy1941} proved that the minimizer of $J(u)$ exists and it can be represented in terms of an incomplete Beta function, and obtained the celebrated Nagy's inequalities. Hence for $d\geq 1$, there exists an optimal radial decreasing function $u_0$ such that
 $$
\beta=\inf_{u\in X^*_{rad}}J(u)=J(u_0).
 $$
\end{proof}

\section{Verification of the best constant for some special cases}
We first prove (\ref{M_c}) for the case $q+1=\frac{pm}{p-1}$.
\begin{lem}\label{lmApp1}
For $u_{c,m}(r)$ in (\ref{uc1}), we have
$$
M_c=\int_{\mathbb{R}^d} u_{c,m}(|x|)^{q+1}\, dx=d\alpha(d)K^{q+1}R^{d+\frac{mp^2}{(p-1)(p-m-1)}}\frac{p-1}{p}\mathcal{B}\left( \frac{d(p-1)}{p}, \frac{pm}{p-m-1}+1\right),
$$
where $\alpha(d)$ is the volume of the $d$-dimensional unit ball.
\end{lem}
\begin{proof}
A simple computation for $u_{c,m} $ gives that
\begin{align}
M_c=&~d\alpha(d)\int_0^{\infty} u_{c,m}(r)^{q+1}r^{d-1}\, dr\nonumber\\
=&~d\alpha(d)K^{q+1}\int_0^{\infty} \left(R^{\frac{p}{p-1}}-r^{\frac{p}{p-1}}\right)_{+}^{\alpha(q+1)}r^{d-1}\, dr\nonumber\\
=&~d\alpha(d)K^{q+1}R^{\frac{p(q+1)}{p-m-1}}\int_0^{\infty} \left(1-\left(\frac{r}{R}\right)^{\frac{p}{p-1}}\right)_{+}^{\alpha(q+1)}r^{d-1}\, dr.\label{M1}
\end{align}
Let $x=\left(\frac{r}{R}\right)^{\frac{p}{p-1}}$. Then we have $r=R x^{\frac{p-1}{p}}$ and $dr=R\frac{p-1}{p} x^{\frac{p-1}{p}-1}\,dx$. Hence we obtain
\begin{eqnarray}
M_c=d\alpha(d)\frac{p-1}{p}K^{q+1}R^{\frac{p(q+1)}{p-m-1}+d}\int_0^{1} \left(1-x\right)^{\frac{(p-1)(q+1)}{p-m-1}}x^{\frac{d(p-1)}{p}-1}\, dx.\label{M2}
\end{eqnarray}
Noticing that $q+1=\frac{pm}{p-1}$, we have
\begin{align*}
M_c=&~d\alpha(d)\frac{p-1}{p}K^{q+1}R^{\frac{p^2m}{(p-1)(p-m-1)}+d}\int_0^{1} \left(1-x\right)^{\frac{pm}{p-m-1}}x^{\frac{d(p-1)}{p}-1}\, dx\\
=&~d\alpha(d)\frac{p-1}{p}K^{q+1}R^{\frac{p^2m}{(p-1)(p-m-1)}+d}\mathcal{B}\left(\frac{d(p-1)}{p},\frac{pm}{p-m-1}+1\right).
\end{align*}
\end{proof}

We recall the best constant $C_{q,m,p}$ given in Remark \ref{rm4.3}
\begin{eqnarray}\label{formula11}
C_{q,m,p}=\left(\frac{1-\theta}{\theta}\right)^{\frac{\theta}{p}}(1-\theta)^{-\frac{1}{m+1}}M_c^{-\frac{\theta}{d}},
\end{eqnarray}
where $M_c$ is given by (\ref{M_c}).
We convert the best constant $\overline C_{q,m,p}$ given in \cite{del2003optimal} into our parameters and it is given by (\ref{barC1})-(\ref{thetaanddelta}).
\begin{lem}\label{lmApp}
For $q+1=\frac{pm}{p-1}$ and $d\geq 2$, the best constant $C_{q,m,p}$ in (\ref{formula11}) is exactly equal to the best constant $\overline C_{q,m,p}$ given in \cite{del2003optimal}.
\end{lem}
\begin{proof}
By (\ref{thetaanddelta}), we can compute
\begin{align}\label{C1}
C_{q,m,p}&=\left(\frac{1-\theta}{\theta}\right)^{\frac{\theta}{p}}(1-\theta)^{-\frac{1}{m+1}}M_c^{-\frac{\theta}{d}}\nonumber\\
&=\left(\frac{\eta}{d(p-m-1)}\right)^{\frac{\theta}{p}}\left(\frac{\eta}{(m+1)(d(p-m-1)+pm)}\right)^{-\frac{1}{m+1}}m^{\frac{\theta}{p}
-\frac{1}{m+1}}M_c^{-\frac{\theta}{d}}.
\end{align}
On the other hand, from (\ref{KR}), we deduce
\begin{eqnarray}\label{KR1}
K^{q+1}R^{d+\frac{mp^2}{(p-1)(p-m-1)}}=m^{-\frac{d(p-1)}{p}-\frac{pm}{p-m-1}}\left(\frac{p}{p-m-1}\right)^{\frac{pm}{m-p+1}}
\left(\frac{d}{(m+1)\theta}\right)^{\frac{pm}{p-m-1}+d}.
\end{eqnarray}
Noticing that the relations hold
\begin{eqnarray*}
&&\left({-\frac{d(p-1)}{p}-\frac{pm}{p-m-1}}\right)\frac{\theta}{d}=\frac{\theta}{p}-\frac{1}{m+1},\\
&&\frac{d}{(m+1)\theta}=\frac{d(p-m-1)+pm}{p-m-1},\quad \left(\frac{pm}{p-m-1}+d\right)\frac{\theta}{d}=\frac{1}{m+1},
\end{eqnarray*}
we have from (\ref{KR1})
\begin{eqnarray}\label{KR2}
\left(K^{q+1}R^{d+\frac{mp^2}{(p-1)(p-m-1)}}\right)^{-\frac{\theta}{d}}=m^{-\frac{\theta}{p}+\frac{1}{m+1}}\left(\frac{p}{p-m-1}\right)^{\frac{\theta pm}{d(p-m-1)}}
\left(\frac{d(p-m-1)+pm}{p-m-1}\right)^{-\frac{1}{m+1}}
\end{eqnarray}
Using (\ref{M_c}) and (\ref{KR2}), we obtain
\begin{align}\label{formula10}
M_c^{-\frac{\theta}{d}}=&m^{-\frac{\theta}{p}+\frac{1}{m+1}}\left(\frac{p}{p-m-1}\right)^{\frac{\theta pm}{d(p-m-1)}}
\left(\frac{d(p-m-1)+pm}{p-m-1}\right)^{-\frac{1}{m+1}}\nonumber\\
&\cdot \left(\frac{p}{p-1}\right)^{\frac{\theta}{d}}\left(\frac{\Gamma(\frac{d}{2}+1)}{d\pi^{d/2}}\right)^{\frac{\theta}{d}}
\left(\frac{\Gamma(\frac{pm}{p-m-1}+d\frac{p-1}{p}+1)}{\Gamma(1+\frac{pm}{p-m-1})\Gamma(d\frac{p-1}{p})}\right)^{\frac{\theta}{d}}.
\end{align}
Plugging (\ref{formula10}) in (\ref{C1}) gives
\begin{align}\label{C2}
C_{q,m,p}=&\left(\frac{\eta}{d(p-m-1)}\right)^{\frac{\theta}{p}}\left(\frac{\eta}{m+1}\right)^{-\frac{1}{m+1}}
 \left(\frac{p}{p-m-1}\right)^{\frac{\theta pm}{d(p-m-1)}}
\left(\frac{1}{p-m-1}\right)^{-\frac{1}{m+1}}\nonumber\\
&\cdot \left(\frac{p}{p-1}\right)^{\frac{\theta}{d}}\left(\frac{\Gamma(\frac{d}{2}+1)}{d\pi^{d/2}}\right)^{\frac{\theta}{d}}
\left(\frac{\Gamma(\frac{pm}{p-m-1}+d\frac{p-1}{p}+1)}{\Gamma(1+\frac{pm}{p-m-1})\Gamma(d\frac{p-1}{p})}\right)^{\frac{\theta}{d}}\nonumber\\
=&\eta^{\frac{\theta}{p}-\frac{1}{m+1}}(d(p-m-1))^{-\frac{\theta}{p}}(m+1)^{\frac{1}{m+1}}p^{\frac{1}{m+1}-\theta}(p-m-1)^{\theta}\nonumber\\
&\cdot \left(\frac{p}{p-1}\right)^{\frac{\theta}{d}}\left(\frac{\Gamma(\frac{d}{2}+1)}{d\pi^{d/2}}\right)^{\frac{\theta}{d}}
\left(\frac{\Gamma(\frac{pm}{p-m-1}+d\frac{p-1}{p}+1)}{\Gamma(1+\frac{pm}{p-m-1})\Gamma(d\frac{p-1}{p})}\right)^{\frac{\theta}{d}}.
\end{align}
Moreover, since $q+1=\frac{pm}{p-1}$, one has
$$
\frac{\theta}{p}-\frac{1}{m+1}=\frac{\theta}{d}-\frac{1-\theta}{q+1}.
$$
Hence the best constant $C_{q,m,p}$ in (\ref{C2}) is exactly equal to the best constant $\overline C_{q,m,p}$ given in (\ref{barC1}).

\end{proof}

In fact, for $d=1$, we find that $\overline C_{q,m,p}$ exactly equals to $C_{q,m,p}$ given in (\ref{thebest}). Detail verifications are provided in the following proposition.
\begin{prop}\label{d1equa}
For $d=1$ and $q=\frac{pm-p+1}{p-1}$, the best constant $\overline C_{q,m,p}$ defined in (\ref{barC1}) exactly equals to $C_{q,m,p}$ in (\ref{thebest}).
\end{prop}
\begin{proof}
Using the formula $\mathcal{B}(a,b)=\frac{\Gamma(a)\Gamma(b)}{\Gamma(a+b)}$ for (\ref{thebest}), we have
\begin{align}\label{thebest1}
C_{q,m,p}=&\frac{(p-m-1)^{\frac{2p-1}{p}\theta}[p(m+1)]^{\frac{1}{m+1}}}{[2(p-1)]^{\theta}
\eta_1^{\frac{\eta_1}{p(m+1)^2}}}\left(\frac{\eta_1}{p-m-1}\right)^{\theta}
\left(\frac{\Gamma\left(\frac{pm}{p-m-1}+1\right)\Gamma\left(\frac{p-1}{p}\right)}
{\Gamma\left(\frac{pm}{p-m-1}+1+\frac{p-1}{p}\right)}\right)^{-\theta}\nonumber\\
=&~2^{-\theta}(p-1)^{-\theta}(p-m-1)^{\frac{p-1}{p}\theta}[p(m+1)]^{\frac{1}{m+1}}
\eta_1^{\theta-\frac{p+(p-1)(m+1)}{p(m+1)^2}}\nonumber\\
&\qquad\cdot \left(\frac{\Gamma\left(\frac{pm}{p-m-1}+1\right)\Gamma\left(\frac{p-1}{p}\right)}
{\Gamma\left(\frac{pm}{p-m-1}+1+\frac{p-1}{p}\right)}\right)^{-\theta}.
\end{align}
Formally, if taking $d=1$ in (\ref{barC1}), we have
\begin{align}\label{barbest}
\overline C_{q,m,p}=&2^{-\theta}(p-1)^{-\theta}(p-m-1)^{\left(1-\frac{1}{p}\right)\theta}[p(m+1)]^{\frac{\theta}{p}+\frac{(p-1)(1-\theta)}{pm}-\theta}
\eta_1^{\theta-\frac{(p-1)(1-\theta)}{pm}}\nonumber\\
& \cdot\left(\frac{\Gamma\left(\frac{pm}{p-m-1}+1+\frac{p-1}{p}\right)}
{\Gamma\left(\frac{pm}{p-m-1}+1\right)\Gamma\left(\frac{p-1}{p}\right)}\right)^{\theta}.
\end{align}
The expression (\ref{thetaanddelta}) of $\theta$ gives that
$$
\frac{\theta}{p}+\frac{(p-1)(1-\theta)}{pm}-\theta=\frac{1}{m+1},\quad \frac{(p-1)(1-\theta)}{pm}=\frac{p+(p-1)(m+1)}{p(m+1)^2}.
$$
Hence from (\ref{thebest1}) and (\ref{barbest}), we know that the result of Lemma \ref{d1equa} holds.
\end{proof}

Next, we prove (\ref{M_cpo}) for the case $m=\frac{p(q-1)+1}{p-1}$.
\begin{lem}\label{lm5.3}
For $u_{c,m}(r)$ in (\ref{uc1po}), we have
$$
M_c=d~\alpha(d)K^{q+1}L^{\frac{d(p-1)}{p}-\alpha(q+1)}\frac{p-1}{p}\mathcal{B}\left(\alpha(q+1)- \frac{d(p-1)}{p}, \frac{d(p-1)}{p}\right).
$$
\end{lem}
\begin{proof}
A simple computation for $u_{c,m}$ gives that
\begin{align}
M_c=&d\alpha(d)\int_0^{\infty} u_{c,m}(r)^{q+1}r^{d-1}\, dr\nonumber\\
=&d\alpha(d)K^{q+1}L^{-\alpha(q+1)}\int_0^{\infty} \left(1+\frac{r^{\frac{p}{p-1}}}{L}\right)^{-\alpha(q+1)}r^{d-1}\, dr.\label{M1po}
\end{align}
Let $y=\left(1+\frac{r^{\frac{p}{p-1}}}{L}\right)^{-1}$. Then we have $r=L^{\frac{p-1}{p}} \left(\frac{1-y}{y}\right)^{\frac{p-1}{p}}$ and
 $$
dr=-\frac{p-1}{p} L^{\frac{p-1}{p}}\left(\frac{1-y}{y}\right)^{\frac{p-1}{p}-1}y^{-2}\,dy.
$$
Hence we have
\begin{align}
M_c=&~d\alpha(d)\frac{p-1}{p}K^{q+1}L^{-\alpha(q+1)+\frac{d(p-1)}{p}}\int_0^{1} y^{\alpha(q+1)-\frac{d(p-1)}{p}-1}\left(1-y\right)^{\frac{d(p-1)}{p}-1}\, dx\nonumber\\
=& ~d\alpha(d)K^{q+1}R^{\frac{d(p-1)}{p}-\alpha(q+1)}\frac{p-1}{p}\mathcal{B}\left(\alpha(q+1)- \frac{d(p-1)}{p}, \frac{d(p-1)}{p}\right).\label{M2}
\end{align}
\end{proof}

We convert the best constant $\bar C_{q,m,p}$ given in \cite[Theorem 1.2]{del2003optimal} into our parameters and it is given by (\ref{barC1po})-(\ref{thetadeltapo}).
\begin{lem}\label{lm5.4}
For $m=\frac{p(q-1)+1}{p-1}$, the best constant $C_{q,m,p}$ in (\ref{formula11}) is exactly equal to the best constant $\bar C_{q,m,p}$ given in \cite[Theorem 1.2]{del2003optimal}.
\end{lem}
\begin{proof}
Due to $m=\frac{p(q-1)+1}{p-1}$, we have
 $$
  \theta=\frac{(q+1-p)d}{q[dp-(d-p)(q+1)]}.
  $$
 Hence it holds
\begin{eqnarray}\label{C1po}
C_{q,m,p}
=\left(\frac{(q+1)\eta_1}{d(q+1-p)}\right)^{\frac{\theta}{p}}\left(\frac{(q+1)\eta_1}{\eta}\right)^{-\frac{1}{m+1}}
q^{\frac{1}{m+1}}M_c^{-\frac{\theta}{d}},
\end{eqnarray}
where $\eta_1=pq-d(q+1-p)$

On the other hand, from (\ref{KRpo}), we deduce
\begin{eqnarray}\label{KR1po}
K^{q+1}L^{\frac{d(p-1)}{p}-\alpha(q+1)}=q^{\frac{d(p-1)}{pq\theta}}\left(\frac{p}{q+1-p}\right)^{\frac{(q+1)(p-1)}{q+1-p}}
\left(\frac{\eta_1}{q+1-p}\right)^{\frac{(q+1)}{q+1-p}-\frac{d}{q\theta}}.
\end{eqnarray}
Hence we have
\begin{eqnarray}\label{KR2po}
\left(K^{q+1}R^{d+\frac{mp^2}{(p-1)(p-m-1)}}\right)^{-\frac{\theta}{d}}=q^{-\frac{(p-1)}{pq}}\left(\frac{p}{q+1-p}\right)^{-\frac{\theta(q+1)(p-1)}{d(q+1-p)}}
\left(\frac{\eta_1}{q+1-p}\right)^{-\frac{\theta(q+1)}{d(q+1-p)}+\frac{1}{q}}.
\end{eqnarray}
Using (\ref{M_cpo}) and (\ref{KR2po}), we obtain
\begin{align}\label{formula10po}
M_c^{-\frac{\theta}{d}}=&~q^{-\frac{p-1}{pq}}\left(\frac{p}{q+1-p}\right)^{-\frac{\theta(q+1)(p-1)}{d(q+1-p)}}
\left(\frac{\eta_1}{q+1-p}\right)^{-\frac{\theta(q+1)}{d(q+1-p)}+\frac{1}{q}}\nonumber\\
&\cdot\left(\frac{p}{p-1}\right)^{\frac{\theta}{d}}\left(\frac{\Gamma\left(\frac{d}{2}+1\right)}{d\pi^{d/2}}\right)^{\frac{\theta}{d}}
  \left(\frac{\Gamma\left(\frac{(q+1)(p-1)}{q+1-p}\right)}{\Gamma\left(\frac{(q+1)(p-1)}{q+1-p}-\frac{d(p-1)}{p}\right)
  \Gamma(\frac{d(p-1)}{p})}\right)^{\frac{\theta}{d}}.
\end{align}
Noticing that the relations hold
\begin{eqnarray*}
\frac{1}{m+1}=\frac{p-1}{pq}, \quad \frac{\theta}{p}-\frac{1}{m+1}=\frac{\theta(q+1)}{d(q+1-p)}-\frac{1}{q}=\theta-\frac{\theta(q+1)(p-1)}{d(q+1-p)},
\end{eqnarray*}
and plugging (\ref{formula10po}) in (\ref{C1po}), we deduce
\begin{align}\label{C2po}
C_{q,m,p}=&\left(\frac{q+1}{d(q+1-p)}\right)^{\frac{\theta}{p}}\left(\frac{q+1}{\eta}\right)^{-\frac{1}{m+1}}p^{\frac{\theta}{p}-\frac{1}{m+1}-\theta}
(q+1-p)^{\theta}\nonumber\\
&\cdot \left(\frac{p}{p-1}\right)^{\frac{\theta}{d}}\left(\frac{\Gamma\left(\frac{d}{2}+1\right)}{d\pi^{d/2}}\right)^{\frac{\theta}{d}}
  \left(\frac{\Gamma\left(\frac{(q+1)(p-1)}{q+1-p}\right)}{\Gamma\left(\frac{\eta(p-1)}{p(q+1-p)}\right)
  \Gamma(\frac{d(p-1)}{p})}\right)^{\frac{\theta}{d}}.
\end{align}
Hence the best constant $C_{q,m,p}$ in (\ref{C2po}) is exactly equal to the best constant $\bar C_{q,m,p}$ given in (\ref{barC1po}).

\end{proof}

If we take $d=1$ in (\ref{barC1po}), we find that $\bar C_{q,m,p}$ exactly equals to $C_{q,m,p}$ given in (\ref{thebestpospe}). Detail verifications is provided in the following proposition. In the other words, we extend the results to the dimension $d=1$.

\begin{prop}\label{d1equa1}
For $d=1$ and $m=\frac{p(q-1)+1}{p-1}$, the best constant $\bar C_{q,m,p}$ defined in (\ref{barC1po}) exactly equals to $C_{q,m,p}$ in (\ref{thebestpospe}).
\end{prop}
\begin{proof}
Using the formula $\mathcal{B}(a,b)=\frac{\Gamma(a)\Gamma(b)}{\Gamma(a+b)}$ for (\ref{thebestpospe}) and noticing that
$$
\frac{(p-1)\eta_2}{p(q+1-p)}=\frac{(p-1)(q+1)}{q+1-p}-\frac{p-1}{p},
$$
we have
\begin{align}\label{thebest1po}
C_{q,m,p}=&~\frac{(q+1-p)^{\frac{2p-1}{p}\theta}\eta_2^{\frac{\eta_2(p-1)\theta}{p(q+1-p)}}}{[2(p-1)]^{\theta}[p(q+1)]^{\frac{(q+1)(p-1)\theta}{q+1-p}}}
\left(\mathcal{B}\left(\frac{(p-1)(q+1)}{q+1-p}-\frac{p-1}{p},\frac{2p-1}{p}\right)\right)^{-\theta},\nonumber\\
=&~\frac{(q+1-p)^{\frac{2p-1}{p}\theta}\eta_2^{\frac{\eta_2(p-1)\theta}{p(q+1-p)}}}{[2(p-1)]^{\theta}[p(q+1)]^{\frac{(q+1)(p-1)\theta}{q+1-p}}}
\left(\frac{\Gamma\left(\frac{(p-1)(q+1)}{q+1-p}-\frac{p-1}{p}\right)\frac{p-1}{p}\Gamma\left(\frac{p-1}{p}\right)}
{\frac{(p-1)(q+1)}{q+1-p}\Gamma\left(\frac{(p-1)(q+1)}{q+1-p}\right)}\right)^{-\theta}\nonumber\\
=&~\frac{(q+1-p)^{\frac{p-1}{p}\theta}\eta_2^{\frac{\eta_2(p-1)\theta}{p(q+1-p)}}}{[2(p-1)]^{\theta}[p(q+1)]^{\frac{(q+1)(p-1)\theta}{q+1-p}-\theta}}
\left(\frac{\Gamma\left(\frac{(p-1)(q+1)}{q+1-p}\right)}
{\Gamma\left(\frac{(p-1)(q+1)}{q+1-p}-\frac{p-1}{p}\right)\Gamma\left(\frac{p-1}{p}\right)}\right)^{\theta}.
\end{align}
Formally, taking $d=1$ in (\ref{barC1po}), we have
\begin{eqnarray}\label{barbestpo}
\bar C_{q,m,p}=\frac{(q+1-p)^{\frac{p-1}{p}\theta}\eta_2^{\frac{p-1}{pq}}}{[2(p-1)]^{\theta}[p(q+1)]^{\frac{p-1}{pq}-\frac{\theta}{p}}}
\left(\frac{\Gamma\left(\frac{(p-1)(q+1)}{q+1-p}\right)}
{\Gamma\left(\frac{(p-1)(q+1)}{q+1-p}-\frac{p-1}{p}\right)\Gamma\left(\frac{p-1}{p}\right)}\right)^{\theta}.
\end{eqnarray}
The expression (\ref{thetadeltapo}) with $d=1$ of $\theta$ shows that
$$
\frac{(q+1)(p-1)\theta}{q+1-p}-\theta=\frac{p-1}{pq}-\frac{\theta}{p},\quad \frac{p-1}{pq}=\frac{\eta_2(p-1)\theta}{p(q+1-p)}.
$$
Hence from (\ref{thebest1po}) and (\ref{barbestpo}), we know that the result of Lemma \ref{d1equa1} holds.
\end{proof}


\end{document}